\documentclass[a4paper,12pt]{amsart}
\pdfoutput=1
\usepackage[utf8]{inputenc}
\usepackage[T1]{fontenc}
\usepackage{lmodern}
\usepackage{amssymb,amsxtra}
\usepackage{nicefrac,mathtools}

\usepackage[lite]{amsrefs}

\renewcommand{\PrintDOI}[1]{\href{http://dx.doi.org/\detokenize{#1}}{doi: \detokenize{#1}}}

\BibSpec{book}{%
  +{}  {\PrintPrimary}                {transition}
  +{,} { \textit}                     {title}
  +{.} { }                            {part}
  +{:} { \textit}                     {subtitle}
  +{,} { \PrintEdition}               {edition}
  +{}  { \PrintEditorsB}              {editor}
  +{,} { \PrintTranslatorsC}          {translator}
  +{,} { \PrintContributions}         {contribution}
  +{,} { }                            {series}
  +{,} { \voltext}                    {volume}
  +{,} { }                            {publisher}
  +{,} { }                            {organization}
  +{,} { }                            {address}
  +{,} { \PrintDateB}                 {date}
  +{,} { }                            {status}
  +{}  { \parenthesize}               {language}
  +{}  { \PrintTranslation}           {translation}
  +{;} { \PrintReprint}               {reprint}
  +{.} { }                            {note}
  +{.} {}                             {transition}
  +{} { \PrintDOI}                   {doi}
  +{} { available at \url}            {eprint}
  +{}  {\SentenceSpace \PrintReviews} {review}
}

\usepackage{enumitem} 
\setlist[enumerate,1]{label=\textup{(\arabic*)}}

\usepackage{tikz}
\usetikzlibrary{matrix,calc,arrows,decorations.pathmorphing}
\tikzset{node distance=2cm, auto}

\tikzset{cd/.style=matrix of math nodes,row sep=2em,column sep=2em, text height=1.5ex, text depth=0.5ex}
\tikzset{cdar/.style=->,auto}
\tikzset{mid/.style={anchor=mid}} 
\tikzset{narrowfill/.style={inner sep=1pt, fill=white}}

\usepackage{microtype}

\usepackage[pdftitle={Aperiodicity, topological freeness and pure outerness:
  from group actions to Fell bundles},
pdfauthor={Bartosz Kosma Kwa\'sniewski, Ralf Meyer},
pdfsubject={Mathematics}
]{hyperref}

\newcommand*{\Cst}{\textup C^*}
\newcommand*{\red}{\textup r}
\newcommand*{\Cred}{\textup C^*_\textup r}
\newcommand*{\Cont}{\textup C}
\newcommand*{\nb}{\nobreakdash}
\newcommand*{\Star}{\(^*\)\nobreakdash-}
\newcommand*{\Mult}{\mathcal M}
\newcommand*{\UM}{\mathcal {UM}}
\newcommand{\idealin}{\mathrel{\triangleleft}} 

\newcommand{\Her}{\mathcal H}
\newcommand{\Hils}{\mathfrak H}
\newcommand*{\defeq}{\mathrel{\vcentcolon=}}
\newcommand*{\congto}{\xrightarrow\sim}
\newcommand*{\norm}[1]{\lVert#1\rVert}
\newcommand*{\abs}[1]{\lvert#1\rvert}
\newcommand*{\braket}[2]{\langle#1{\mid}#2\rangle}
\newcommand*{\bra}[1]{\langle#1{\mid}}
\newcommand*{\ket}[1]{{\mid}#1\rangle}

\newcommand*{\id}{\mathrm{id}}
\newcommand*{\Bor}{\mathrm{Bor}}
\newcommand*{\ima}{\textup i}
\newcommand*{\Ad}{\textup{Ad}}

\newcommand{\Bound}{\mathbb B}
\newcommand{\Comp}{\mathbb K}

\DeclareMathOperator{\dashind}{-Ind}

\DeclareMathOperator{\Spec}{Sp}
\DeclareMathOperator{\ord}{ord}
\DeclareMathOperator{\Homeo}{Homeo}
\DeclareMathOperator{\Kish}{Kish}
\DeclareMathOperator{\Aut}{Aut}
\DeclareMathOperator{\clsp}{\overline{span}}

\newcommand{\I}{\mathcal I}

\newcommand{\E}{\mathcal E}
\newcommand{\F}{\mathcal F}
\newcommand{\CC}{\mathcal C}

\newcommand{\Z}{\mathbb Z}
\newcommand{\N}{\mathbb N}
\newcommand{\T}{\mathbb T}

\numberwithin{equation}{section}
\newtheorem{theorem}{Theorem}[section]
\newtheorem{lemma}[theorem]{Lemma}
\newtheorem{proposition}[theorem]{Proposition}
\newtheorem{corollary}[theorem]{Corollary}
\theoremstyle{definition}
\newtheorem{definition}[theorem]{Definition}
\newtheorem{example}[theorem]{Example}
\newtheorem{remark}[theorem]{Remark}

\title[Aperiodicity, topological freeness and pure outerness]%
{Aperiodicity, topological freeness and pure outerness:
  from group actions to Fell bundles}

\author{Bartosz Kosma Kwa\'sniewski}
\email{bartoszk@math.uwb.edu.pl}
 \address{Institute of Mathematics,  University  of Bia\l ystok\\
ul. K. Cio\l kowskiego 1M, 15-245 Bia\l ystok,   Poland // Department of Mathematics and Computer Science, The University of Southern Denmark,
Campusvej 55, DK-5230 Odense M, Denmark }

\author{Ralf Meyer}
\email{rmeyer2@uni-goettingen.de}
\address{Mathematisches Institut\\
 Georg-August-Universit\"at G\"ottingen\\
 Bunsenstra\ss e 3--5\\
 37073 G\"ottingen\\
 Germany}

\keywords{group action; Fell bundle;
  aperiodic; topologically free; purely outer; Connes spectrum;
  purely infinite; strongly purely infinite; filling family}
\subjclass[2010]{Primary 46L55; Secondary 46L40}

\begin{document}

\begin{abstract}
  We generalise various non-triviality conditions for group actions to
  Fell bundles over discrete groups and prove several implications
  between them.  We also study sufficient criteria for the reduced
  section \(\Cst\)\nb-algebra~\(\Cred(\mathcal{B})\)
  of a Fell bundle \(\mathcal{B}=(B_g)_{g\in G}\)
  to be strongly purely infinite.  If the unit fibre \(A\defeq B_e\)
  contains an essential ideal that is separable or of Type~I,
  then~\(\mathcal{B}\)
  is aperiodic if and only if~\(\mathcal{B}\)
  is topologically free.  If, in addition, \(G=\Z\)
  or \(G=\Z/p\)
  for a square-free number~\(p\),
  then these equivalent conditions are satisfied if and only if~\(A\)
  detects ideals in~\(\Cred(\mathcal{B})\),
  if and only if~\(A^+\setminus\{0\}\)
  supports~\(\Cred(\mathcal{B})^+\setminus\{0\}\)
  in the Cuntz sense.  For~\(G\)
  as above and arbitrary~\(A\),
  \(\Cred(\mathcal{B})\)
  is simple if and only if~\(\mathcal{B}\)
  is minimal and pointwise outer.  In general, \(\mathcal{B}\)
  is aperiodic if and only if each of its non-trivial fibres has a
  non-trivial Connes spectrum.  If~\(G\)
  is finite or if~\(A\)
  contains an essential ideal that is of Type~I or simple, then
  aperiodicity is equivalent to pointwise pure outerness.
\end{abstract}

\maketitle
\setcounter{tocdepth}{1}

\section{Introduction}
\label{sec:introduction}

Several deep results on the relationship between various
non-triviality conditions for group actions and the simplicity of
reduced crossed products were proved around~1980 by Olesen and
Pedersen \cites{Olesen-Pedersen:Applications_Connes,
  Olesen-Pedersen:Applications_Connes_2,
  Olesen-Pedersen:Applications_Connes_3}, Kishimoto
\cites{Kishimoto:Simple_crossed, Kishimoto:Outer_crossed,
  Kishimoto:Freely_acting}, and
Rieffel~\cite{Rieffel:Actions_finite}.  Their powerful results are
used, for instance, in
\cites{Jolissaint-Robertson:Simple_purely_infinite,
  Sierakowski:IdealStructureCrossedProducts,
  Rordam-Sierakowski:Purely_infinite,
  Pasnicu-Phillips:Spectrally_free,
  Kirchberg-Sierakowski:Strong_pure} to study of the ideal structure
and pure infiniteness of crossed products for actions of discrete
groups.

The main point of this article is to generalise this theory to Fell
bundles over discrete groups.  This contains (twisted) crossed
products for partial group actions as a special case, see
\cites{Exel:TwistedPartialActions, Exel:Partial_dynamical} or
Example~\ref{ex:twisted_fell_bundles} below. In
fact, 
Fell bundles over~\(G\)
model all \(\Cst\)\nb-algebras
graded by~\(G\),
and many important \(\Cst\)\nb-algebras
come with such gradings.  For instance, the relative Cuntz--Pimsner
algebras of a \(\Cst\)\nb-correspondence
\cites{Muhly-Solel:Tensor, Pimsner:Generalizing_Cuntz-Krieger} and
Doplicher--Roberts algebras
\cite{Doplicher-Pinzari-Zuccante:Hilbert_bimodule} are naturally
\(\Z\)\nb-graded,
and these gradings are well understood, see
\cites{Kwasniewski:Cuntz-Pimsner-Doplicher,
  Schweizer:Dilations_correspondences,
  Abadie-Eilers-Exel:Morita_bimodules}.  The Cuntz--Nica--Toeplitz
algebra \cite{Sims-Yeend:Cstar_product_systems} of a product system
over a quasi-lattice order is graded by the ambient group, and this
grading is exploited, for instance, in
\cite{Carlsen-Larsen-Sims-Vittadello:Co-universal}.  The
Cuntz--Pimsner algebras \cite{Fowler:Product_systems} of product
systems over Ore semigroups are naturally graded by the group
completion of the Ore semigroup, with well understood fibres, see
\cites{Kwasniewski-Szymanski:Ore, Albandik-Meyer:Product}.  Thus an
extension of the classical theory to Fell bundles allows to study the
ideal structure and pure infiniteness properties for large classes of
\(\Cst\)\nb-algebras.

Let \(\alpha\colon G\to \Aut(A)\)
be an action of a discrete group~\(G\)
on a \(\Cst\)-algebra~\(A\).
Popular among the non-triviality conditions for such actions are
\emph{proper outerness} (see~\cite{Elliott:Simple_crossed}) and
\emph{topological freeness}
(see~\cite{Archbold-Spielberg:Topologically_free}).  However, proper
outerness is only useful if~\(A\)
is separable.  Then it is equivalent to topological freeness and to
\emph{aperiodicity} by the work of Olesen and Pedersen.  Aperiodicity
requires a seemingly technical condition for~\(\alpha_g\)
for all \(g\in G\setminus\{e\}\)
(see Definition~\ref{def:Kishimoto} below), which we call
\emph{Kishimoto's condition} because its role was first highlighted by
Kishimoto~\cite{Kishimoto:Outer_crossed}.  The term ``aperiodicity''
was coined
in~\cite{Muhly-Solel:Morita_equivalence_of_tensor_algebras}, where
single \(\Cst\)\nb-correspondences
were considered, and carried over to Fell bundles over discrete groups
in~\cite{Kwasniewski-Szymanski:Pure_infinite}.  By a result of
Kishimoto~\cite{Kishimoto:Freely_acting}, the aperiodic group actions
are the same as the \emph{pointwise spectrally non-trivial} actions of
Pasnicu and Philips~\cite{Pasnicu-Phillips:Spectrally_free}.
All conditions above except
proper outerness imply that~\(A\)
\emph{detects} ideals in \(B\defeq A\rtimes_{\alpha,\red} G\),
that is, that \(I\cap A\neq0\)
whenever~\(I\)
is a non-zero ideal in~\(B\).
Their residual versions together with exactness imply that~\(A\)
\emph{separates ideals} in~\(B\),
that is, \(I\cap A\neq J\cap A\)
whenever \(I,J\idealin B\)
are ideals with \(I\neq J\),
see~\cite{Sierakowski:IdealStructureCrossedProducts}.
Furthermore, aperiodicity implies that elements of \(A^+\setminus \{0\}\)
\emph{support} the non-zero positive elements in the reduced crossed
product \(A\rtimes_{\alpha,\red} G\)
in the Cuntz sense; this is exactly the property one uses to detect
pure infiniteness, see \cite{Kwasniewski:Crossed_products}*{Subsection
  2.8} and~\cite{Kwasniewski-Szymanski:Pure_infinite}.  By arguments
in~\cite{Kirchberg-Sierakowski:Strong_pure}, residual
aperiodicity implies that~\(A^+\)
is a \emph{filling family} for \(A\rtimes_{\alpha,\red} G\).
Filling families are introduced
in~\cite{Kirchberg-Sierakowski:Filling_families} to detect strong pure
infiniteness.

Aperiodic and topologically free Fell bundles over discrete groups
have been defined already in
\cites{Kwasniewski-Szymanski:Pure_infinite, Abadie-Abadie:Ideals,
  Kwasniewski-Szymanski:Ore}.  Several other non-triviality conditions
for group actions generalise readily to Fell bundles (see
Definition~\ref{def:aperiodic} below).  This includes pointwise pure
outerness or pointwise pure universal weak outerness, but apparently
not proper outerness.  To establish relationships among these
pointwise conditions, it suffices to study a single Hilbert bimodule.

We highlight our main achievements.  Let
\(\mathcal{B}=(B_g)_{g\in G}\)
be a Fell bundle with unit fibre \(A\defeq B_e\).
If~\(A\)
contains an essential
ideal that is separable or of Type~I,
then~\(\mathcal{B}\)
is aperiodic if and only if~\(\mathcal{B}\)
is topologically free, if and only if~\(\mathcal{B}\) is pointwise
purely universally weakly outer
(Theorems \ref{thm:separable_type_I_equivalences}
and~\ref{the:Hilbi_Kish_vs_weakly_inner}).  If, in addition,
\(G=\Z\)
or \(G=\Z/p\)
for a square-free number~\(p\),
then this is further equivalent to the condition that~\(A\)
detects ideals in \(\Cred(\mathcal{B})= \Cst(\mathcal{B})\)
(Theorem~\ref{the:Fell_uniqueness}).  Under the latter assumptions,
\(\mathcal{B}\)
is residually aperiodic if and only if~\(A^+\)
is a filling family for \(\Cst(\mathcal{B})\),
if and only if~\(A\)
separates ideals in \(\Cst(\mathcal{B})\)
(Theorem~\ref{the:automatic_gauge-invariance}).
We generalise the Connes spectrum and use it to characterise
aperiodicity of Fell bundles
(Theorem~\ref{the:Equivalence_for_finite2}), and the property
that~\(A\)
detects ideals in~\(\Cst(\mathcal{B})\)
when~\(G\)
is Abelian
(Proposition~\ref{pro:Fell_Connes_spectrum_vs_detect_ideals}).
If~\(G\)
is finite or~\(A\) contains an essential ideal
that is simple or of Type~I, then~\(\mathcal{B}\)
is aperiodic if and only if~\(\mathcal{B}\)
is pointwise purely outer (Theorem~\ref{the:Equivalence_for_finite2}).
If \(G=\Z\)
or \(G=\Z/p\)
for a square-free number~\(p\),
then~\(\Cst(\mathcal{B})\)
is simple if and only if~\(\mathcal{B}\)
is pointwise purely outer and minimal (Theorem~\ref{the:Fell_simple}).

For group actions by automorphisms, most of our results are already
contained in the classic articles by Olesen--Pedersen, Kishimoto and
Rieffel mentioned above.  The results about \(G=\Z/p\)
with square-free~\(p\)
are not stated explicitly there, but the proofs for~\(\Z\)
evidently cover this case as well.  There are many other important
results in these articles.  The culmination of the work of Olesen and
Pedersen in \cite{Olesen-Pedersen:Applications_Connes_3}*{Theorem 6.6}
shows that eleven non-triviality properties are equivalent for an
automorphism of a separable \(\Cst\)\nb-algebra.
The work of Kishimoto adds a few more equivalent properties.  Some of
the properties of automorphisms that are used in the proofs do not
generalise to Hilbert bimodules.  Therefore, it seems difficult to
prove our main results directly.  We proceed differently.
We reduce our main results to the special case of group actions by automorphisms.
We use that every Fell bundle is Morita equivalent to
a Fell bundle that comes from a partial action that globalises to an
action by automorphisms.  Such a \emph{Morita globalisation} for Fell
bundles over discrete groups is constructed by
Abadie~\cite{Abadie:Enveloping} and by
Quigg~\cite{Quigg:Discrete_coactions_and_bundles} using Takai
Duality (see
Section~\ref{sec:Morita_globalisation}).  In order to show that this
construction preserves the non-triviality properties we are interested
in, we introduce the concepts of a \emph{Morita covering} and a
\emph{dense Morita covering} (see Section~\ref{sec:Kish_permanence}).
We show that these preserve Kishimoto's condition and a number of
other properties of Hilbert bimodules and Fell bundles (see
Proposition~\ref{pro:dense_Morita_covering}) and that Abadie's Morita
globalisation is also a Morita covering.  This allows us to reduce our
main theorems to the special case treated in the classical theory.
Instead of separability or simplicity
of~\(A\),
it suffices to assume that~\(A\)
contains an essential ideal with that property or an essential ideal
of Type~I.

It is not clear whether the conditions that are relevant for pure
infiniteness and strong pure infiniteness are invariant under Morita
globalisations.  So we study the relationships between them
directly.  In particular, in Section~\ref{sec:strong_pi} we
formulate strong pure infiniteness criteria for reduced Fell bundle
section \(\Cst\)\nb-algebras that generalise those in
\cites{Kirchberg-Sierakowski:Strong_pure,
  Kwasniewski-Szymanski:Pure_infinite}.

The paper is organised as follows.  Section~\ref{sec:automorphisms}
collects the classical results about automorphisms and group actions
by automorphisms in a way suitable for our later generalisations, with
a few small additions.  Section~\ref{sec:Hilmi_crossed} recalls
preliminaries on Hilbert bimodules, Fell bundles and their crossed
products.  Section~\ref{sec:nontriviality} introduces some of the
non-triviality conditions for Hilbert bimodules and Fell bundles, as
well as notions of residually supporting and filling subalgebras.
Section~\ref{sec:strong_pi} contains general criteria for strong pure
infiniteness of Fell bundle \(\Cst\)\nb-algebras.
Section~\ref{sec:Kish_permanence} studies permanence of non-triviality
conditions for single Hilbert bimodules under Morita restrictions and
dense Morita coverings.  Section~\ref{sec:Morita_globalisation}
studies a canonical Morita globalisation for a Fell bundle that goes
back to Abadie~\cite{Abadie:Enveloping} and
Quigg~\cite{Quigg:Discrete_coactions_and_bundles}.  In
Section~\ref{sec:tok}, we
use Morita globalisations and Morita coverings to generalise and
extend relationships between non-triviality conditions from single
automorphisms to single Hilbert bimodules.  In
Section~\ref{sec:Connes_spectrum}, we define the Connes and strong
Connes spectra for Fell bundles over Abelian discrete groups and for
single Hilbert bimodules.  We relate these notions to aperiodicity or
pure outerness of Fell bundles, and to detection of ideals
in~\(\Cred(\mathcal{B})\).
We show that a number of conditions are equivalent if \(G=\Z\)
or \(G=\Z/p\)
for a square-free number~\(p\).
We also prove residual versions of our results, which are related to
separation of ideals in~\(\Cred(\mathcal{B})\).

\subsection{Acknowledgements}

The research leading to these results has been funded by the
European Union's 7th Framework Programme (FP7/2007--2013) under
grant agreement number 621724.  The first-named author was partially
supported by the NCN (National Centre of Science) grant
2014/14/E/ST1/00525.  This work was finished while he participated in
the Simons Semester at IMPAN -- Fundation grant 346300 and the Polish
Government MNiSW 2015--2019 matching fund.  He also would like to
express his gratitude to Suliman Albandik for inviting him to the
Mathematisches Institut in G\"ottingen, where this project started.

\section{Group actions by automorphisms}
\label{sec:automorphisms}

Throughout this article, \(G\)
is a discrete group and~\(e\)
denotes its unit element.  Let \(\alpha\colon G\to \Aut(A)\)
be an action of~\(G\)
on a \(\Cst\)\nb-algebra~\(A\).
Let \(\Her(A)\)
be the set of \emph{non-zero}, hereditary subalgebras of~\(A\),
and let
\[
\Her^\alpha(A) \defeq
\{D\in \Her(A)\mid \alpha_g(D)=D \text{ for all }g\in G\}.
\]
Let \(\I(A)\) be the ideal lattice of~\(A\), and \(\I^\alpha(A)\)
the sublattice of \(\alpha\)\nb-invariant ideals.

For the time being, we assume that~\(G\)
is Abelian, so that the Pontryagin dual~\(\widehat{G}\)
and various spectra for group actions are defined.  The \emph{Arveson
  spectrum} of~\(\alpha\)
is the set \(\Spec(\alpha)\)
of all \(\chi\in\widehat{G}\)
for which there is a net~\((a_\lambda)\)
in~\(A\)
with \(\norm{a_\lambda}=1\)
for all~\(\lambda\)
and \(\lim {}\norm{\alpha_g(a_\lambda) - \chi(g)a_{\lambda}}=0\)
for all \(g\in G\),
compare \cite{Pedersen:Cstar_automorphisms}*{Proposition 8.1.9(iii)}.
The Connes and Borchers spectra were introduced by
Olesen~\cite{Olesen:Inner}, see also
\cite{Pedersen:Cstar_automorphisms}*{Chapter~8}.  The \emph{Connes
  spectrum} of~\(\alpha\) is
\[
\Gamma(\alpha)\defeq  \bigcap_{D\in \Her^\alpha(A)} \Spec(\alpha|_D).
\]
It is a closed subgroup of~\(\widehat{G}\).  The \emph{Borchers
  spectrum} \(\Gamma_\Bor(\alpha)\) of~\(\alpha\) is a
similar intersection taken over those \(D\in \Her^\alpha(A)\)
for which the ideal~\(\overline{A D A}\) is essential in~\(A\).  In
general, it
is the closure of a union of subgroups of~\(\widehat{G}\), see
\cite{Pedersen:Cstar_automorphisms}*{Propositions 8.8.4 and~8.8.5}.
The strong Connes spectrum~\(\tilde\Gamma(\alpha)\) is introduced by
Kishimoto~\cite{Kishimoto:Simple_crossed} and used by him in
\cites{Kishimoto:Outer_crossed, Kishimoto:Freely_acting}.  Since the
group~\(G\) is discrete, \(\tilde\Gamma(\alpha)\) is just a residual
version of~\(\Gamma(\alpha)\).  Namely,
\cite{Kishimoto:Freely_acting}*{Proposition~4.1} identifies
\begin{equation}
  \label{eq:residual_Connes_is_strong_Connes}
  \tilde\Gamma(\alpha) = \bigcap_{I\in\I^\alpha(A), I\neq A} \Gamma(\alpha|_{A/I}),
\end{equation}
where \(\alpha|_{A/I}\) denotes the induced action
on~\(A/I\).  By~\eqref{eq:residual_Connes_is_strong_Connes},
\(\tilde\Gamma(\alpha)\) is a closed subgroup of~\(\widehat{G}\).

We are going to describe the Connes spectrum and the strong Connes
spectrum in ways that generalise immediately to Fell bundles.  Then we
use these spectra to characterise when~\(A\)
detects or separates ideals in~\(A\rtimes_\alpha G\).


\begin{proposition}
  \label{pro:Connes_spectrum}%
  Let \(\alpha\colon G\to \Aut(A)\) be an action of a discrete
  Abelian group, and let \(\beta\colon \widehat{G}\to
  \Aut(A\rtimes_\alpha G)\) be the dual action.  Let \(\chi\in
  \widehat{G}\).
  \begin{enumerate}
  \item \label{en:Strong_Connes_spectrum_characterisation1}%
    \(\chi\in \Gamma(\alpha)\) if and only \(I \cap
    \beta_\chi(I)\neq 0\) for each non-zero ideal~\(I\)
    in~\(A\rtimes_\alpha G\).
  \item \label{en:Strong_Connes_spectrum_characterisation}%
    \(\chi\in \tilde \Gamma(\alpha)\) if and only \(\beta_\chi(I)=
    I\) for any ideal~\(I\) in~\(A\rtimes_\alpha G\).
  \end{enumerate}
\end{proposition}

\begin{proof}
  Part~\ref{en:Strong_Connes_spectrum_characterisation1} is
  \cite{Olesen-Pedersen:Applications_Connes}*{Proposition~5.4} and
  part \ref{en:Strong_Connes_spectrum_characterisation} is almost
  \cite{Kishimoto:Simple_crossed}*{Lemma~3.4}; we may replace
  \(\beta_\chi(I)\subseteq I\) by \(\beta_\chi(I)= I\) because
  \(\tilde\Gamma(\alpha)\) is a subgroup.
\end{proof}

\begin{definition}
  Let~\(A\) be a \(\Cst\)\nb-subalgebra of a \(\Cst\)-algebra \(B\).
  We say that~\(A\) \emph{detects ideals} in~\(B\) if \(J\cap A= 0\)
  implies \(J=0\) for each ideal~\(J\) in~\(B\).  It \emph{separates
    ideals} if \(I \cap A = J\cap A\) for two ideals \(I,J\idealin
  B\) only happens if already \(I=J\).
\end{definition}

\begin{remark}
  \label{rem:separates_is_residual_detects}
  A \(\Cst\)\nb-subalgebra \(A\subseteq B\) separates ideals
  in~\(B\) if and only if~\(A\) residually detects ideals
  in~\(B\), that is, for each ideal~\(J\) in~\(B\) the image
  of~\(A\) detects ideals in the quotient~\(B/J\).
\end{remark}

\begin{theorem}[\cite{Olesen-Pedersen:Applications_Connes_2}*{Theorem~2.5}]
  \label{the:Connes_spectrum_vs_detect_ideals}
  Let \(\alpha\colon G\to \Aut(A)\) be an action of a discrete
  Abelian group~\(G\).  Then \(A\)~detects ideals in
  \(A\rtimes_\alpha G\) if and only if\/ \(\Gamma(\alpha) =
  \widehat{G}\).
\end{theorem}


An action is called
\emph{minimal} if \(0\)
and~\(A\)
are the only invariant
ideals in~\(A\).
This is necessary for the crossed product to be simple.
We denote by \(\Mult(A)\) the multiplier algebra of~\(A\), and by \(\UM(A)\)
the group of its unitaries.

\begin{theorem}
  \label{the:automorphism_simple}
  Let \(\alpha\colon G\to \Aut(A)\) be a minimal action of a
  discrete Abelian group~\(G\).  Consider the following conditions:
  \begin{enumerate}[wide,label=\textup{(\ref*{the:automorphism_simple}.\arabic*)}]
  \item \label{en:automorphism_simple_simple}%
    \(A\rtimes_\alpha G\) is simple;
  \item \label{en:automorphism_simple_Connes}%
    \(\Gamma(\alpha)=\widehat{G}\);
  \item \label{en:automorphism_simple_centre}%
    \(\Mult(A\rtimes_\alpha G)\) has trivial centre;
  \item \label{en:automorphism_simple_invariant_outer}%
    there are no \(g\in G\setminus\{e\}\) and \(u\in\UM(A)\) with
    \(\alpha_g = \Ad_u\) and \(\alpha_h(u)=u\) for all \(h\in G\);
  \item \label{en:automorphism_simple_outer}
    there are no \(g\in G\setminus\{e\}\) and \(u\in\UM(A)\) with
    \(\alpha_g = \Ad_u\).
  \end{enumerate}
  Then \ref{en:automorphism_simple_simple}\(\Leftrightarrow
  \)\ref{en:automorphism_simple_Connes}\(\Rightarrow
  \)\ref{en:automorphism_simple_centre}\(\Rightarrow
  \)\ref{en:automorphism_simple_invariant_outer}\(\Leftarrow
  \)\ref{en:automorphism_simple_outer}.  If the group~\(G\) is
  cyclic or finite, then
  \ref{en:automorphism_simple_simple}--\ref{en:automorphism_simple_invariant_outer}
  are equivalent.  If~\(G\) is \(\Z\) or~\(\Z/p\) for a square-free
  number~\(p\), then
  \ref{en:automorphism_simple_simple}--\ref{en:automorphism_simple_outer}
  are equivalent.
\end{theorem}

\begin{proof}
  The equivalence of \ref{en:automorphism_simple_simple}
  and~\ref{en:automorphism_simple_Connes} is
  \cite{Olesen-Pedersen:Applications_Connes_2}*{Theorem 3.1}.
  All central multipliers of a simple \(\Cst\)\nb-algebra are scalar
  multiples of~\(1\).  Thus
  \ref{en:automorphism_simple_simple}\(\Rightarrow
  \)\ref{en:automorphism_simple_centre}.  The implication
  \ref{en:automorphism_simple_outer}\(\Rightarrow
  \)\ref{en:automorphism_simple_invariant_outer} is trivial.

  We prove by contradiction that~\ref{en:automorphism_simple_centre}
  implies~\ref{en:automorphism_simple_invariant_outer}.  So assume
  that \(\alpha_g = \Ad_u\)
  for some \(g\in G\setminus\{e\}\)
  and some \(G\)\nb-invariant
  \(u\in\UM(A)\).
  Let \(G\ni h\mapsto \lambda_h\)
  denote the canonical homomorphism from~\(G\)
  to the group of unitary multipliers of~\(A\rtimes_\alpha G\).
  The unitary \(u^*\lambda_g\in\UM(A\rtimes_\alpha G)\)
  commutes with \(A\subseteq \Mult(A\rtimes_\alpha G)\)
  because \(u^*\lambda_g a (u^*\lambda_g)^* = u^* \alpha_g(a) u = a\)
  for all \(a\in A\),
  and it commutes with all~\(\lambda_h\)
  because
  \(\lambda_h (u^* \lambda_g) \lambda_h^* = \alpha_h(u)^* \lambda_{h g
    h^{-1}} = u^* \lambda_g\);
  here we use that~\(u\)
  is \(\lambda_h\)\nb-invariant
  and~\(G\)
  Abelian.  Thus~\(u^*\lambda_g\)
  is a central unitary multiplier of~\(A\rtimes_\alpha G\).
  It is not a scalar multiple of~\(1\)
  because \(g\neq e\).
  Thus~\(\Mult(A\rtimes_\alpha G)\) has non-trivial centre.

  Suppose \(G=\Z\)
  or that~\(G\)
  is finite.  We prove by contradiction that
  \ref{en:automorphism_simple_invariant_outer} implies
  \ref{en:automorphism_simple_Connes}.  So assume
  \(\Gamma(\alpha)\neq\widehat{G}\).
  Then there is \(g\in \Gamma(\alpha)^\bot\setminus\{ e\}\).
  All proper closed subgroups of \(\widehat{\Z}=\T\)
  or~\(\widehat{F}\)
  for a finite Abelian group~\(F\)
  are discrete.  So~\(\Gamma(\alpha)\)
  is discrete.  The quotient~\(\widehat{G}/\Gamma(\alpha)\)
  is compact.  So
  \cite{Olesen-Pedersen:Applications_Connes_3}*{Theorem 4.5} applies;
  it gives \(\alpha_g = \Ad_u\)
  for a \(G\)\nb-invariant
  \(u\in\UM(A)\),
  contradicting~\ref{en:automorphism_simple_invariant_outer}.

  Finally, we prove
  \ref{en:automorphism_simple_invariant_outer}\(\Rightarrow
  \)\ref{en:automorphism_simple_outer}
  if \(G\)
  is \(\Z\)
  or~\(\Z/p\)
  with square-free~\(p\).
  Suppose \(\alpha_g = \Ad_u\)
  for some \(g\in G\setminus\{ e\}\)
  and \(u\in\UM(A)\).
  Then \(\alpha_g(u) = \Ad_u(u)=u\).
  We argue as in the proof of
  \cite{Olesen-Pedersen:Applications_Connes_2}*{Theorem~4.6} to show
  that \(\alpha_{g\cdot g}=\Ad_w\)
  for an \(\alpha\)\nb-invariant unitary multiplier
  \(w\in\UM(A)\).
  If \(G=\Z/p\),
  we lift~\(g\)
  to \(\hat{g}\in\Z\).
  Since~\(p\)
  does not divide~\(\hat{g}\)
  and~\(p\)
  is square-free, it does not divide~\(\hat{g}^2\)
  either, so \(g^2\neq e\) in~\(G\).  Define
  \[
  w\defeq u\alpha_1(u)\alpha_2(u)\dotsm \alpha_{\hat{g}-1}(u).
  \]
  Then
  \(\Ad_w = \Ad_u \circ \Ad_{\alpha_1(u)} \circ \dotsb \circ
  \Ad_{\alpha_{\hat{g}-1}}(u) = \Ad_u^{\hat{g}} = \alpha_{\hat{g}\cdot \hat{g}}\)
  because \(\Ad_u \circ \Ad_v = \Ad_{u v}\)
  and \(\Ad_{\alpha(v)} = \alpha\circ \Ad_v \circ \alpha^{-1}\)
  for any automorphism~\(\alpha\)
  and any unitaries \(u,v\).
  And~\(w\)
  is invariant under~\(\alpha_1\)
  and hence under~\(\alpha_h\) for all \(h\in G\) because
  \[
  \alpha_1(w)
  = \alpha_1(u)\alpha_2(u)\alpha_3(u)\dotsm \alpha_{\hat{g}}(u)
  = u^* w \alpha_{\hat{g}}(u)
  = u^* w u
  = \alpha_g^{-1}(w)
  = w;
  \]
  the last step uses \(\alpha_g(u)=u\).
\end{proof}


\begin{remark}
  \label{rem:Phillips_examples}
  Let \(G=\Z/4\)
  and let~\(A\)
  be the CAR algebra, that is, the UHF algebra of type~\(2^\infty\).
  Phillips has constructed an action \(\alpha\colon G\to\Aut(A)\)
  such that~\(\alpha_2\)
  is inner and \(A\rtimes_\alpha \Z/4\)
  is simple, see
  \cite{Phillips:Equivariant_K-theory_freeness}*{Example 9.3.9 and
    Remark 9.3.10}.  Thus \ref{en:automorphism_simple_outer} is
  strictly weaker than
  \ref{en:automorphism_simple_simple}%
  --\ref{en:automorphism_simple_invariant_outer}
  for the finite cyclic group~\(\Z/4\).
\end{remark}

Next we consider some properties of a single automorphism
\(\alpha\in\Aut(A)\).  We will later apply these to group actions by
requiring them for all \(g\in G\setminus\{ e\}\).  Let
\(\Gamma(\alpha)\), \(\Gamma_\Bor(\alpha)\) and
\(\tilde\Gamma(\alpha)\) be the corresponding spectra for the
\(\Z\)\nb-action generated by~\(\alpha\).

\begin{lemma}
  \label{lem:explanation_of_spectra_for_automorphisms}
  Let \(\alpha\colon G\to \Aut(A)\) be an action of a cyclic
  group~\(G\) with generator \(g\in G\).  Then
  \[
  \]
  \begin{gather}
    \label{eq:Gamma_single}
    \Gamma(\alpha)=\Gamma(\alpha_g),
    \qquad
    \Gamma_\Bor(\alpha)=\Gamma_\Bor(\alpha_g),
    \qquad
    \tilde\Gamma(\alpha)=\tilde\Gamma(\alpha_g),\\
    \label{eq:Bor_through_Connes}
    \Gamma_\Bor(\alpha)
    = \overline{\bigcup_{D\in \Her^\alpha(A)} \Gamma(\alpha|_D)}.
  \end{gather}
\end{lemma}

\begin{proof}
  The Arveson spectrum \(\Spec(\alpha)\)
  of~\(\alpha\)
  coincides with the spectrum \(\Spec(\alpha_g)\)
  of~\(\alpha_g\)
  treated as an operator on~\(A\).
  A subset of~\(A\)
  is invariant for the action~\(\alpha\)
  if and only if it is invariant under the automorphism~\(\alpha_g\).
  Thus the spectra \(\Gamma(\alpha)\)
  and \(\Gamma_\Bor(\alpha)\)
  depend only on~\(\alpha_g\)
  and, in view of \eqref{eq:residual_Connes_is_strong_Connes}, the
  same applies to~\(\tilde\Gamma(\alpha)\).
  This implies~\eqref{eq:Gamma_single}.
  Equation~\eqref{eq:Bor_through_Connes} follows from
  \cite{Olesen-Pedersen:Applications_Connes_3}*{Theorem~3.9}.
\end{proof}

\begin{definition}[see \cite{Kishimoto:Outer_crossed}*{Lemma~1.1}]
  \label{def:Kishimoto}
  An automorphism \(\alpha\in\Aut(A)\)
  \emph{satifies Kishimoto's condition} if, for all \(D\in \Her(A)\),
  \(b\in A\),
  \begin{equation}
    \label{eq:Kishimoto_condition}
    \inf {}\{\norm{a b \alpha(a)} \mid a\in D^+,\ \norm{a}=1\}=0.
  \end{equation}
\end{definition}

A similar condition was first used by Connes in the von Neumann algebraic
setting and later by Elliott~\cite{Elliott:Simple_crossed} to prove
that reduced crossed products for outer group actions
on AF-algebras are simple.  Kishimoto
identified~\eqref{eq:Kishimoto_condition} as a key step to
generalise Elliott's result to group actions on arbitrary simple
\(\Cst\)\nb-algebras.

The following elementary lemma will be used several times:

\begin{lemma}
  \label{lem:technical}
  Let \(b\in A^+\)
  with \(\norm{b}=1\)
  and let \(\varepsilon>0\).
  There is \(d\in \overline{b A b}^+\)
  with \(\norm{d}=1\),
  such that
  \begin{equation}
    \label{eq:def_D0}
    D_0\defeq \{x\in A \mid d x= x = x d\}
    \in \Her(\overline{b A b})
    \subseteq \Her(A)
  \end{equation}
  and \(\norm{b x - x}<\varepsilon \norm{x}\)
  and \(\norm{b x}\ge (1-\varepsilon) \norm{x}\) for all \(x\in D_0\).
\end{lemma}

\begin{proof}
  We may assume \(\varepsilon\in(0,1)\).
  Given \(\delta\in(0,1)\),
  we define
  \begin{equation}
    \label{eq:f_varepsilon}
    f_\delta\colon [0,1]\to[0,1],\qquad
    t \mapsto
    \begin{cases}
      0&\text{if }0\le t < 1-\delta,\\
      \frac{2}{\delta}(t- 1+\delta)&
      \text{if } 1-\delta\le t < 1-\nicefrac{\delta}{2},\\
      1&\text{if }1-\nicefrac{\delta}{2}\le t \le1.
    \end{cases}
  \end{equation}
  This function is continuous, and \(\norm{f_\delta(b)}=1\)
  because \(\norm{b}=1\).
  Put \(d\defeq f_\varepsilon(b)\in \overline{b A b}^+\).
  Then~\(D_0\)
  in~\eqref{eq:def_D0} is a hereditary subalgebra
  of~\(\overline{b A b}\).
  It is non-zero because it
  contains~\(f_{\nicefrac{\varepsilon}{2}}(b)\).
  Let \(x\in D_0\).
  Then
  \(\norm{b x - x} = \norm{b d x - d x} \le \norm{b d - d} \norm{x}
  \le \varepsilon \norm{x}\)
  because \(\abs{(t-1)f_\varepsilon(t)} \le \varepsilon\)
  for all \(t\in[0,1]\).
  Hence
  \(\norm{x} \le \norm{b x - x} + \norm{b x} \le \varepsilon \norm{x}
  + \norm{b x}\), that is, \(\norm{b x}\ge (1-\varepsilon) \norm{x}\).
\end{proof}

\begin{theorem}
  \label{the:Kishimoto_spectrum}
  For any \(\alpha\in \Aut(A)\) the following conditions are
  equivalent:
  \begin{enumerate}[wide,label=\textup{(\ref*{the:Kishimoto_spectrum}.\arabic*)}]
  \item \label{en:Kishimoto_spectrum}%
    \(\alpha\) satisfies Kishimoto's condition:
    \eqref{eq:Kishimoto_condition} holds for all \(D\in \Her(A)\),
    \(b\in A\);
  \item \label{en:Kishimoto_spectrum_Mult}%
    \eqref{eq:Kishimoto_condition} holds for all \(D\in \Her(A)\) and
    \(b\in \Mult(A)\);
  \item \label{en:Kishimoto_spectrum_1}%
    \eqref{eq:Kishimoto_condition} holds for all \(D\in \Her(A)\)
    and \(b=1\);
  \item \label{en:Kishimoto_spectrum_Borchers}%
    \(\Gamma_\Bor(\alpha|_I)\neq \{1\}\) for all non-zero \(I\in
    \I^\alpha(A)\);
  \item \label{en:Kishimoto_spectrum_Connes}%
    for \(I\in \I^\alpha(A)\)
    with \(I\neq0\)
    there is \(D\in\Her^\alpha(I)\)
    with \(\Gamma(\alpha|_D)\neq \{1\}\);
  \item \label{en:Kishimoto_spectrum_commutator}%
    \(\sup {}\{\norm{x-\alpha(x)} \mid x\in D^+,\ \norm{x}=1\}=1\) for all
    \(D\in \Her(A)\).
  \end{enumerate}
\end{theorem}

\begin{proof}
  The conditions
  \ref{en:Kishimoto_spectrum_Mult}--\ref{en:Kishimoto_spectrum_Borchers}
  are equivalent by \cite{Kishimoto:Freely_acting}*{Theorem~2.1}.
  Conditions \ref{en:Kishimoto_spectrum_1}
  and~\ref{en:Kishimoto_spectrum_commutator} are equivalent by
  \cite{Olesen-Pedersen:Applications_Connes_3}*{Theorem 5.1}, and
  \ref{en:Kishimoto_spectrum_Borchers}
  and~\ref{en:Kishimoto_spectrum_Connes} are equivalent
  by~\eqref{eq:Bor_through_Connes}.  We check
  that~\ref{en:Kishimoto_spectrum}
  implies~\ref{en:Kishimoto_spectrum_Mult}, which is rather implicit
  in \cites{Kishimoto:Outer_crossed, Kishimoto:Freely_acting}.  Let
  \(a\in D^+\)
  with \(\norm{a}=1\)
  and \(\varepsilon>0\).
  Pick \(d\in \overline{a A a}{}^+\)
  and \(D_0\in\Her(D)\)
  as in Lemma~\ref{lem:technical}.  Kishimoto's condition for \(D_0\)
  and \(d b\in A\)
  gives \(x \in D^+\)
  with \(x d = x\),
  \(\norm{x}=1\),
  and \(\norm{x b \alpha(x)} = \norm{x d b \alpha(x)} < \varepsilon\).
\end{proof}

\begin{remark}
  \label{rem:spectrally non-trivial}
  Kishimoto~\cite{Kishimoto:Freely_acting} calls automorphisms
  satisfying~\ref{en:Kishimoto_spectrum_Borchers}
  \emph{freely acting}.  Pasnicu and
  Phillips~\cite{Pasnicu-Phillips:Spectrally_free} call such
  automorphisms \emph{spectrally non-trivial}.  We only mentioned
  the Borchers spectrum to show that the conditions used in
  \cites{Kishimoto:Freely_acting, Pasnicu-Phillips:Spectrally_free}
  are equivalent to what we call Kishimoto's condition.  We prefer
  the formulation using~\eqref{eq:Kishimoto_condition} because it
  generalises to Hilbert bimodules.
\end{remark}

The following definition names several triviality and non-triviality
conditions for automorphisms:

\begin{definition}
  \label{def:properly_outer}
  We call~\(\alpha\in\Aut(A)\) \emph{inner} or
  \emph{universally weakly inner} if there are unitaries~\(u\)
  in \(\Mult(A)\) or in the bidual \(\textup{W}^*\)\nb-algebra~\(A^{**}\)
  with \(\alpha = \Ad_u\), respectively.
  We call~\(\alpha\)
  \emph{partly inner} or \emph{partly universally weakly inner} if
  there are \(0\neq I\in \I^\alpha(A)\)
  and a unitary~\(u\)
  in \(\UM(I)\)
  or~\(I^{**}\),
  respectively, such that \(\alpha|_I= \Ad_u\).
  We call \(\alpha\in\Aut(A)\) \emph{outer},
  \emph{purely outer} or \emph{purely universally weakly outer} if it
  is not inner, not partly inner or not partly universally weakly inner,
  respectively, see \cite{Rieffel:Actions_finite}*{Section~1}.

  We call~\(\alpha\) \emph{properly outer} if \(\norm{\alpha|_I -
    \Ad_u}=2\) for any \(0\neq I\in \I^\alpha(A)\) and any unitary
  multiplier \(u\in\UM(I)\) (see
  \cite{Elliott:Simple_crossed}*{Definition~2.1}).

  Let~\(\widehat{A}\) be the spectrum of~\(A\).  We call~\(\alpha\)
  \emph{topologically non-trivial} if the set of
  \([\pi]\in\widehat{A}\) with \(\widehat{\alpha}[\pi] \neq [\pi]\)
  is dense or, equivalently, \(\{[\pi]\in \widehat{A}:
  [\pi\circ\alpha] = [\pi]\}\) has empty interior in~\(\widehat{A}\).
\end{definition}

\begin{theorem}
  \label{the:automorphism_Kishimoto_vs_properly_outer}
  Consider the following conditions for \(\alpha\in\Aut(A)\):
  \begin{enumerate}[wide,label=\textup{(\ref*{the:automorphism_Kishimoto_vs_properly_outer}.\arabic*)}]
  \item \label{en:automorphism_Kishimoto_vs_properly_outer_Kishimoto}%
    \(\alpha\) satisfies Kishimoto's condition;
  \item \label{en:automorphism_Kishimoto_vs_properly_outer_weakly_inner}%
    \(\alpha\) is purely universally weakly outer;
  \item \label{en:automorphism_Kishimoto_vs_properly_outer_topological}%
    \(\alpha\) is topologically non-trivial;
  \item \label{en:automorphism_Kishimoto_vs_properly_outer_distance2}%
    \(\alpha\) is properly outer;
  \item \label{en:automorphism_Kishimoto_vs_properly_outer_purely_outer}%
    \(\alpha\) is purely outer.
  \end{enumerate}
  Then
  \ref{en:automorphism_Kishimoto_vs_properly_outer_weakly_inner}%
  \(\Leftarrow\)\ref{en:automorphism_Kishimoto_vs_properly_outer_topological}%
  \(\Rightarrow\)\ref{en:automorphism_Kishimoto_vs_properly_outer_distance2}%
  \(\Rightarrow\)\ref{en:automorphism_Kishimoto_vs_properly_outer_purely_outer}%
  \(\Leftarrow\)\ref{en:automorphism_Kishimoto_vs_properly_outer_Kishimoto}.
  If~\(A\) is separable then
  \ref{en:automorphism_Kishimoto_vs_properly_outer_Kishimoto}--\ref{en:automorphism_Kishimoto_vs_properly_outer_distance2}
  are equivalent.  If~\(A\) is simple then
  \ref{en:automorphism_Kishimoto_vs_properly_outer_Kishimoto}%
  \(\Leftrightarrow\)\ref{en:automorphism_Kishimoto_vs_properly_outer_purely_outer}.
\end{theorem}

\begin{proof}
  The implications
  \ref{en:automorphism_Kishimoto_vs_properly_outer_weakly_inner}%
  \(\Leftarrow\)\ref{en:automorphism_Kishimoto_vs_properly_outer_topological}%
  \(\Rightarrow\)\ref{en:automorphism_Kishimoto_vs_properly_outer_distance2}
  follow from
  \cite{Olesen-Pedersen:Applications_Connes_2}*{Lemma~4.3} and the
  proof of
  \cite{Archbold-Spielberg:Topologically_free}*{Proposition~1},
  respectively.  The implication
  \ref{en:automorphism_Kishimoto_vs_properly_outer_distance2}\(
  \Rightarrow\)\ref{en:automorphism_Kishimoto_vs_properly_outer_purely_outer}
  is obvious.  To see
  \ref{en:automorphism_Kishimoto_vs_properly_outer_purely_outer}\(
  \Leftarrow\)\ref{en:automorphism_Kishimoto_vs_properly_outer_Kishimoto}
  assume that \(\alpha|_I= \Ad_u\) for some \(0\neq I\in
  \I^\alpha(A)\) and a unitary~\(u\) in \(\UM(I)\).  Let \(b\defeq
  u^*\).  Then \(\norm{ab\alpha_g(a))}=\norm{aau^*} = 1\) for all
  \(a\in I^+\) with \(\norm{a}=1\), which
  contradicts~\ref{en:Kishimoto_spectrum_Mult}.

  If~\(A\)
  is separable, then
  \ref{en:automorphism_Kishimoto_vs_properly_outer_Kishimoto}--%
  \ref{en:automorphism_Kishimoto_vs_properly_outer_distance2}
  are equivalent by
  \cite{Olesen-Pedersen:Applications_Connes_3}*{Theorem 6.6}.  Indeed,
  Kishimoto's condition is equivalent to
  \cite{Olesen-Pedersen:Applications_Connes_3}*{Theorem 6.6.(iv)} by
  Theorem~\ref{the:Kishimoto_spectrum}, and conditions
  \cite{Olesen-Pedersen:Applications_Connes_3}*{Theorem 6.6.(ii)
    and~(ix)} are
  \ref{en:automorphism_Kishimoto_vs_properly_outer_distance2}
  and~\ref{en:automorphism_Kishimoto_vs_properly_outer_weakly_inner},
  respectively.

  If~\(A\)
  is simple and~\(\alpha\)
  is outer, then \(\Gamma(\alpha)\neq\{1\}\)
  by \cite{Pedersen:Cstar_automorphisms}*{Corollary 8.9.10}.  This
  implies~\ref{en:Kishimoto_spectrum_Connes}, which is equivalent to
  Kishimoto's condition by Theorem~\ref{the:Kishimoto_spectrum}.  Thus
  \ref{en:automorphism_Kishimoto_vs_properly_outer_Kishimoto}
  and~\ref{en:automorphism_Kishimoto_vs_properly_outer_purely_outer}
  are equivalent.
  %
\end{proof}

In general, pure outerness does not imply proper outerness, even for
separable \(\Cst\)\nb-algebras.
We thank George Elliott for explaining the following counterexample.

\begin{example}
  \label{exa:outer_not_proper}
  Let~\(\N^\infty\)
  be the set of all sequences~\((n_k)_{k\in\N}\)
  with \(n_k\in\N\)
  for all \(k\in\N\)
  and \(n_k=0\)
  for all but finitely many~\(k\).
  Let \(\Hils = \ell^2(\N^\infty)\),
  viewed as the infinite tensor product of copies of~\(\ell^2(\N)\).
  Let \(A_n \subseteq \Bound(\Hils)\)
  be the unital \(\Cst\)\nb-subalgebra
  of~\(\Bound(\Hils)\)
  spanned by operators of the form \(x\otimes 1\)
  with \(x\in \Comp(\ell^2(\N^m))\)
  for \(m=0,1,\dotsc,n\).
  That is, \(x\otimes 1\)
  acts by~\(x\)
  on the tensor product of the first~\(m\)
  tensor factors~\(\ell^2(\N)\),
  and identically on the remaining tensor factors~\(\ell^2(\N)\).
  We have \(A_0 \subseteq A_1 \subseteq A_2 \subseteq \dotsb\).
  Let~\(A\)
  be the closure of~\(\bigcup A_n\).
  The \(\Cst\)\nb-algebra~\(A\)
  is isomorphic to the fixed-point algebra~\(\mathcal{O}_\infty^\T\)
  for the standard gauge action of~\(\T\)
  on the Cuntz algebra~\(\mathcal{O}_\infty\).
  It is an AF-algebra.  Let \(I_n\idealin A\)
  be the ideal generated by \(\Comp(\ell^2(\N^n))\otimes 1\).
  These ideals form a decreasing chain
  \(A=I_0 \supsetneq I_1 \supsetneq I_2 \supsetneq \dotsb\)
  with \(\bigcap I_n=0\).
  Any non-zero ideal of~\(A\)
  is of the form~\(I_n\) for some \(n\in\N\).

  Let \(P\in\Bound(\ell^2(\N))\)
  be the projection onto \(\ell^2(2\N)\),
  the even numbers.  Let \(P_m\in\Bound(\Hils)\)
  denote the operator that acts by~\(P\)
  on the \(m\)th
  tensor factor and identically on the other factors.  Let
  \(h \defeq \sum_{m=1}^\infty 2^{-m-1} P_m \in \Bound(\Hils)\).
  If \(a\in \Comp(\ell^2(\N^k))\otimes 1\),
  then \([P_m,a]\)
  vanishes for \(k<m\),
  and \([P_m,a]\in \Comp(\ell^2(\N^k))\otimes 1\)
  if \(k\ge m\).
  Thus \([P_m,A]\subseteq A\)
  for all \(m\in\N\)
  and hence \([h,A]\subseteq A\).
  The \Star{}derivation \(\delta(A)\defeq \ima [h,A]\)
  of~\(A\)
  generates an automorphism
  \(\alpha \defeq \exp(2\pi \delta) = \Ad_u\)
  with \(u \defeq \exp(2\pi \ima h)\).

  The automorphism~\(\alpha\)
  is universally weakly inner by construction.  If
  \(x\in\Comp(\ell^2(\N^n))\)
  is non-zero, then \(h\cdot (x\otimes 1)\)
  is not in~\(A\)
  because it contains~\(2^{-m-1} P\)
  as its \(m\)th
  tensor factor for \(m>n\),
  which has distance~\(2^{-m-1}\)
  from~\(\Comp(\ell^2(\N))\).
  Hence~\(h\)
  is not a multiplier of~\(I_n\)
  for any \(n\in\N\).
  Since \(\norm{h}\le \nicefrac12\),
  we get~\(h\)
  from~\(u\)
  by \(h = \log(u)\).
  So~\(u\)
  is not a multiplier of~\(I_n\)
  either.  If there were another unitary \(v\in \UM(I_n)\)
  with \(\Ad_v = \Ad_u\),
  then~\(v^* u\)
  would commute with the image of~\(I_n\)
  in \(\Bound(\Hils)\).
  Since~\(I_n\)
  acts irreducibly on~\(\Hils\),
  the unitary~\(v^* u\)
  would have to be a scalar multiple of~\(1\),
  contradicting \(u\notin \Mult(I_n)\).
  So~\(\alpha\) is not partly inner.
\end{example}




All the properties of a single automorphism defined above generalise
to group actions by requiring them pointwise.  Actions that satisfy
Kishimoto's condition pointwise are called \emph{aperiodic} in
\cites{Kwasniewski-Szymanski:Pure_infinite,
  Muhly-Solel:Morita_equivalence_of_tensor_algebras}.  Pasnicu and
Phillips \cite{Pasnicu-Phillips:Spectrally_free} call such actions
\emph{pointwise spectrally non-trivial}, compare
Remark~\ref{rem:spectrally non-trivial}.


\begin{definition}
  \label{def:pointwise_group_actions}
  Let \(\alpha\colon G\to \Aut(A)\) be an action of a discrete
  group~\(G\).  We call~\(\alpha\) \emph{pointwise outer},
  \emph{pointwise purely outer} or \emph{pointwise purely
    universally weakly outer} if, for each \(g\in G\setminus\{e\}\),
  the automorphism \(\alpha_g\) is outer, purely outer or purely
  universally weakly outer, respectively.  We call~\(\alpha\)
  \emph{aperiodic} if~\(\alpha_g\) satisfies Kishimoto's condition
  for all \(g\in G\setminus\{e\}\).
\end{definition}

\begin{definition}
  \label{def:topological_nontiriviality_and_freeness}
  Let \(\alpha\colon G\to \Aut(A)\) be an action of a discrete
  group~\(G\).  We call~\(\alpha\) \emph{pointwise topologically
    non-trivial} if~\(\alpha_g\) is topologically non-trivial for
  each \(g\in G\setminus\{e\}\), that is, for each \(g\in
  G\setminus\{e\}\), the set
  \begin{math}
    F_g \defeq \{[\pi]\in \widehat{A} \mid [\pi\circ\alpha]=[\pi]\}
  \end{math}
  has empty interior in \(\widehat{A}\).  We call~\(\alpha\)
  \emph{topologically free} if the dual topological
  action \(\widehat\alpha\colon G\to \Homeo(\widehat{A})\) is
  topologically free, that is, the union \(F_{g_1}\cup
  F_{g_2}\cup \dotsb \cup F_{g_n}\) has empty interior
  in~\(\widehat{A}\) for any \(g_1,\dotsc,g_n \in G\setminus\{e\}\)
  (see \cite{Archbold-Spielberg:Topologically_free}*{Definition~1}).
\end{definition}

The following lemma shows that topological freeness and pointwise
topological non-triviality are equivalent in the separable case.

\begin{lemma}
  \label{lem:separable_topological_equivalences}
  Let \(\alpha\colon G\to \Aut(A)\)
  be an action of a discrete group~\(G\)
  on a separable \(\Cst\)\nb-algebra~\(A\).
  The following are equivalent:
  \begin{enumerate}[wide,label=\textup{(\ref*{lem:separable_topological_equivalences}.\arabic*)}]
  \item \label{en:separable_topolol_pointwise_free_aut}%
    \(\alpha\) is pointwise topologically non-trivial;
  \item \label{en:separable_topolol_free_aut}%
    \(\mathcal\alpha\) is topologically free.
  \end{enumerate}
  Moreover, if~\(G\)
  is countable, then the above conditions are equivalent to:
  \begin{enumerate}[wide,label=\textup{(\ref*{lem:separable_topological_equivalences}.\arabic*)},resume]
  \item \label{en:separable_densly_free_aut}%
    \(\widehat\alpha\) is free on a dense subset of~\(\widehat{A}\).
  \end{enumerate}
\end{lemma}

\begin{proof}
  The implications \ref{en:separable_densly_free_aut}%
  \(\Rightarrow\)\ref{en:separable_topolol_free_aut}%
  \(\Rightarrow\)\ref{en:separable_topolol_pointwise_free_aut}
  are obvious.  Conversely, if~\(G\)
  is countable then \ref{en:separable_topolol_pointwise_free_aut}
  implies~\ref{en:separable_densly_free_aut} by the proof of
  \cite{Olesen-Pedersen:Applications_Connes_2}*{Proposition~4.4}.
  Since~\ref{en:separable_topolol_free_aut} holds for~\(G\)
  if and only if it holds for all countable subgroups of~\(G\),
  the equivalence \ref{en:separable_topolol_pointwise_free_aut}%
  \(\Leftrightarrow\)\ref{en:separable_topolol_free_aut}
  for countable~\(G\) implies the same for all~\(G\).
\end{proof}

The following theorem highlights the role of aperiodicity and
topological freeness.

\begin{theorem}
  \label{the:Kishimoto_to_detect}
  Let~\(A\) be a \(\Cst\)\nb-algebra and \(G\) a discrete group.  If
  \(\alpha\colon G\to\Aut(A)\) is an aperiodic or topologically free
  group action, then~\(A\) detects ideals in the reduced crossed
  product \(A\rtimes_{\alpha,\red} G\).
\end{theorem}

\begin{proof}
  The statement for topological freeness is
  \cite{Archbold-Spielberg:Topologically_free}*{Theorem~1}.  The
  statement using aperiodicity is contained, for instance, in
  \cite{Giordano-Sierakowski:Purely_infinite}*{Theorem~3.12}.
\end{proof}

The next theorem asserts that aperiodicity and topogical freeness are
also necessary for~\(A\)
to detect ideals in \(A\rtimes_{\alpha,\red} G\)
if~\(A\)
is separable and~\(G\)
is cyclic of infinite or square-free order.  This fails for most
groups~\(G\),
even if~\(A\)
is separable and simple, so that topological non-triviality and
aperiodicity are equivalent to various pointwise outerness conditions
by Theorem~\ref{the:automorphism_Kishimoto_vs_properly_outer} and
Lemma~\ref{lem:separable_topological_equivalences}.  For the
group~\(\Z/4\),
the example mentioned in Remark~\ref{rem:Phillips_examples} is a
counterexample: the crossed product is simple although
\(2\in \Z/4\setminus\{0\}\)
acts by an inner automorphism.  The noncommutative torus~\(A_\theta\)
is stably isomorphic to \(\Comp(\ell^2(\Z))\rtimes \Z^2\)
for a \(\Z^2\)\nb-action
on~\(\Comp(\ell^2(\Z))\).
It is simple for irrational~\(\theta\)
although all automorphisms of~\(\Comp(\ell^2(\Z))\) are inner.

\begin{theorem}
  \label{the:automorphism_detect_vs_Kishimoto}
  Let \(G=\Z\)
  or \(G=\Z/p\)
  for a square-free number~\(p\),
  let~\(A\)
  be a separable \(\Cst\)\nb-algebra
  and let \(\alpha\colon G\to\Aut(A)\)
  be a group action.  The following are equivalent:
  \begin{enumerate}[wide,label=\textup{(\ref*{the:automorphism_detect_vs_Kishimoto}.\arabic*)}]
  \item \label{en:automorphism_detect_vs_Kishimoto_detect}%
    \(A\)~detects ideals in \(A\rtimes_\alpha G= A\rtimes_{\alpha,\red} G\);
  \item \label{en:automorphism_detect_vs_Kishimoto_Kish}%
    \(\alpha\)~is aperiodic;
 \item \label{en:separable_topolo_free_aut}%
   \(\mathcal\alpha\)~is topologically free.
  \end{enumerate}
\end{theorem}

\begin{proof}[Proof of
  Theorem~\textup{\ref{the:automorphism_detect_vs_Kishimoto}}]
  The statements \ref{en:automorphism_detect_vs_Kishimoto_Kish}
  and~\ref{en:separable_topolo_free_aut} are equivalent by
  Theorem~\ref{the:automorphism_Kishimoto_vs_properly_outer} and
  Lemma~\ref{lem:separable_topological_equivalences}.  If \(G=\Z\),
  then the equivalence of
  \ref{en:automorphism_detect_vs_Kishimoto_detect}
  and~\ref{en:automorphism_detect_vs_Kishimoto_Kish} is contained in
  \cite{Olesen-Pedersen:Applications_Connes_3}*{Theorem~10.4}.  We
  claim that the proof of
  \cite{Olesen-Pedersen:Applications_Connes_3}*{Theorem~10.4} still
  works for \(G=\Z/p\)
  if~\(p\)
  is square-free and shows that
  \ref{en:automorphism_detect_vs_Kishimoto_detect}
  and~\ref{en:automorphism_detect_vs_Kishimoto_Kish} are equivalent.
  Since~\ref{en:automorphism_detect_vs_Kishimoto_Kish} always
  implies~\ref{en:automorphism_detect_vs_Kishimoto_detect} by
  Theorem~\ref{the:Kishimoto_to_detect}, we only have to look at how
  the implications (i)\(\Rightarrow\)(ii)\(\Rightarrow\)(iii)
  in \cite{Olesen-Pedersen:Applications_Connes_3}*{Theorem~10.4} are
  proved.  Let \(n\in\Z\)
  be such that \(\alpha^n\in\Aut(A)\)
  is not properly outer and~\(p\)
  does not divide~\(n\).
  Then~\(\alpha^n\)
  is not purely universally weakly outer by
  Theorem~\ref{the:automorphism_Kishimoto_vs_properly_outer}.  The
  proof in~\cite{Olesen-Pedersen:Applications_Connes_3} first provides
  \(B\in\Her^\alpha(A)\)
  such that \(\alpha^{n^2}|_B = \exp(\delta)\)
  for an \(\alpha\)\nb-invariant
  \Star{}derivation \(\delta\colon B\to B\).
  Then \cite{Olesen-Pedersen:Applications_Connes_3}*{Proposition~4.2}
  shows that \(n^2 \mathrel{\bot} \Gamma(\alpha)\).
  Since~\(p\)
  is square-free and does not divide~\(n\),
  it does not divide~\(n^2\).
  So \(\Gamma(\alpha)\neq \widehat{G}\),
  which is equivalent
  to~\ref{en:automorphism_detect_vs_Kishimoto_detect} by
  Theorem~\ref{the:Connes_spectrum_vs_detect_ideals}.
\end{proof}

For each \(n>0\),
let \(\T_n\defeq \{z\in \T\mid z^n=1\}\)
be the roots of unity of order~\(n\).
We also put \(\T_\infty\defeq \T\).
Let~\(\ord(g)\) be the order of an element~\(g\) of a group~\(G\).

\begin{theorem}
  \label{the:Equivalence_for_finite}
  Let~\(A\) be a \(\Cst\)\nb-algebra, \(G\) a discrete group and
  \(\alpha\colon G\to\Aut(A)\) a group action.  The following are
  equivalent:
  \begin{enumerate}[wide,label=\textup{(\ref*{the:Equivalence_for_finite}.\arabic*)}]
  \item \label{en:Equiv_finite_Kishimoto}%
    \(\alpha\) is aperiodic;
  \item \label{en:Equiv_finite_spectra1}%
    \(\Gamma(\alpha_g)=\T_{\ord(g)}\) for all \(g\in G\);
  \item \label{en:Equiv_finite_spectra2}%
    \(\Gamma(\alpha_g)\neq\{1\}\) for all \(g\in G\setminus\{e\}\);
  \end{enumerate}
  The above equivalent conditions imply the following:
  \begin{enumerate}[wide,label=\textup{(\ref*{the:Equivalence_for_finite}.\arabic*)},resume]
  \item \label{en:Equiv_finite_outer}%
    \(\alpha\) is pointwise purely outer.
  \end{enumerate}
  If~\(G\)
  is finite or~\(A\)
  is simple, then all the conditions
  \ref{en:Equiv_finite_Kishimoto}--\ref{en:Equiv_finite_outer} are
  equivalent.
\end{theorem}

\begin{proof}
  The implication \ref{en:Equiv_finite_spectra1}%
  \(\Rightarrow\)\ref{en:Equiv_finite_spectra2}
  is trivial, and
  \ref{en:Equiv_finite_spectra2}\(\Rightarrow\)\ref{en:Equiv_finite_Kishimoto}
  follows from Theorem~\ref{the:Kishimoto_spectrum}
  (condition~\ref{en:Kishimoto_spectrum_Connes} holds because
  \(\Gamma(\alpha_g)\subseteq \Gamma(\alpha_g|_D)\)
  for any restriction~\(\alpha_g|_D\)
  of~\(\alpha_g\)).
  Theorem \ref{the:automorphism_Kishimoto_vs_properly_outer} shows
  that \ref{en:Equiv_finite_Kishimoto} implies
  \ref{en:Equiv_finite_outer}.  To prove the first part of the
  theorem, it remains to show that~\ref{en:Equiv_finite_Kishimoto}
  implies~\ref{en:Equiv_finite_spectra1}.  So
  assume~\ref{en:Equiv_finite_Kishimoto} and pick \(g\in G\).
  The automorphism~\(\alpha_g\)
  generates an action \(\beta\colon \langle g \rangle \to \Aut(A)\)
  of the subgroup \(\langle g \rangle \cong \Z/\ord(g)\)
  of~\(G\)
  generated by~\(g\),
  which is aperiodic by~\ref{en:Equiv_finite_Kishimoto}.  Hence~\(A\)
  detects ideals in \(A\rtimes_\beta \langle g \rangle\)
  by Theorem~\ref{the:Kishimoto_to_detect}.  Thus
  Theorem~\ref{the:Connes_spectrum_vs_detect_ideals} and
  Lemma~\ref{lem:explanation_of_spectra_for_automorphisms} imply
  \(\Gamma(\alpha_g)=\T_{\ord(g)}\).

  If~\(A\)
  is simple, then conditions \ref{en:Equiv_finite_Kishimoto}
  and~\ref{en:Equiv_finite_outer} are equivalent by
  Theorem~\ref{the:automorphism_Kishimoto_vs_properly_outer}.  Thus it
  remains to show that~\ref{en:Equiv_finite_outer}
  implies~\ref{en:Equiv_finite_spectra1} if~\(G\)
  is finite.  Let \(g\in G\setminus\{e\}\)
  and let~\(\beta\)
  be the action of~\(\langle g\rangle\)
  generated by~\(\alpha_g\)
  as above.  By assumption, \(\beta\)
  is pointwise purely outer.  Therefore, \(A\)
  detects ideals in \(A\rtimes_\beta \langle g \rangle\)
  by \cite{Rieffel:Actions_finite}*{Theorem~1.1}.  Hence
  \(\Gamma(\alpha_g)=\Gamma(\beta)=\T_{\ord(g)}\)
  by Theorem~\ref{the:Connes_spectrum_vs_detect_ideals} and
  Lemma~\ref{lem:explanation_of_spectra_for_automorphisms}.
\end{proof}

\begin{example}
  \label{exa:non-detect_although_not_partly_inner}
  Let \(\bar\alpha\colon \Z\to\Aut(A)\) be the \(\Z\)\nb-action
  generated by the automorphism~\(\alpha\) of
  \(A=\mathcal{O}_\infty^\T\) built in
  Example~\ref{exa:outer_not_proper}.  Since~\(\alpha\) is
  universally weakly inner, so are all its powers~\(\bar\alpha_n\)
  for \(n\in\Z\).  On the one hand, the action~\(\bar\alpha\) is far
  from being aperiodic: \(\bar\alpha_n\)~cannot satisfy Kishimoto's
  condition for any \(n\in \Z\setminus\{0\}\) by
  Theorem~\ref{the:automorphism_Kishimoto_vs_properly_outer}.  On the
  other
  hand, \(\bar\alpha_n\) is purely outer for all
  \(n\in\Z\setminus\{0\}\) by the same argument as for~\(\alpha\).
  Thus the last statement in
  Theorem~\ref{the:Equivalence_for_finite} may fail if~\(G\) is
  infinite and~\(A\) is not simple.
\end{example}

\section{Hilbert bimodules, Fell bundles, and their crossed products}
\label{sec:Hilmi_crossed}

We recall some basic definitions and results about Hilbert
bimodules and Fell bundles.  See
\cites{Raeburn-Williams:Morita_equivalence, Exel:Partial_dynamical}
for more details on Hilbert bimodules, Morita equivalence, and Fell
bundles, and also the section on preliminaries
in~\cite{Kwasniewski-Szymanski:Pure_infinite}.

\begin{definition}
  \label{def:Hilbert_bimodule}
  A \emph{Hilbert \(A,B\)\nb-bimodule}~\(\E\)
  is a vector space with a right Hilbert \(B\)\nb-module
  and a left Hilbert \(A\)\nb-module
  structure, such that
  \({}_A\braket{x}{y}\cdot z = x\cdot\braket{y}{z}_B\)
  for all \(x,y,z\in \E\).
  The \emph{conjugate Hilbert \(B,A\)\nb-bimodule}~\(\E^*\)
  is the complex-conjugate vector space of~\(\E\)
  with the canonical Hilbert \(B,A\)\nb-bimodule
  structure: \(b\cdot^* x\cdot^* a \defeq a^*\cdot x\cdot b^*\),
  \({}_B\braket{x}{y}^* \defeq \braket{x}{y}_B\),
  \(\braket{x}{y}_A^* \defeq {}_A\braket{x}{y}\).
  If \(A=B\), we call~\(\E\) a \emph{Hilbert \(A\)\nb-bimodule}.

  The closed linear spans of the left and right inner products
  in~\(\E\)
  are ideals in \(A\)
  and~\(B\), which we denote by \({}_A\braket{\E}{\E}\)
  and~\(\braket{\E}{\E}_B\),
  respectively.  The Hilbert bimodule~\(\E\)
  is \emph{full over~\(A\)}
  if \({}_A\braket{\E}{\E}=A\);
  \emph{full over~\(B\)}
  if \(\braket{\E}{\E}_B=B\);
  and \emph{full} or an \emph{equivalence
    \textup{(}bimodule\textup{)}} if both \({}_A\braket{\E}{\E}=A\)
  and \(\braket{\E}{\E}_B=B\).
  The \(\Cst\)\nb-algebras
  \(A\)
  and~\(B\)
  are \emph{Morita--Rieffel equivalent} if an equivalence
  \(A,B\)-bimodule exists.
\end{definition}

\begin{definition}
  \label{def:Hilbert_bimodule_from_isomorphism}
  Let \(\varphi\colon B\congto A\) be an isomorphism.
  Let~\(A_\varphi\) be~\(A\)
  with the obvious left Hilbert \(A\)\nb-module
  structure and the right Hilbert \(B\)\nb-module
  structure \(a\cdot b\defeq a\varphi(b)\),
  \(\braket{a_1}{a_2}_B \defeq \varphi^{-1}(a_1^*\cdot a_2)\)
  for \(a,a_1,a_2\in A\),
  \(b\in B\).  This is an equivalence bimodule.
\end{definition}

We view a Hilbert \(A,B\)-bimodule as an arrow from~\(B\) to~\(A\),
that is, the source is on the right and the range on the left.  This
convention is helpful in connection with groupoid
\(\Cst\)\nb-algebras and with Fell bundles over groups and inverse
semigroups, see~\cite{Buss-Meyer:Actions_groupoids}.  In this
article, the direction convention does not matter much.  It ensures
that the isomorphisms in Example~\ref{exa:automorphism_to_Fell}
below preserve the natural \(G\)\nb-gradings.

\begin{remark}
  \label{rem:opposite_convention}
  The Hilbert bimodule~\(A_\varphi\)
  is isomorphic through~\(\varphi^{-1}\)
  to~\(B\)
  with the obvious right Hilbert \(B\)\nb-module
  structure and the left Hilbert \(A\)\nb-module
  structure by \(a\cdot b \defeq \varphi^{-1}(a)b\)
  and \({}_A\braket{b_1}{b_2} \defeq \varphi(b_1 b_2^*)\)
  for \(a\in A\)
  and \(b,b_1,b_2\in B\).
  We denote this Hilbert \(A,B\)\nb-bimodule
  by~\({}_{\varphi^{-1}}B\).
  The construction~\({}_\varphi B\)
  (without inverse) still gives a \(\Cst\)\nb-correspondence
  for any morphism \(\varphi\colon A\to B\).
  This is why the conventions for~\({}_\varphi B\)
  are used, for instance, in~\cite{Buss-Meyer-Zhu:Higher_twisted}.  So
  the conventions in~\cite{Buss-Meyer-Zhu:Higher_twisted} would
  associate the equivalence bimodule~\(A_\varphi\)
  to~\(\varphi^{-1}\) and not to~\(\varphi\).
\end{remark}

We may compose Hilbert bimodules by the interior tensor product,
see~\cite{Lance:Hilbert_modules}.  The bimodule
structure and the inner products induce isomorphisms
of Hilbert bimodules
\begin{alignat}{2}
  \label{eq:left_mult}
  A\otimes_A \E&\congto \E,&\qquad
  a\otimes \xi&\mapsto a\cdot\xi,\\
  \label{eq:right_mult}
  \E\otimes_B B&\congto \E,&\qquad
  \xi\otimes b&\mapsto \xi\cdot b,\\
  \label{eq:left_innprod}
  \E\otimes_B \E^*&\congto {}_A\braket{\E}{\E}\idealin A,&\qquad
  \xi_1\otimes\xi_2&\mapsto {}_A\braket{\xi_1}{\xi_2},\\
  \label{eq:right_innprod}
  \E^*\otimes_A \E&\congto \braket{\E}{\E}_B\idealin B,&\qquad
  \xi_1\otimes\xi_2&\mapsto \braket{\xi_1}{\xi_2}_B.
\end{alignat}
These are part of a bicategory structure with
Hilbert bimodules as arrows, see~\cite{Buss-Meyer:Actions_groupoids}.
More generally, if \(I\idealin A\)
and \(J\idealin B\)
are ideals, the multiplication maps restrict to isomorphisms
\begin{equation}
  \label{eq:restricted_mult}
  I\otimes_A \E \congto I\cdot \E\subseteq \E,\qquad
  \E\otimes_B J \congto \E\cdot J\subseteq \E.
\end{equation}

A Hilbert \(A,B\)-bimodule~\(\E\)
is an \({}_A\braket{\E}{\E},\braket{\E}{\E}_B\)-equivalence
bimodule and thus an equivalence between ideals in \(A\)
and~\(B\).
Conversely, any equivalence between ideals in \(A\)
and~\(B\)
comes from a unique Hilbert \(A,B\)\nb-bimodule.
Thus Hilbert \(A,B\)-bimodules
are the same as equivalences between ideals in \(A\)
and~\(B\).
Hilbert bimodules are interpreted
in~\cite{Buss-Meyer:Actions_groupoids} as \emph{partial equivalences}
of \(\Cst\)\nb-algebras
because a \emph{partial isomorphism} is defined as an isomorphism
between ideals in \(A\) and~\(B\).

A Hilbert \(A,B\)-bimodule~\(X\) induces a \emph{dual partial homeomorphism}
\[
\widehat{X}\colon \widehat{B}\supseteq \widehat{\braket{X}{X}_B}
\congto
\widehat{{}_A\braket{X}{X}}
\subseteq \widehat{A},\qquad
[\pi]\mapsto [X\dashind^A_B (\pi)],
\]
where \(X\dashind^A_B (\pi)\colon A \to \Bound(X\otimes_\pi
\Hils_\pi)\) is the Hilbert space~\(X\otimes_\pi \Hils_\pi\) with
the obvious representation of~\(A\).  The partial homeomorphism
of~\(\widehat{A}\)
associated to the identity Hilbert \(A\)\nb-bimodule~\(A\)
is the identity on~\(\widehat{A}\),
and the partial homeomorphism associated to \(X\otimes_B Y\)
for two composable Hilbert bimodules is the product
\(\widehat{X}\circ\widehat{Y}\)
of partial homeomorphisms.  Hence~\(\widehat{X^*}\)
is the partial inverse~\(\widehat{X}^{-1}\)
of~\(\widehat{X}\)
(see \cites{Kwasniewski-Szymanski:Pure_infinite,
  Abadie-Abadie:Ideals}).

\begin{example}
  \label{exa:widehat_partial_auto}
  Let \(\varphi\colon J\congto I\) be an isomorphism between two
  ideals \(I\idealin A\) and \(J\idealin B\).  Consider the
  equivalence \(I, J\)-bimodule~\(I_\varphi\) as a Hilbert
  \(A,B\)-bimodule.  Then~\(A\) acts on the left on~\(I_\varphi\) by
  multiplication and~\(B\) acts on the right by \(x\cdot b \defeq
  \varphi(\varphi^{-1}(x)b)\) for \(x\in I\), \(b\in B\).  The
  partial homeomorphism \(\widehat{I_\varphi}\colon
  \widehat{B}\supseteq \widehat{J} \congto \widehat{I}\subseteq
  \widehat{A}\) is the partial inverse to \(\widehat{\varphi}\colon
  \widehat{A}\supseteq \widehat{I} \congto \widehat{J}\subseteq
  \widehat{B}\).  In particular, \(\widehat{A_\alpha} =
  \widehat{\alpha}^{-1}\) for an automorphism~\(\alpha\).
\end{example}

\begin{definition}
  \label{def:invariant_ideal}
  Let~\(X\)
  be a Hilbert \(A\)\nb-bimodule.
  An ideal \(I\idealin A\)
  is \emph{\(X\)\nb-invariant}
  if \(I X = X I\).
  Let~\(\I^X(A)\)
  be the lattice of \(X\)\nb-invariant
  ideals.  We call~\(X\)
  \emph{minimal} if \(\I^X(A) = \{0,A\}\).
  If \(I\in \I^X(A)\),
  then \(X I\)
  is naturally a Hilbert \(I\)\nb-bimodule,
  called \emph{restriction of~\(X\)
    to~\(I\)},
  and the quotient~\(X/XI\)
  is naturally a Hilbert \(A/I\)\nb-bimodule,
  called \emph{the restriction of~\(X\) to~\(A/I\)}.
\end{definition}

\begin{remark}
  \label{rem:restrictions_dual}
  Let \(f\colon M\supseteq \Delta \congto f(\Delta) \subseteq M\)
  be a partial bijection of a set~\(M\).
  We call a subset \(U\subseteq M\)
  \emph{\(f\)\nb-invariant}
  if \(f(U\cap \Delta) = U\cap f(\Delta)\).
  For any such~\(U\),
  let~\(f|_U\)
  be the restriction of~\(f\)
  to a partial bijection
  \(U\supseteq U\cap\Delta \congto U\cap f(\Delta) \subseteq U\)
  of~\(U\).
  If~\(X\)
  is a Hilbert \(A\)\nb-bimodule,
  there is a bijection between \(X\)\nb-invariant
  ideals in~\(A\)
  and \(\widehat{X}\)\nb-invariant
  open subsets of~\(\widehat{A}\).
  Moreover, if \(I\in \I^X(A)\),
  then \cite{Kwasniewski-Szymanski:Pure_infinite}*{Lemma~2.2} gives
  \(\widehat{X I} = \widehat{X}|_{\widehat{I}}\) and \(\widehat{X/X I}
  = \widehat{X}|_{\widehat{A}\setminus \widehat{I}}\).
\end{remark}

\begin{definition}[\cite{Exel:Partial_dynamical}*{Definition~16.1}]
  \label{def:Fell_bundle}
  Let~\(G\)
  be a discrete group.  A \emph{Fell bundle}~\(\mathcal{B}\)
  over~\(G\)
  consists of Banach spaces \((B_g)_{g\in G}\)
  with bilinear, associative multiplication maps
  and conjugate-linear, antimultiplicative involutions
  \[
  {\cdot}\colon B_g \times B_h \to B_{g h},\qquad
  ^*\colon B_g \to B_{g^{-1}}
  \]
  for \(g,h\in G\), such that
  \(\norm{ab}\leq \norm{a}\norm{b}\) and \(\norm{a^*a}=\norm{a}^2\)
  for all \(a\in B_g\),
  \(b\in B_h\),
  \(g,h\in G\), and for each
  \(a\in B_g\) there is \(c\in B_e\) with \(a^* a = c^* c\).
  Then \(A\defeq B_e\)
  is a \(\Cst\)\nb-algebra and \(a^*a \geq 0\)
  in~\(A\)
  for all \(a\in B_g\),
  \(g\in G\).
  So each~\(B_g\)
  is a Hilbert \(A\)\nb-bimodule
  with inner products \({}_A\braket{x}{y} \defeq x\cdot y^*\)
  and \(\braket{x}{y}_A \defeq x^*\cdot y\)
  for \(x,y\in B_g\),
  \(g\in G\).
  Linear extension of the multiplication and involution defines a
  \Star{}algebra structure on \(\bigoplus_{g\in G} B_g\).
  The (full) \emph{cross-section \(\Cst\)\nb-algebra}
  \(\Cst(\mathcal{B})\)
  of~\(\mathcal{B}\)
  is the \(\Cst\)\nb-completion
  of the \Star{}algebra \(\bigoplus_{g\in G} B_g\)
  for its maximal \(\Cst\)\nb-norm.
  The \emph{reduced cross-section \(\Cst\)\nb-algebra}
  \(\Cred(\mathcal{B})\)
  of~\(\mathcal{B}\)
  is the \(\Cst\)\nb-completion
  of the \Star{}algebra \(\bigoplus_{g\in G} B_g\)
  in the minimal \(\Cst\)\nb-norm \(\norm{\cdot}_\red\) that satisfies
  \begin{equation}
    \label{eq:Fell_minimal_norm}
    \norm{b_e}\le \biggl\| \sum_{g\in G} b_g \biggr\|_\red\qquad
    \text{for all}\ \sum_{g\in G} b_g \in \bigoplus_{g\in G} B_g.
  \end{equation}
\end{definition}

The coordinate projection \(\bigoplus_{g\in G} B_g \to B_e=A\)
extends to a faithful conditional expectation \(E\colon
\Cred(\mathcal{B}) \to A\).  Exel calls a Fell bundle amenable if
\(\Lambda\colon \Cst(\mathcal{B})\to \Cred(\mathcal{B})\) is an
isomorphism.  This always happens if the underlying group~\(G\) is
amenable and, in particular, if~\(G\) is Abelian, see
\cite{Exel:Partial_dynamical}*{Theorem 20.7}.

\begin{remark}
  Let \(\mathcal{B}=(B_g)_{g\in G}\)
  be a Fell bundle over a discrete Abelian group~\(G\).
  There is a \emph{dual action}~\(\beta\)
  of~\(\widehat{G}\)
  on~\(\Cst(\mathcal{B})\)
  defined by \(\beta_\chi(x) = \chi(g)\cdot x\)
  for all \(\chi\in\widehat{G}\),
  \(g\in G\),
  \(x\in B_g \subseteq \Cst(\mathcal{B})\).
  Thus \(\bigoplus_{g\in G} B_g\)
  is the decomposition of~\(\Cst(\mathcal{B})\)
  into its homogeneous subspaces for the \(\widehat{G}\)\nb-action.
  Conversely, if~\(B\)
  is any \(\Cst\)\nb-algebra
  with a continuous \(\widehat{G}\)\nb-action~\(\beta\),
  then the homogeneous subspaces
  \[
  B_g \defeq \{b\in B\mid \beta_\chi(b) = \chi(g)\cdot b
  \ \text{for all}\ \chi\in\widehat{G}\}
  \]
  with the multiplication and involution from~\(B\)
  form a Fell bundle.  Moreover, \(B\cong\Cst(\mathcal{B})\)
  by the obvious \(\widehat{G}\)\nb-equivariant
  isomorphism.  Thus Fell bundles over~\(G\)
  are the same as spectral decompositions for \(\widehat{G}\)\nb-actions
  on \(\Cst\)\nb-algebras.
  This is a Fell bundle variant of Takai duality.
\end{remark}

\begin{example}[Crossed products]
  \label{exa:automorphism_to_Fell}
  Let \(\alpha\colon G\to\Aut(A)\)
  be a group action.
  Then \(\mathcal{A}_\alpha\defeq (A_{\alpha_g})_{g\in G}\)
  with the multiplication maps
  \((a,g)\cdot (b,h) \defeq (a\alpha_g(b),g\cdot h)\)
  and involutions \((a,g)^* \defeq (\alpha_g^{-1}(a^*),g^{-1})\)
  for all \(a,b\in A\),
  \(g,h\in G\)
  is a Fell bundle over~\(G\).  We have natural isomorphisms
  \[
  \Cst(\mathcal{A}_\alpha)\cong A\rtimes_\alpha G\qquad
  \Cred(\mathcal{A}_\alpha)\cong A\rtimes_{\alpha,\red} G.
  \]
  One could also associate to~\(\alpha\) the Fell
  bundle~\({}_\alpha\mathcal{A} \defeq ({}_{\alpha_g^{-1}}A)_{g\in
    G}\).  The isomorphisms \(A_{\alpha_g}\cong
  {}_{\alpha_g^{-1}}A\) of Hilbert bimodules in
  Remark~\ref{rem:opposite_convention} combine to an isomorphism of
  Fell bundles \(\mathcal{A}_{\alpha}\cong{_\alpha}\mathcal{A}\).
\end{example}

\begin{example}[Crossed products by partial actions]
  Example \ref{exa:automorphism_to_Fell} generalises naturally to
  partial actions.
  Let \(\alpha = (\alpha_g)_{g \in G}\)
  be a \emph{partial action of~\(G\)
    on a \(\Cst\)\nb-algebra
    \(A\)},
  that is, for each \(g \in G\),
  \(\alpha_g\colon D_{g^{-1}}\congto D_g\)
  is an isomorphism between ideals of~\(A\)
  such that \(\alpha_e = \id_A\)
  and \(\alpha_{g h}\)
  extends \(\alpha_g\circ \alpha_h\)
  for all \(g,h \in G\).
  Then \(\mathcal{A}_\alpha \defeq ((D_g)_{\alpha_g})_{g\in G}\)
  with involutions as above and the multiplication maps
  \((a,g)\cdot (b,h) \defeq (\alpha_g(\alpha_{g^{-1}}(a)\cdot b),g\cdot h)\)
  for all \(a\in D_g\),
  \(b\in D_h\),
  \(g,h\in G\)
  is a Fell bundle over~\(G\).
  The \emph{crossed products} \(A\rtimes_\alpha G\)
  and \(A\rtimes_{\alpha,\red} G\)
  may be defined as \(\Cst(\mathcal{A}_\alpha)\)
  and \(\Cred(\mathcal{A}_\alpha)\),
  respectively, see \cite{Exel:Partial_dynamical}*{Proposition~16.28}.
  Fell bundles may be interpreted as partial group actions by Hilbert
  bimodules, compare~\cite{Buss-Meyer:Actions_groupoids}.  The full
  and reduced section \(\Cst\)\nb-algebras
  of a Fell bundle play the role of the full and reduced crossed
  products for a partial group action.
\end{example}

\begin{example}[Twisted cross section algebras]
  \label{ex:twisted_fell_bundles}%
  Let \(\mathcal{B}=(B_g)_{g\in G}\)
  be a Fell bundle over a discrete group \(G\)
  and let \(\omega\colon G\times G\to \T\)
  be a \(2\)-cocycle, that is,
  \[
  \omega(g,h)\omega(gh,k)=\omega(g,hk)\omega(h,k)
  \quad\text{and}\quad
  \omega(e,g)=\omega(g,e)=1
  \]
  for all \(g,h,k \in G\).
  Then we deform the multiplication and involution in~\(\mathcal{B}\)
  by the formulas \(b\cdot_\omega c\defeq \omega(g,h) bc\),
  \(b^{* \omega}=\overline{\omega(g,g^{-1})} b^*\)
  for \(b\in B_g\),
  \(c\in B_h \).
  This gives another Fell bundle
  \(\mathcal{B}^{\omega}=(B_g)_{g\in G}\),
  see \cite{Raeburn:Deformations}*{Proposition~3.3}.  We call
  \(\Cst(\mathcal{B}^\omega)\)
  and \(\Cred(\mathcal{B}^\omega)\)
  the \emph{full} and \emph{reduced cross section \(\Cst\)\nb-algebra
    of~\(\mathcal{B}\)
    twisted by~\(\omega\)},
  respectively.  If \(\alpha = (\alpha_g)_{g \in G}\)
  is a partial action of~\(G\) on a \(\Cst\)\nb-algebra~\(A\), then
  \[
  A\rtimes_\alpha^\omega G =\Cst(\mathcal{A}_\alpha^{\omega}), \qquad
  A\rtimes_{\alpha,\red}^\omega G=\Cred(\mathcal{A}_\alpha^{\omega})
  \]
  are the twisted full and reduced partial crossed products,
  compare~\cite{Exel:TwistedPartialActions}.
\end{example}

\begin{example}[Crossed products by Hilbert bimodules]
  \label{def:Fell_bundle_for_Hilbert_bimodule}
  Let~\(X\)
  be a Hilbert \(A\)\nb-bimodule.
  Let \(X_0 \defeq A\),
  \(X_n \defeq X^{\otimes_A n}\),
  and \(X_{-n} \defeq (X^*)^{\otimes_A n}\)
  for \(n\in\N\).
  This becomes a Fell bundle for the obvious involutions and the
  obvious multiplication maps between \(X_n\)
  and~\(X_m\)
  if \(n,m\)
  have the same sign, and more complicated multiplication maps that
  use the inner product maps \eqref{eq:left_innprod}
  and~\eqref{eq:right_innprod} if the signs are different.  The
  \emph{Hilbert bimodule crossed product} \(A\rtimes_X \Z\)
  is the cross-section \(\Cst\)\nb-algebra
  \(\Cst((X_n)_{n\in \Z})\)
  of this Fell bundle~\((X_n)_{n\in \Z}\),
  see~\cite{Abadie-Eilers-Exel:Morita_bimodules}.
  A Fell bundle~\((B_n)_{n\in\Z}\)
  is of this form for some Hilbert bimodule~\(X\)
  if and only if \((B_1)^{\otimes_A n} \cong B_n\)
  for \(n>0\),
  and then \(A=B_0\)
  and \(X=B_1\).
  Such Fell bundles are called \emph{semi-saturated} in
  \cite{Exel:Circle_actions}*{Definition~4.1}.
\end{example}

Let \(\mathcal{B}=(B_g)_{g\in G}\) be a Fell bundle and let
\(A\defeq B_e\) be its unit fibre.  An ideal \(I\idealin A\) is
\emph{\(\mathcal{B}\)\nb-invariant} if it is \(B_g\)\nb-invariant
for all \(g\in G\), that is, \(I B_g = B_g I\) for all \(g\in G\).
Let \(\I^{\mathcal{B}}(A) \subseteq \I(A)\) be the lattice of
\(\mathcal{B}\)\nb-invariant ideals.  We call~\(\mathcal{B}\)
\emph{minimal} if \(\I^\mathcal{B}(A) = \{0,A\}\).  If \(I\in
\I^{\mathcal{B}}(A)\), then the family of restrictions
\(\mathcal{B}|_I = (B_g I)_{g\in G}\) forms a Fell subbundle (even
an ideal) of~\(\mathcal{B}\), which we call the \emph{restriction
  of~\(\mathcal{B}\) to~\(I\)}.  We \emph{restrict the Fell bundle
  to~\(A/I\)} by taking \(\mathcal{B}|_{A/I} \defeq
\mathcal{B}\mathbin/\mathcal{B}|_I = (B_g/B_g I)_{g\in G}\) with the induced
multiplication maps and involutions.

An ideal~\(J\)
in~\(\Cred(\mathcal{B})\)
is called \emph{graded} if \(\bigoplus (B_g \cap J)\)
is dense in~\(J\).
Let \(\I^{\widehat{G}}(\Cred(\mathcal{B}))\)
denote the lattice of graded ideals.  If~\(G\)
is Abelian, then an ideal~\(J\)
in \(\Cst(\mathcal{B})=\Cred(\mathcal{B})\)
is graded if and only if~\(J\)
is invariant under the dual action
\(\beta\colon \widehat{G}\to \Aut(\Cst(\mathcal{B}))\).

\begin{proposition}
  \label{pro:gauge-invariance_vs_separation}
  Let~\(\mathcal{B}\) be a Fell bundle over~\(G\).
  The map
  \[
  \I(\Cred(\mathcal{B})) \to \I(A),\qquad
  I \mapsto I\cap A,
  \]
  preserves intersections, and its image
  is~\(\I^{\mathcal{B}}(A)\),
  the lattice of \(\mathcal{B}\)\nb-invariant
  ideals in~\(A\).
  It restricts to a lattice isomorphism
  \(\I^{\widehat{G}}(\Cred(\mathcal{B})) \congto \I^{\mathcal{B}}(A)\)
  on the sublattice \(\I^{\widehat{G}}(\Cred(\mathcal{B}))\)
  of graded ideals.  The inverse isomorphism maps
  \(I\in\I^{\mathcal{B}}(A)\) to the graded ideal
  \(\Cred(\mathcal{B}|_I)\) in \(\Cred(\mathcal{B})\).
\end{proposition}

\begin{proof}
  This follows from
  \cite{Kwasniewski-Szymanski:Pure_infinite}*{Proposition~3.2 and
    Corollary~3.3}.
\end{proof}

\begin{corollary}
  \label{cor:detection_in_reduced_cross}
  The unit fibre \(A=B_e\)
  detects ideals in \(\Cred(\mathcal{B})\)
  if and only if each non-zero ideal in~\(\Cred(\mathcal{B})\)
  contains a non-zero graded ideal.
\end{corollary}

\begin{corollary}
  \label{cor:separation_in_reduced_cross}
  The unit fibre \(A=B_e\)
  separates ideals in~\(\Cred(\mathcal{B})\)
  if and only if each ideal in~\(\Cred(\mathcal{B})\) is graded.
\end{corollary}

\begin{definition}
  A Fell bundle~\(\mathcal{B}\)
  is called \emph{exact} if the canonical maps
  \[
  \Cred(\mathcal{B}) \mathbin{/} \Cred(\mathcal{B}|_I) \to
  \Cred(\mathcal{B}|_{A/I})
  \]
  are isomorphisms for all \(I\in\I^{\mathcal{B}}(A)\)
  (see \cite{Abadie-Abadie:Ideals}*{Definition~3.14} or
  \cite{Kwasniewski-Szymanski:Pure_infinite}*{Definition~3.4}).  This
  holds automatically if the group~\(G\) is exact.
\end{definition}

\begin{proposition}
  \label{pro:separates_vs_exactness_B-residual_detect}
  The unit fibre \(A=B_e\) separates ideals
  in~\(\Cred(\mathcal{B})\) if and only if~\(\mathcal{B}\) is exact
  and~\(A/I\) detects ideals in \(\Cred(\mathcal{B}|_{A/I})\) for
  each \(I\in\I^{\mathcal{B}}(A)\) with \(I\neq A\).
\end{proposition}

\begin{proof}
  Use Corollary~\ref{cor:separation_in_reduced_cross} and
  \cite{Kwasniewski-Szymanski:Pure_infinite}*{Theorem~3.12}.
\end{proof}

The Hilbert \(A\)\nb-bimodules~\(B_g\) in a Fell bundle induce
partial homeomorphisms~\(\widehat{B_g}\) of~\(\widehat{A}\) as
above.  The range of~\(\widehat{B}_g\) for \(g\in G\) is the open
subset of~\(\widehat{A}\) corresponding to the ideal \(D_g\defeq
{}_A\braket{B_g}{B_g} = B_g B_{g^{-1}}\).

\begin{lemma}[\cite{Abadie-Abadie:Ideals}*{Proposition~2.2}]
  \label{lem:partial_action_on_irreps}
  The family \(\widehat{\mathcal{B}}\defeq (\widehat{B_g})_{g\in G}\)
  forms a partial action of~\(G\)
  on the space~\(\widehat{A}\),
  that is,
  \(\widehat{B_g}\colon \widehat{D}_{g^{-1}}\congto \widehat{D_g}\)
  are homeomorphisms between open subsets of~\(\widehat{A}\)
  such that \(\widehat{B_e} = \id_{\widehat{A}}\)
  and \(\widehat{B_{g h}}\)
  extends \(\widehat{B_g}\circ \widehat{B_h}\) for \(g,h \in G\).
\end{lemma}


\begin{remark}
  \label{rem:restrictions_of_duals_to_Fell_bundles}
  Remark~\ref{rem:restrictions_dual} generalises to Fell bundles in
  the obvious way: the open subset~\(\widehat{I}\)
  of~\(\widehat{A}\)
  is invariant for the partial action
  \(\widehat{\mathcal{B}} = (\widehat{B_g})_{g\in G}\)
  if and only if~\(I\)
  is \(\mathcal{B}\)\nb-invariant,
  and the partial actions \(\widehat{\mathcal{B}|_I}\)
  and \(\widehat{\mathcal{B}|_{A/I}}\)
  dual to the restrictions of~\(\mathcal{B}\)
  to \(I\)
  and~\(A/I\)
  agree with the restrictions
  \((\widehat{B_g}|_{\widehat{I}})_{g\in G}\)
  and \((\widehat{B_g}|_{\widehat{A}\setminus \widehat{I}})_{g\in G}\)
  of the partial action
  \(\widehat{\mathcal{B}} = (\widehat{B_g})_{g\in G}\)
  to \(\widehat{I}\) and~\(\widehat{A/I}\), respectively.
\end{remark}

\section{Non-triviality conditions for Hilbert bimodules and Fell bundles}
\label{sec:nontriviality}

We are going to generalise various non-triviality conditions from
automorphisms and group actions to Hilbert bimodules and to Fell
bundles over discrete groups.

It will become useful in a future project to formulate the following
definition and lemma for arbitrary bimodules, without inner
products.

\begin{definition}
  \label{def:Kishimoto_Hilm}
  Let~\(A\) be a \(\Cst\)\nb-algebra and~\(X\) an \(A\)\nb-bimodule.
  Let \(\Kish(X)\subseteq X\) be the set of all \(x\in X\) with
  \begin{equation}
    \label{eq:inf_equation}
    \inf {}\{\norm{a x a} \mid a\in D^+,\ \norm{a}=1\}=0
    \qquad\text{for all } D\in \Her(A).
  \end{equation}
  We say that \(x\in X\)
  \emph{satisfies Kishimoto's condition} if \(x\in\Kish(X)\),
  and that~\(X\)
  \emph{satisfies Kishimoto's condition} if \(\Kish(X)=X\).
\end{definition}

\begin{lemma}
  \label{lem:kishimoto_subspace}
  For any \(A\)\nb-bimodule~\(X\), the subset~\(\Kish(X)\) is a
  closed linear subspace of~\(X\).
\end{lemma}

\begin{proof}
  The subset~\(\Kish(X)\)
  is clearly closed under limits and under multiplication by scalars.
  We show that it is closed under addition.  Let \(x,y\in \Kish(X)\).
  We want to prove that \(x+y\in\Kish(X)\).
  We may assume \(\norm{x}=1\).
  Let \(D\in\Her(A)\)
  and \(\varepsilon>0\).
  Since \(x\in\Kish(X)\),
  there is \(b\in D^+\)
  with \(\norm{b}=1\)
  and \(\norm{b x b}<\varepsilon\).
  Lemma~\ref{lem:technical} gives \(d\in (b D b)^+\)
  such that \(D_0 \defeq \{a\in D\mid d a = a = a d\}\)
  is a non-zero hereditary subalgebra of~\(D\)
  and hence of~\(A\),
  and \(\norm{a-b a} = \norm{a-a b}<\varepsilon \norm{a}\)
  for all \(a\in D_0^+\).
  Therefore, if \(a\in D_0^+\) and \(\norm{a}\le1\), then
  \begin{align*}
    \norm{a x a}
    &\le \norm{(a-a b) x a} + \norm{a b x (a-b a)}
    + \norm{a b x b a} < 3\varepsilon.  
  \end{align*}
  Since \(y\in\Kish(X)\)
  and \(D_0\in\Her(A)\),
  there is \(a\in D_0^+ \subseteq D^+\)
  with \(\norm{a}=1\)
  and \(\norm{a y a}<\varepsilon\).
  Hence \(\norm{a(x+y)a}<4\varepsilon\).
  This proves that \(x+y\in \Kish(X)\).
\end{proof}

From now on, \(X\) will be a Hilbert bimodule over a
\(\Cst\)\nb-algebra~\(A\).  It may be weakly completed to a Hilbert
bimodule~\(X^{**}\) over the bidual
\(\textup{W}^*\)\nb-algebra~\(A^{**}\),
see~\cite{Buss-Exel-Meyer:Reduced}.

\begin{definition}
  \label{def:topological_condition_Hilm}
  A Hilbert \(A\)\nb-bimodule~\(X\) is
  \emph{topologically non-trivial} if the subset
  \(\{[\pi]\in \widehat{A}\mid \widehat{X}([\pi])=[\pi]\}\)
  in~\(\widehat{A}\)
  has empty interior.  Equivalently, any open subset of
  \(\widehat{\braket{X}{X}_A} \subseteq \widehat{A}\)
  contains~\([\pi]\) with \(\widehat{X}([\pi]) \neq [\pi]\).

  A Hilbert \(A\)\nb-bimodule~\(X\)
  is \emph{inner} if it is isomorphic to~\(A\)
  as a Hilbert bimodule.  It is \emph{universally weakly inner} if
  \(X^{**} \cong A^{**}\)
  as a Hilbert \(A^{**}\)\nb-bimodule.
  It is \emph{partly inner} if there is \(0\neq I\in\I(A)\)
  so that \(X\cdot I\)
  is isomorphic to~\(I\) as a Hilbert \(A,I\)\nb-bimodule.
  It is \emph{partly universally weakly inner} if there is
  \(I\in\I(A)\)
  so that \(X^{**}\cdot I^{**} \cong I^{**}\) as a Hilbert
  \(A^{**},I^{**}\)\nb-bimodule.
  In both cases, \(I\) is \(X\)\nb-invariant automatically.
  We call~\(X\) \emph{purely outer} or \emph{purely universally
    weakly outer} if~\(X\) is not partly inner or not partly
  universally weakly inner, respectively.
\end{definition}


\begin{lemma}
  \label{lem:outer_Hilbert_bimodule}
  An automorphism \(\alpha\in\Aut(A)\)
  satisfies Kishimoto's condition if and only if~\(A_\alpha\)
  does; and it is topologically non-trivial, inner, partly inner,
  or partly universally weakly inner, respectively, if and only
  if~\(A_\alpha\) is so.
\end{lemma}

\begin{proof}
  The assertion for Kishimoto's condition is trivial.  The assertion
  for topological non-triviality follows from
  Example~\ref{exa:widehat_partial_auto}.  Let
  \(0\neq I\in \I^\alpha(A)\).
  Then \(I A_\alpha = A_\alpha I = I_{\alpha|_I}\).
  An isomorphism of Hilbert bimodules \(I \congto I_{\alpha|_I}\)
  is, in particular, a unitary operator of left Hilbert
  \(I\)\nb-modules.
  So it is of the form \(x\mapsto x\cdot u\)
  for a unitary~\(u\) in \(\Bound(I) = \Mult(I)\).
  This map is an isomorphism of right Hilbert modules as well if and
  only if \(x\cdot u \cdot \alpha(y) = x \cdot y\cdot u\)
  for all \(x,y\in I\).
  Equivalently, \(\alpha|_I = \Ad_{u^*}\).
  Thus~\(\alpha\)
  is partly inner if and only if~\(A_\alpha\)
  is.  The case \(A=I\)
  shows that~\(\alpha\)
  is inner if and only if~\(A_\alpha\)
  is.  The Hilbert \(A^{**}\)\nb-bimodule~\((A_\alpha)^{**}\)
  associated to~\(A_\alpha\)
  is equal to \((A^{**})_{\alpha^{**}}\).
  Hence the same argument shows that~\(\alpha\) is (partly)
  universally weakly inner if and only if~\(A_\alpha\) is.
\end{proof}

All the properties of Hilbert bimodules defined above translate
naturally to Fell bundles by considering them pointwise (fibrewise).
Aperiodicity for Fell bundles was introduced in \cite{Kwasniewski-Szymanski:Pure_infinite}*{Definition~4.1}.
 Topological freeness, which is not a pointwise condition, was considered for systems dual to saturated Fell bundles  in \cite{Kwasniewski-Szymanski:Ore}*{Corollary 6.5}, and for  general Fell bundles in  \cite{Abadie-Abadie:Ideals}*{Section 3}.
We are not aware of any source that uses other conditions in the general context of Fell bundles.
\begin{definition}
  \label{def:aperiodic}
  Let \(\mathcal{B}=(B_g)_{g\in G}\) be a Fell bundle over a
  discrete group~\(G\).  Let \(A\defeq B_e\) be its unit fibre.  We
  call~\(\mathcal{B}\) \emph{pointwise outer}, \emph{pointwise
    purely outer} or \emph{pointwise purely universally weakly
    outer} if, for each \(g\in G\setminus\{e\}\), the Hilbert
  \(A\)\nb-bimodule~\(B_g\) is outer, purely outer or purely
  universally weakly outer, respectively.  We call~\(\mathcal{B}\)
  \emph{aperiodic} if~\(B_g\) satisfies Kishimoto's condition for
  all \(g\in G\setminus\{e\}\).  We
  call~\(\mathcal{B}\) \emph{pointwise topologically non-trivial}
  if~\(B_g\) is topologically non-trivial for all \(g\in
  G\setminus\{e\}\), that is, for each \(g\in G\setminus\{e\}\) the
  subset
  \[
  F_g \defeq \{[\pi]\in \widehat{A} \mid
  \widehat{B_g}([\pi])=[\pi]\}
  \]
  in~\(\widehat{A}\)
  has empty interior; here
  \(\widehat{\mathcal{B}}=(\widehat{B_g})_{g\in G}\)
  is the dual partial action, see
  Lemma~\ref{lem:partial_action_on_irreps}.  We call~\(\mathcal{B}\)
  \emph{topologically free} if~\(\widehat{\mathcal{B}}\)
  is topologically free, that is, for any
  \(g_1,\dotsc,g_n \in G\setminus\{e\}\)
  the union \(F_{g_1}\cup F_{g_2}\cup \dotsb \cup F_{g_n}\)
  has empty interior in~\(\widehat{A}\)
  (see~\cite{Lebedev:Topologically_free}).  This only depends on the
  partial \(G\)\nb-action
  on~\(\widehat{A}\) induced by the Fell bundle.
\end{definition}

\begin{remark}
  \label{rem:twisted_properties}
  Let \(\mathcal{B}^{\omega}=(B_g)_{g\in G}\)
  be a deformation of \(\mathcal{B}=(B_g)_{g\in G}\)
  by a \(2\)\nb-cocycle
  \(\omega\colon G\times G\to \T\),
  see Example~\ref{ex:twisted_fell_bundles}.  Then the fibres~\(B_g\),
  \(g\in G\),
  of the bundles \(\mathcal{B}\)
  and~\(\mathcal{B}^\omega\)
  coincide not only as Banach spaces, but also as Hilbert bimodules
  over \(A\defeq B_e\).
  Hence~\(\mathcal{B}\)
  has any of the properties in Definition~\ref{def:aperiodic} if and
  only if~\(\mathcal{B}^\omega\) has this property.
 \end{remark}

\begin{proposition}
  \label{pro:Fell_Kishimoto_to_detect}
  Let \(\mathcal{B}=(B_g)_{g\in G}\) be a Fell bundle over a
  discrete group~\(G\).  If~\(\mathcal{B}\) is aperiodic or
  topologically free, then \(A\defeq B_e\) detects ideals
  in~\(\Cred(\mathcal{B})\).
\end{proposition}

\begin{proof}
  The two assertions are
  \cite{Kwasniewski-Szymanski:Pure_infinite}*{Corollary~4.3} and
  \cite{Abadie-Abadie:Ideals}*{Corollary~3.4}.
\end{proof}

\begin{definition}
  \label{def:Cuntz_preorder}
  For \(a, b\in B^+\),
  we write \(a \precsim b\)
  and say that \emph{\(a\)~supports~\(b\)}
  if there is a sequence~\((x_k)_{k\in\N}\)
  in~\(B\)
  with \(\lim x^*_k b x_k = a\)
  (see \cite{Cuntz:Dimension_functions}).
  We call \(a, b\in A^+\)
  \emph{Cuntz equivalent} if \(a\preceq b\)
  and \(b\preceq a\).
  Let \(A\subseteq B\)
  be a \(\Cst\)\nb-subalgebra.
  We say that~\(A\)
  \emph{supports}~\(B\)
  if for each \(b\in B^+\setminus\{0\}\)
  there is \(a\in A^+\setminus\{0\}\)
  with \(a\precsim b\).  We say that~\(A\)
  \emph{residually supports}~\(B\)
  if for each ideal~\(I\)
  in~\(B\)
  the image of~\(A\)
  in the quotient~\(B/I\)
  supports~\(B/I\)
  (see \cite{Kwasniewski:Crossed_products}*{Definition~2.39}).
\end{definition}

\begin{lemma}
  \label{lem:support_implies_detect}
  If~\(A\) supports~\(B\), then~\(A\) detects ideals in~\(B\).
  If~\(A\) residually supports~\(B\), then~\(A\) separates
  ideals in~\(B\).
\end{lemma}

\begin{proof}
  A \(\Cst\)\nb-subalgebra~\(A\) supports~\(B\) if and only if for
  each \(b\in B^+\setminus\{0\}\) there is \(z\in B\) such that \(0
  \neq zbz^*\in A\), see
  \cite{Pasnicu-Phillips:Spectrally_free}*{Proposition~3.9}.  This
  implies the first assertion.  The second one follows from this and
  Remark~\ref{rem:separates_is_residual_detects}.
\end{proof}

\begin{proposition}[\cite{Kwasniewski-Szymanski:Pure_infinite}*{Corollary~4.4}]
  \label{the:general_uniqueness}
  If~\(\mathcal{B}\)
  is aperiodic, then~\(A\defeq B_e\)
  supports ideals in~\(\Cred(\mathcal{B})\).
\end{proposition}

We define residual versions of the non-triviality conditions for
Fell bundles:

\begin{definition}
  \label{def:residual}
  A Fell bundle \(\mathcal{B} = (B_g)_{g\in G}\) with unit fibre
  \(A\defeq B_e\) is \emph{residually pointwise purely outer},
  \emph{residually pointwise purely universally weakly outer},
  \emph{residually aperiodic}, \emph{residually pointwise
    topologically non-trivial} or \emph{residually topologically
    free} if, for each \(\mathcal{B}\)\nb-invariant ideal \(I\idealin
  A\), the restriction \(\mathcal{B}|_{A/I}\) is pointwise purely
  outer, pointwise purely universally weakly outer,
  aperiodic, pointwise topologically non-trivial or topologically
  free, respectively (compare
  \cite{Kwasniewski-Szymanski:Pure_infinite}*{Definition~4.7} and
  \cite{Giordano-Sierakowski:Purely_infinite}*{Definition~3.4}).
\end{definition}

\begin{remark}
  By Remark~\ref{rem:restrictions_of_duals_to_Fell_bundles}, a Fell
  bundle~\(\mathcal{B}\) is residually topologically free if and
  only if each restriction of the dual partial
  action~\(\widehat{\mathcal{B}}\) to a closed invariant subset is
  topologically free.  Moreover, if~\(G\) is Abelian
  and~\(\widehat{A}\) is Hausdorff, then both residual topological
  freeness and residual pointwise topological non-triviality
  coincide with freeness, compare the proof
  of~\cite{Pasnicu-Phillips:Spectrally_free}*{Proposition~1.10}.
\end{remark}

\begin{remark}
  \label{rem:twisted_properties2}
  Let \(\mathcal{B}^{\omega}=(B_g)_{g\in G}\)
  be a deformation of \(\mathcal{B}=(B_g)_{g\in G}\)
  by a \(2\)\nb-cocycle~\(\omega\),
  see Example~\ref{ex:twisted_fell_bundles}.  Then
  \(\I^{\mathcal{B}}(A)=\I^{\mathcal{B}^{\omega}}(A)\)
  and for any \(I\in \I^{\mathcal{B}}(A)\) and \(g\in G\)
  we have \(B_g|_{A/I} = B_g^{\omega}|_{A/I}\),
  see Remark~\ref{rem:twisted_properties}.  Hence~\(\mathcal{B}\)
  has any of the properties in Definition~\ref{def:residual} if and
  only if~\(\mathcal{B}^{\omega}\) has this property.
\end{remark}

We add one more separation condition introduced recently by Kirchberg
and Sierakowski in~\cite{Kirchberg-Sierakowski:Filling_families} to
detect strong pure infiniteness.

\begin{definition}[\cite{Kirchberg-Sierakowski:Filling_families}*{Definition~4.2}]
  \label{def:filling_family}
  Let~\(B\)
  be a \(\Cst\)\nb-algebra.
  A subset \(\mathcal{F}\subseteq B^+\)
  is a \emph{filling family for}~\(B\)
  if, for each \(D \in \Her(B)\)
  and each \(I\in \I(B)\)
  with \(D\not\subseteq I\),
  there is \(z\in B\setminus I\)
  such that \(z^*z\in D\) and \(zz^* \in \mathcal{F}\).
\end{definition}


It suffices to consider only primitive ideals~\(I\)
in the above definition.

\begin{lemma}
  \label{lem:filling_implies_residually_support}
  If~\(A\) is a \(\Cst\)\nb-subalgebra of~\(B\) and~\(A^+\)
  is a filling family for~\(B\), then~\(A\) residually supports~\(B\).
\end{lemma}

\begin{proof}
  Let \(I\in \I(B)\)
  and let \(q\colon B\to B/I\)
  be the quotient map.  Let \(b\in q(B)^+\setminus \{0\}\).
  There is \(d\in B^+\setminus I\)
  with \(q(d)=b\).
  Let \(D\defeq \overline{d B d}\).
  Then \(q(D)=\overline{bq(B)b}\)
  and \(D\not\subseteq I\).
  Hence there is \(z\in B\setminus I\)
  with \(z^*z\in D\)
  and \(zz^* \in A^+\).
  Thus \(q(z)\neq0\)
  in~\(B/I\),
  \(q(z)^*q(z)\in \overline{bq(B)b}\)
  and \(q(z)q(z)^*\in q(A)^+\setminus\{0\}\).
  Hence~\(q(A)\)
  supports~\(q(B)\)
  by \cite{Pasnicu-Phillips:Spectrally_free}*{Proposition~3.9}.
\end{proof}

In a future work, we shall prove the converse implication in the
above lemma.

\begin{theorem}
  \label{the:general_separation_Fell_bundles}
  Let \(\mathcal{B}=(B_g)_{g\in G}\) be an exact Fell bundle over a
  discrete group~\(G\) and let \(A\defeq B_e\).  Consider the
  following conditions:
  \begin{enumerate}[wide,label=\textup{(\ref*{the:general_separation_Fell_bundles}.\arabic*)}]
  \item \label{en:general_separation_of_ideals}%
    \(A\) separates ideals of~\(\Cred(\mathcal{B})\);
  \item \label{en:general_Fourier}%
    each ideal in~\(\mathcal{B}\)
    is of the form \(\Cred(\mathcal{B}|_I)\)
    for a \(\mathcal{B}\)\nb-invariant ideal \(I\idealin A\);
  \item \label{en:general_residual_support}%
    \(A\) residually supports~\(\Cred(\mathcal{B})\);
  \item \label{en:general_filling}%
    \(A^+\)~is a filling family for \(\Cred(\mathcal{B})\);
  \item \label{en:general_separation2}%
    \(\mathcal{B}\) is residually aperiodic;
  \item \label{en:general_separation3}%
    \(\mathcal{B}\) is residually topologically free.
  \end{enumerate}
  Then 
  \ref{en:general_separation3}\(\Rightarrow
  \)\ref{en:general_separation_of_ideals}\(\Leftrightarrow
  \)\ref{en:general_Fourier}\(\Leftarrow
  \)\ref{en:general_residual_support}\(\Leftarrow
  \)\ref{en:general_filling}\(\Leftarrow
  \)\ref{en:general_separation2}.
\end{theorem}

\begin{proof}
  The equivalence between
  \ref{en:general_separation_of_ideals} and~\ref{en:general_Fourier}
  follows from
  \cite{Kwasniewski-Szymanski:Pure_infinite}*{Theorem~3.12}, compare
  also Propositions \ref{pro:gauge-invariance_vs_separation}
  and~\ref{pro:separates_vs_exactness_B-residual_detect}.  That
  \ref{en:general_separation3} implies
  \ref{en:general_separation_of_ideals} follows from
  \cite{Abadie-Abadie:Ideals}*{Corollary~3.20}, from
  \cite{Kwasniewski-Szymanski:Pure_infinite}*{Corollary~3.23}, or by
  combining
  Propositions \ref{pro:Fell_Kishimoto_to_detect}
  and~\ref{pro:separates_vs_exactness_B-residual_detect}.
  Condition~\ref{en:general_residual_support}
  implies~\ref{en:general_separation_of_ideals} by
  Lemma~\ref{lem:support_implies_detect}.  The implication
  \ref{en:general_filling}\(\Rightarrow
  \)\ref{en:general_residual_support} is
  Lemma~\ref{lem:filling_implies_residually_support}.  Thus we must only
  prove that~\ref{en:general_separation2}
  implies~\ref{en:general_filling}.  We mimic the proof of
  \cite{Kirchberg-Sierakowski:Strong_pure}*{Theorem~3.8}.

  Assume that~\(\mathcal{B}\)
  is residually aperiodic.  Pick \(D\in\Her(\Cred(\mathcal{B}))\)
  and \(J\in\I(\Cred(\mathcal{B}))\)
  with \(D\not\subseteq J\).
  We need \(z\in A\)
  with \(z^*z\in D\)
  and \(zz^* \in A^+ \setminus J\).
  Propositions \ref{pro:Fell_Kishimoto_to_detect}
  and~\ref{pro:separates_vs_exactness_B-residual_detect} show
  that \(A\subseteq\Cred(\mathcal{B})\)
  separates ideals.  Hence
  \(\I(\Cred(\mathcal{B}))\cong \I^{\mathcal{B}}(A)\)
  by Proposition~\ref{pro:gauge-invariance_vs_separation}.
  Since~\(\mathcal{B}\)
  is exact, we may identify \(\Cred(\mathcal{B})/J\)
  with \(\Cred(\mathcal{B}|_{A/I})\),
  where \(I\defeq A\cap J\).
  In particular, there is a faithful conditional expectation
  \(E\colon \Cred(\mathcal{B}) /J \to A/I\).
  Let \(q\colon \Cred(\mathcal{B})\to \Cred(\mathcal{B})/J\)
  be the quotient map.  Let \(d\in D^+\setminus J\).
  Define \(b\defeq q(d)\)
  and \(\varepsilon\defeq \frac{1}{4}\norm{E(b)}>0\).
  Since~\(\mathcal{B}|_{A/I}\)
  is aperiodic, \cite{Kwasniewski-Szymanski:Pure_infinite}*{Lemma~4.2}
  gives \(x\in (A/I)^+\) satisfying
  \[
  \norm{x}=1,\qquad
  \norm{xbx-xE(b)x}<\varepsilon, \qquad
  \norm{xE(b)x}> \norm{E(b)}-\varepsilon = 3\varepsilon.
  \]
  Now
  \cite{Kirchberg-Rordam:Infinite_absorbing}*{Lemma~2.2} gives a
  contraction \(y\in \Cred(\mathcal{B}) /J\) with
  \[
  y^*(xbx)y = (xE(b)x-\varepsilon)_+ \in (A/I)^+.
  \]
  Moreover, \(y^*xbxy \neq 0\) because
  \[
  \norm{(xE(b)x-\varepsilon)_+}\ge
  \norm{xE(b)x}-\varepsilon > 2\varepsilon > 0.
  \]
  There are \(c\in A^+\)
  and a contraction \(w\in \Cred(\mathcal{B})\)
  with \(q(c)=(xE(b)x-\varepsilon)_+\)
  and \(q(w)=xy\).
  Then \(q(c)=y^*xbxy= q(w^*dw)\).
  So \(c=w^*dw+ v\) for some \(v\in J\).

  Since an approximate unit in~\(I\)
  is also one for~\(J\),
  there is a contraction \(f\in I^+\)
  with \(\norm{v-fv}<\varepsilon\).
  Let~\(1\)
  denote the formal unit in the unitisations of \(A\)
  or~\(\Cred(\mathcal{B})\)
  and let \(g\defeq 1-f\in A^+\).  Then \(\norm{g}\le 1\) and
  \[
  \norm{gw^*dwg - gcg} = \norm{g v g} \le  \norm{ v-fv } < \varepsilon.
  \]
  Now \cite{Kirchberg-Rordam:Infinite_absorbing}*{Lemma~2.2} gives a
  contraction \(h\in \Cred(\mathcal{B})\) with
  \[
  h^* (gw^*dwg) h = (gcg-\varepsilon)_+ \in  A^+.
  \]
  Let \(z\defeq (d^{\nicefrac12}wgh)^*\).  Then \(z^*z\in D\) and \(z z^*=
  (gcg-\varepsilon)_+ \in A^+\).  Moreover, since \(q(gcg)=q(c+
  fcf-cf-fc)=q(c)\), we get
  \begin{align*}
    \norm{q(zz^*)}
    &=
    \norm{q((gcg)-\varepsilon)_+)}=(\norm{q(gcg)}-\varepsilon)_+=(\norm{q(c)}-\varepsilon)_+\\
    &=
    (\norm{(xE(b)x-\varepsilon)_+}-\varepsilon)_+
    = \norm{xE(b)x}
    - 2 \varepsilon > \varepsilon > 0.
  \end{align*}
  Hence \(zz^*\notin J\).
\end{proof}

\section{Strong pure infiniteness of reduced section
  \texorpdfstring{$\Cst$}{C*}-algebras}
\label{sec:strong_pi}

We generalise the main results
from~\cite{Kirchberg-Sierakowski:Strong_pure} to Fell bundles and
slightly improve the main result
of~\cite{Kwasniewski-Szymanski:Pure_infinite}.  This shows why
residually supporting and filling families are important.

If \(a,b \in A\) are elements of a \(\Cst\)\nb-algebra~\(A\) and
\(\varepsilon >0\), we write \(a\approx_\varepsilon b\) if
\(\norm{a-b}<\varepsilon\).  Infinite and properly infinite
elements in~\(A^+\) are defined
in~\cite{Kirchberg-Rordam:Non-simple_pi}.  We recall their
equivalent description in
\cite{Kwasniewski-Szymanski:Pure_infinite}*{Lemma~2.1}.
We also introduce the notion of separated pairs of elements
in~\(A^+\).

\begin{definition}
  \label{def:kinds_of_elements}
  Let~\(A\) be a \(\Cst\)\nb-algebra and let \(a\in A^+\setminus\{0\}\).
  \begin{enumerate}
  \item We call \(a\in A^+\) \emph{infinite} in~\(A\) if there is
    \(b\in A^+\setminus\{0\}\) such that for all \(\varepsilon >0\)
    there are \(x,y\in a A\) with \(x^*x\approx_\varepsilon a\),
    \(y^*y\approx_\varepsilon b\) and \(x^*y\approx_\varepsilon 0\).
  \item We call \(a^+\setminus\{0\}\) \emph{properly infinite} if
    for all \(\varepsilon >0\) there are \(x,y\in a A\) with
    \(x^*x\approx_\varepsilon a\), \(y^*y\approx_\varepsilon a\) and
    \(x^*y\approx_\varepsilon 0\).
  \item We call \(a,b\in A^+\) \emph{separated in~\(A\)} if for all
    \(\varepsilon >0\) there are \(x\in a A\) and \(y\in bA\) with
    \(x^*x\approx_\varepsilon a\), \(y^*y\approx_\varepsilon b\) and
    \(x^*y\approx_\varepsilon 0\).
  \end{enumerate}
\end{definition}

Any properly infinite element is infinite.  An element \(a\in
A^+\setminus\{0\}\) is infinite if and only if it is separated from
some other element \(b\in A^+\setminus\{0\}\), and properly infinite
if and only if it is separated from itself.  Moreover, \(a\in
A^+\setminus\{0\}\) is properly infinite if and only if \(a+I\)
in~\(A/I\) is either zero or infinite for each ideal~\(I\) in~\(A\),
see \cite{Kirchberg-Rordam:Non-simple_pi}*{Proposition~3.14}.  Thus
proper infiniteness is residual infiniteness.  For the above
approximate equalities, we may assume \(x=\sqrt{a} d_1\) and
\(y=\sqrt{b} d_2\) for some \(d_1, d_2 \in A\); thus we may
reformulate the condition of being separated as follows: for all
\(\varepsilon >0\) there are \(d_1, d_2 \in A\) with \(d_1^*a d_1
\approx_\varepsilon a\), \(d_2^* b d_2\approx_\varepsilon b\) and
\(d_1^*a^{\nicefrac12} b^{\nicefrac12} d_2\approx_\varepsilon 0\).

The following definition uses characterizations of purely infinite
and strongly purely infinite \(\Cst\)\nb-algebras in
\cite{Kirchberg-Rordam:Non-simple_pi}*{Theorem~4.16} and
\cite{Kirchberg-Rordam:Infinite_absorbing}*{Remark~5.10}.

\begin{definition}
  A \(\Cst\)\nb-algebra~\(A\)
  is \emph{purely infinite} if each element
  \(a\in A^+\setminus\{0\}\)
  is properly infinite, and \emph{strongly purely infinite} if each
  pair of elements \(a,b\in A^+\setminus\{0\}\) is separated in~\(A\).
\end{definition}

We present two general criteria that allow to prove that a
\(\Cst\)\nb-algebra~\(B\)
is strongly purely infinite by analysing a \(\Cst\)\nb-subalgebra
\(A\subseteq B\).
The first one, established
in~\cite{Kirchberg-Sierakowski:Filling_families}, requires no
structural assumptions on the \(\Cst\)\nb-algebras,
but the conditions on elements are quite strong.

\begin{proposition}
  \label{pro:Kirchberg_Sierakowski}
  Let~\(A\)
  be a \(\Cst\)\nb-subalgebra
  of~\(B\)
  such that~\(A^+\)
  is a filling family for~\(B\).
  Then~\(B\)
  is strongly purely infinite if and only if each pair of elements
  \(a,b\in A^+\)
  has the matrix diagonalisation property in~\(B\),
  that is, for each \(x\in B\)
  with
  \(\left(\begin{smallmatrix} a&x^*\\x&b \end{smallmatrix}\right) \in
  M_2(B)^+ \)
  and each \(\varepsilon>0\)
  there are \(d_1\in B\) and \(d_2\in B\) such that
  \[
  d_1^*a d_1 \approx_\varepsilon a,\qquad
  d_2^* b d_2\approx_\varepsilon b,\qquad
  d_1^* x d_2\approx_\varepsilon 0.
  \]
\end{proposition}

\begin{proof}
  The subset~\(A^+\)
  is closed under \(\varepsilon\)\nb-cut-downs,
  so that \cite{Kirchberg-Sierakowski:Filling_families}*{Lemma~5.4}
  applies here.  Therefore, it suffices to check the
  ``matrix diagonalisation property'' needed
  in~\cite{Kirchberg-Sierakowski:Filling_families}*{Theorem~1.1} for
  pairs of elements.  Hence the claim follows
  from~\cite{Kirchberg-Sierakowski:Filling_families}*{Theorem~1.1}.
\end{proof}

The matrix diagonalisation described in the above proposition is a
strong version of the separation of elements introduced in
Definition~\ref{def:kinds_of_elements}.  It is hard to check it in
practice.  However, if ``the ratio of ideals in~\(B\)
to projections in~\(A\)
is not large'' we may characterise when~\(B\)
is strongly purely infinite using much weaker conditions.  The ideal
property is one of the conditions that guarantee that the ideals and
projections in a \(\Cst\)\nb-algebra
are well-balanced.  We recall that a \(\Cst\)\nb-algebra~\(A\)
has \emph{the ideal property} if the set of
projections in~\(A\)
separates ideals in~\(A\).
A \(\Cst\)\nb-algebra
with the ideal property is purely infinite if and only if it is
strongly purely infinite by
\cite{Pasnicu-Rordam:Purely_infinite_rr0}*{Proposition~2.14}.

The proof of the following proposition uses known arguments, see, for
instance, the proofs of
\cite{Kwasniewski:Crossed_products}*{Proposition~2.46} and
\cite{Kwasniewski-Szymanski:Pure_infinite}*{Theorem 4.10}.

\begin{proposition}
  \label{pro:pure_infiniteness}
  Assume that \(A\subseteq B\)
  residually supports~\(B\).
  Let \(\I^B(A)\defeq \{J\cap A\mid J\triangleleft B\}\)
  and assume that~\(\I^B(A)\)
  is finite or that projections in~\(A\)
  separate the ideals in~\(\I^B(A)\).
  Then~\(B\)
  is purely infinite if and only if each element in
  \(A^+\setminus \{0\}\)
  is properly infinite in~\(B\).
  Moreover, if~\(B\)
  is purely infinite, then it has the ideal property, and so~\(B\)
  is strongly purely infinite.
 \end{proposition}

\begin{proof}
  We may assume that each element in \(A^+\setminus \{0\}\) is
  properly infinite in~\(B\).  It suffices to prove that~\(B\) is
  purely infinite and has the ideal property.  Then it is strongly
  purely infinite by
  \cite{Pasnicu-Rordam:Purely_infinite_rr0}*{Proposition~2.14}.
  By
  Lemma~\ref{lem:support_implies_detect}, \(A\)~separates
  ideals in~\(B\).
  In particular, restriction to~\(A\)
  is a lattice isomorphism \(\I(B) \congto \I^B(A)\).

  Suppose first that \(\I^B(A)\cong \I(B)\)
  is finite.  Then there is a chain
  \(0 = J_0\subsetneq J_1\subsetneq J_2 \subsetneq \dotsb\subsetneq
  J_n=B\)
  of ideals in~\(B\)
  such that \(J_i/J_{i-1}\)
  is simple for \(i=1,\dotsc,n\).
  Fix~\(i\)
  and let \(q\colon B\to B/J_{i-1}\)
  be the quotient map.  For any \(b\in q(J_i)^+\setminus\{0\}\)
  there is \(a\in q(A)^+\setminus\{0\}\)
  such that \(a\precsim b\)
  in~\(B/J_{i-1}\).
  Thus \(a\in q(J_i\cap A)^+\setminus\{0\}\).
  Since~\(a\)
  is the image under~\(q\)
  of a properly infinite element in~\(A^+\),
  it is properly infinite.  Since~\(q(J_i)\)
  is simple, \(b\precsim a\)
  by \cite{Kirchberg-Rordam:Non-simple_pi}*{Proposition 3.5(ii)}.
  Hence~\(b\)
  is Cuntz equivalent to~\(a\).
  Thus~\(b\)
  is properly infinite in~\(q(J_i)\).
  It follows that~\(q(J_i)\)
  is purely infinite.  Since pure infiniteness is closed under
  extensions by \cite{Kirchberg-Rordam:Non-simple_pi}*{Theorem 4.19},
  \(B\)
  is purely infinite.  Since~\(B\)
  has finite ideal structure, \(B\)
  has the ideal property by
  \cite{Kwasniewski-Szymanski:Pure_infinite}*{Lemma~4.9}.

  Now suppose that projections in~\(A\)
  separate the ideals in \(\I^B(A)\cong \I(B)\).
  Then~\(B\)
  has the ideal property.  Let \(b\in B^+\setminus\{0\}\)
  and let~\(J\)
  be an ideal in~\(B\)
  with \(b\notin J\).
  By \cite{Kirchberg-Rordam:Non-simple_pi}*{Proposition~3.14}, it
  suffices to show that the image of~\(b\)
  under the quotient map \(q\colon B\to B/J\)
  is infinite in~\(q(B)\).
  We may find \(a\in A^+\setminus J\)
  such that \(q(a)\precsim q(b)\).
  By \cite{Kirchberg-Rordam:Non-simple_pi}*{Proposition 3.14},
  \(q(a)\)
  is properly infinite in~\(q(B)\).
  By our assumption, we may find a non-zero projection~\(p\)
  that belongs to \(A\cap \overline{BaB}\)
  but not to \(I\defeq A\cap J\).
  Since~\(q(a)\)
  is properly infinite and~\(q(p)\)
  belongs to the ideal \(\overline{q(B)q(a)q(B)}\),
  we get \(q(p) \precsim q(a)\)
  by \cite{Kirchberg-Rordam:Non-simple_pi}*{Proposition 3.5(ii)}.
  Hence \(q(p) \precsim q(b)\).
  Since~\(q(p)\)
  is a properly infinite projection, \(b\)
  is infinite by \cite{Kirchberg-Rordam:Non-simple_pi}*{Lemma~3.17},
  see also \cite{Kirchberg-Rordam:Non-simple_pi}*{Lemma 3.12(iv)}.
\end{proof}

The analogues of (properly and residually) infinite elements for Fell
bundles are defined in
\cite{Kwasniewski-Szymanski:Pure_infinite}*{Definition~5.1}.  We add
strongly separated pairs to this list:

\begin{definition}
  \label{def:infinite_elements_B}
  Let \(\mathcal{B}=\{B_g\}_{g\in G}\)
  be a Fell bundle and let \(A\defeq B_e\).
  \begin{enumerate}
  \item \label{en:infinite_elements_B1}%
    We call \(a\in A^+\setminus\{0\}\)
    \emph{\(\mathcal{B}\)\nb-infinite}
    if there is \(b\in A^+\setminus\{0\}\)
    such that for each \(\varepsilon >0\)
    there are \(n,m\in\N\)
    and \(t_i\in G\),
    \(a_i\in a B_{t_i}\) for \(i=1,\dotsc,n+m\) such that
    \[
    a \approx_\varepsilon \sum_{i=1}^n a_i^*a_i, \quad
    b\approx_\varepsilon\sum_{i=n+1}^{n+m}a_i^*a_i
    \quad \text{and} \quad
    a_i^* a_j\approx_{\varepsilon/\max\{n^2, m^2\}} 0
    \quad \text{for}\ i\neq j.
    \]
    We call \(a\in A^+\setminus\{0\}\)
    \emph{residually \(\mathcal{B}\)\nb-infinite}
    if \(a+I\)
    is \(\mathcal{B}|_{A/I}\)\nb-infinite
    in \(A/I\)
    for all \(I\in \I^{\mathcal{B}}(A)\) with \(a\notin I\).

  \item \label{en:infinite_elements_B2}%
    We call \(a\in A^+\setminus\{0\}\)
    \emph{properly \(\mathcal{B}\)\nb-infinite}
    (or \(\mathcal{B}\)\nb-paradoxical)
    if for each \(\varepsilon >0\)
    there are \(n,m\in\N\)
    and \(a_i\in a B_{t_i}\),
    \(t_i\in G\) for \(i=1,\dotsc,n+m\) such that
    \[
    a \approx_\varepsilon \sum_{i=1}^n a_i^*a_i,\quad
    a\approx_\varepsilon\sum_{i=n+1}^{n+m}a_i^*a_i
    \quad \text{and}\quad
    a_i^* a_j\approx_{\varepsilon/\max\{n^2, m^2\}} 0
    \quad\text{for}\ i\neq j.
    \]
  \item \label{en:infinite_elements_B3}%
    We call \(a,b\in A^+\setminus\{0\}\)
    \emph{strongly \(\mathcal{B}\)\nb-separated}
    if for each \(c\in B_g\), \(g\in G\),
    and each \(\varepsilon >0\)
    there are \(n,m\in\N\)
    and \(s_i\in G\),
    \(a_i\in a B_{s_i}\)
    for \(i=1,\dotsc,n\)
    and \(t_j\in G\),
    \(b_j\in bB_{t_j}\) for \(j=1,\dotsc,m\) such that
    \[
    a \approx_\varepsilon \sum_{i=1}^n a_i^*a_i,\qquad
    b\approx_\varepsilon\sum_{j=1}^mb_i^*b_i
    \quad \text{and} \quad
    a_i^* c b_j\approx_{\varepsilon/nm} 0
    \quad\text{for all}\ i, j
    \]
    and \(a_i^* a_j\approx_{\varepsilon/n^2} 0\)
    and \(b_i^* b_j\approx_{\varepsilon/m^2} 0\)
    for all \(i, j\) with \(i\neq j\).
  \end{enumerate}
\end{definition}

\begin{remark}
  Let \(\alpha\colon G\to \Aut(A)\)
  be a group action and let \(\mathcal{B}\defeq \mathcal{A}_\alpha\)
  be the associated Fell bundle.  By
  \cite{Kirchberg-Sierakowski:Strong_pure}*{Remark~5.4},
  the action~\(\alpha\)
  is \(G\)\nb-separating
  in the sense of
  \cite{Kirchberg-Sierakowski:Strong_pure}*{Definition~5.1}
  if and only each pair \(a,b\in A^+\setminus\{0\}\)
  is strongly \(\mathcal{B}\)\nb-separated
  with \(n=m=1\).
  Allowing arbitrary \(n\)
  and~\(m\)
  gives a weaker condition.
\end{remark}

\begin{lemma}
  \label{lem:from separiation to diagonalisation}
  Let \(\mathcal{B}=\{B_g\}_{g\in G}\)
  be a Fell bundle.  Put \(A\defeq B_e\)
  and let \(B=\overline{\bigoplus_{g\in G} B_g}\)
  be any \(\Cst\)\nb-completion.
  If each pair of elements \(a,b\in A^+\setminus\{0\}\)
  is strongly \(\mathcal{B}\)\nb-separated,
  then each pair of elements \(a,b\in A^+\setminus\{0\}\)
  has the matrix diagonalisation in~\(B\)
  described in Proposition~\textup{\ref{pro:Kirchberg_Sierakowski}}.
\end{lemma}

\begin{proof}
  Let \(\CC\defeq \bigcup_{g\in G} B_g\)
  and \(\mathcal{S}\defeq \bigoplus_{g\in G} B_g\).
  Since each pair of elements in \(A^+\setminus\{0\}\)
  is strongly \(\mathcal{B}\)\nb-separated,
  each pair of elements \(a, b\in A^+\setminus\{0\}\)
  has the matrix diagonalisation property with respect to \(\CC\)
  and~\(\mathcal{S}\)
  as introduced in
  \cite{Kirchberg-Sierakowski:Filling_families}*{Definition~4.6}.
  Indeed, let \(x\in B_g\)
  be such that
  \(\left(\begin{smallmatrix} a&x^*\\x&b \end{smallmatrix}\right) \in
  M_2(B)^+\)
  and let \(\varepsilon >0\).
  Let \(a_i\in a B_{s_i}\)
  and \(b_j\in b B_{t_j}\)
  satisfy the conditions described in
  Definition~\ref{def:infinite_elements_B}.\ref{en:infinite_elements_B3}
  with \(c\defeq a^{\nicefrac12}x b^{\nicefrac12}\).
  We may assume \(a_i=a^{\nicefrac12}x_i\)
  and \(b_j=b^{\nicefrac12}y_j\) for some \(x_i\), \(y_j\).  Let
  \[
  d_1\defeq \sum_{i=1}^n x_i,\qquad
  d_2\defeq \sum_{j=1}^n y_j.
  \]
  The estimates in
  Definition~\ref{def:infinite_elements_B}.\ref{en:infinite_elements_B3},
  imply
  \[
  d_1^*a d_1 \approx_{2\varepsilon} a,\qquad
  d_2^* b d_2\approx_{2\varepsilon} b,\qquad
  d_1^* x d_2\approx_\varepsilon 0.
  \]
  This proves our claim.

  Clearly, \(\mathcal{S}\)
  is a multiplicative subsemigroup of~\(B\),
  \(\mathcal{S}^* \mathcal{C}\mathcal{S}\subseteq \mathcal{C}\),
  \(A \mathcal{S} A\subseteq \mathcal{S}\),
  and \(\clsp\{\mathcal{C}\}=B\).
  Thus \cite{Kirchberg-Sierakowski:Filling_families}*{Lemma~5.6}
  implies that each pair of elements \(a, b\in A^+\setminus\{0\}\)
  has the matrix diagonalisation property in~\(B\).
\end{proof}

\begin{theorem}
  \label{the:pure_infiniteness_paradoxical_Fell_bundles}
  Let \(\mathcal{B}=\{B_g\}_{g\in G}\)
  be a residually aperiodic, exact Fell bundle.  Let \(A\defeq B_e\).
  Then \(\Cred(\mathcal{B})\)
  is strongly purely infinite if one of the following conditions
  holds:
  \begin{enumerate}[wide,label=\textup{(\ref*{the:pure_infiniteness_paradoxical_Fell_bundles}.\arabic*)}]
  \item\label{en:separation_of_pairs}%
    each pair of elements in \(A^+\setminus\{0\}\)
    is strongly \(\mathcal{B}\)\nb-separated;
  \item\label{en:finite_plus_infinity}%
    \(\I^{\mathcal{B}}(A)\)
    is finite and each element in \(A^+\setminus\{0\}\)
    is Cuntz equivalent in~\(\Cred(\mathcal{B})\) to a residually \(\mathcal{B}\)\nb-infinite
    element;
  \item \label{en:IP_plus_infinity}%
    projections in~\(A\)
    separate ideals in~\(\I^\mathcal{B}(A)\)
    and each element in \(A^+\setminus\{0\}\)
    is Cuntz equivalent in~\(\Cred(\mathcal{B})\) to a residually \(\mathcal{B}\)\nb-infinite
    element;
  \item\label{en:RR_zero_plus_infinity}%
    \(A\)~is
    of real rank zero and each non-zero projection in~\(A\)
    is Cuntz equivalent  in~\(\Cred(\mathcal{B})\) to one which is
    residually \(\mathcal{B}\)\nb-infinite.
  \end{enumerate}
\end{theorem}

\begin{proof}
  By Theorem~\ref{the:general_separation_Fell_bundles}, \(A^+\)
  is a filling family for~\(\Cred(\mathcal{B})\),
  and~\(A\)
  residually supports~\(\Cred(\mathcal{B})\).
  Therefore, \ref{en:separation_of_pairs} implies that
  \(\Cred(\mathcal{B})\)
  is strongly purely infinite by Lemma~\ref{lem:from separiation to
    diagonalisation} and Proposition~\ref{pro:Kirchberg_Sierakowski}.
  If we assume \ref{en:finite_plus_infinity}
  or~\ref{en:IP_plus_infinity}, then \(\Cred(\mathcal{B})\)
  is strongly purely infinite by
  Proposition~\ref{pro:pure_infiniteness}: the assumption that
  \(a\in A^+\setminus\{0\}\)
  is Cuntz equivalent to a residually \(\mathcal{B}\)\nb-infinite
  element implies that~\(a\)
  is properly infinite in~\(\Cred(\mathcal{B})\).
  By \cite{Kwasniewski:Crossed_products}*{Lemma~2.44},
  condition~\ref{en:RR_zero_plus_infinity} implies that each element
  in \(A^+\setminus\{0\}\)
  is properly infinite in \(\Cred(\mathcal{B})\).
  Hence~\(\Cred(\mathcal{B})\)
  is strongly purely infinite, again by
  Proposition~\ref{pro:pure_infiniteness}
\end{proof}

\begin{remark}
  \label{rem:twisted_properties3}
  Let \(\mathcal{B}^{\omega}=(B_g)_{g\in G}\)
  be a deformation of a Fell bundle \(\mathcal{B}=(B_g)_{g\in G}\)
  by a \(2\)\nb-cocycle~\(\omega\).
  For any \(b\in B_g\),
  we have \(b^{* \omega}\cdot_\omega b=b^*b\in A\defeq B_e\).
  Thus every property of \(a\in A^+\setminus\{0\}\)
  described in Definition~\ref{def:infinite_elements_B} holds
  in~\(\mathcal{B}\)
  if and only if it holds in~\(\mathcal{B}^{\omega}\).
  Accordingly, by Remark~\ref{rem:twisted_properties2}, if~\(G\)
  is exact or~\(\mathcal{B}\)
  is minimal, then the assumptions and conditions in
  Theorem~\ref{the:pure_infiniteness_paradoxical_Fell_bundles} hold
  for~\(\mathcal{B}\)
  if and only if they hold for~\(\mathcal{B}^{\omega}\).
  The pure infiniteness criteria in
  Theorem~\ref{the:pure_infiniteness_paradoxical_Fell_bundles} when
  applied to crossed products are weaker than those in
  \cites{Laca-Spielberg:Purely_infinite,
    Jolissaint-Robertson:Simple_purely_infinite,
    Rordam-Sierakowski:Purely_infinite,
    Giordano-Sierakowski:Purely_infinite,
    Kirchberg-Sierakowski:Strong_pure}, compare
  \cite{Kwasniewski-Szymanski:Pure_infinite}*{Corollary~5.14}.  In
  particular, the pure infiniteness criteria in the above sources
  remain valid also in the twisted case.  Moreover, twisted
  \(k\)\nb-graph \(\Cst\)\nb-algebras
  \cite{Kumjian-Pask-Sims:Homology}*{Definition 7.4} are modelled in
  a natural way by twisted Fell bundles, see
  \cite{Raeburn:Deformations}*{Corollary 4.9}.  Therefore,
  Theorem~\ref{the:pure_infiniteness_paradoxical_Fell_bundles} can
  be used to analyse when such algebras are purely infinite,
  compare
  \cite{Kwasniewski-Szymanski:Pure_infinite}*{Theorem~7.9}.
\end{remark}


\section{Morita restrictions and coverings of Hilbert bimodules}
\label{sec:Kish_permanence}

First, we prove that Kishimoto's condition is invariant under
Morita equivalence and preserved by restriction to possibly
non-invariant ideals.  We combine both processes in one step, which we
call Morita restriction.

\begin{proposition}
  \label{pro:Kish_Morita_restriction}
  Let~\(Y\)
  be a Hilbert \(B\)\nb-bimodule
  and~\(\E\)
  a Hilbert \(A,B\)\nb-bimodule
  that is full over~\(A\).
  Then \(X \defeq \E \otimes_B Y \otimes_B \E^*\)
  is a Hilbert \(A\)\nb-bimodule.
  If~\(Y\) satisfies Kishimoto's condition, then so does~\(X\).
\end{proposition}

\begin{proof}
Let
  \(x \defeq e \otimes_B y_0 \otimes_B f^* \in \E\otimes_B Y \otimes_B
  \E^*\)
  with \(e,f\in \E\),
  \(y_0\in Y\).  Elements of this form
  are linearly dense in~\(X\)
  by definition.  Hence, by Lemma~\ref{lem:kishimoto_subspace}, it suffices
  to check Kishimoto's condition for~\(x\)
  of this form.  Fix \(D\in \Her(A)\)
  and \(\varepsilon>0\).
  We are going to find \(a\in D^+\)
  with \(\norm{a}=1\) and \(\norm{a x a}<\varepsilon.\)

  Let \(\F\defeq D\cdot \E=\{d\cdot x\mid d\in D,\ x\in \E\}\).
  This is a Hilbert \(B\)\nb-submodule
  of~\(\E\).
  The left action of~\(A\)
  on~\(\E\)
  gives an isomorphism \(A\cong \Comp(\E)\)
  because~\(\E\)
  is full over~\(A\).
  We claim that this isomorphism maps~\(D\)
  onto~\(\Comp(\F)\).
  Since \(D=D A D\),
  it maps~\(D\)
  onto \(D \Comp(\E) D\),
  which is the closed linear span of
  \(d_1\ket{x}\bra{y}d_2^* = \ket{d_1 x}\bra{d_2 y}\)
  for \(d_1,d_2\in D\),
  \(x,y\in\E\).
  This is the same as the closed linear span of \(\ket{x}\bra{y}\)
  for \(x, y\in \F\),
  which is \(\Comp(\F)\), viewed as a subalgebra of~\(\Comp(\E)\).

  So \(\F\neq\{0\}\)
  and there is \(\eta\in \F\)
  with \(\norm{\eta}=1\).
  Lemma~\ref{lem:technical} for
  \(\abs{\eta} \defeq \sqrt{\langle\eta,\eta \rangle_B}\in B\)
  gives \(D_0\in \Her(\overline{\abs{\eta} B\abs{\eta}})\)
  with \(\norm{\abs{\eta} b} \ge (1-\varepsilon) \norm{b}\)
  for all \(b\in D_0\).  If \(b\in D_0\), then
  \begin{equation}
    \label{eq:lower_bound_eta_b}
    \norm{\eta b}^2
    = \norm{\braket{\eta b}{\eta b}_B}
    = \norm{b^* \abs{\eta}^2 b}
    = \norm{\abs{\eta} b}^2
    \ge (1-\varepsilon)^2\norm{b}^2.
  \end{equation}
  Put
  \[
  y\defeq \braket{\eta}{e}_B \cdot y_0\cdot \braket{f}{\eta}_B \in Y.
  \]
  Kishimoto's condition for~\(Y\)
  gives \(b\in D_0^+\)
  with \(\norm{b}=1\)
  and \(\norm{b y b}<\varepsilon (1-\varepsilon)^2\).
  Let \(b_0\defeq b/\norm{\eta b}\in D_0^+\),
  so that \(\norm{\eta b_0} = 1\).
  Equation~\eqref{eq:lower_bound_eta_b} implies
  \(\norm{b_0 y b_0}<\varepsilon\).
  The rank-one operator \(\ket{\eta b_0}\bra{\eta b_0}\)
  belongs to~\(\Comp(\F)^+\)
  and has norm
  \(\norm{\ket{\eta b_0}\bra{\eta b_0}} = \norm{\eta b_0}^2 = 1\).
  The isomorphism \(\Comp(\F)\cong D\)
  maps it to an element \(a\in D^+\) with \(\norm{a}=1\) and
  \begin{align*}
    \norm{a x a}
    &= \norm{ \ket{\eta b_0}\bra{\eta b_0} e \otimes y_0 \otimes f^*
      \ket{\eta b_0}\bra{\eta b_0}}\\
    &= \norm{\eta b_0\cdot\braket{\eta b_0}{e}_B \otimes y_0
      \otimes (\eta b_0\cdot \braket{\eta b_0}{f}_B)^*}\\
    &= \norm{\eta b_0\otimes b_0^*\braket{\eta}{e}_B\cdot y_0 \cdot
      \braket{f}{\eta}_B b_0 \otimes (\eta b_0)^*}\\
    &= \norm{\eta b_0\otimes b_0 y b_0 \otimes  (\eta b_0)^*}
    < \varepsilon.
  \end{align*}
  Hence \(x\in\Kish(X)\).
\end{proof}

\begin{definition}
  \label{def:Morita_restriction}
  In the situation of Proposition~\ref{pro:Kish_Morita_restriction},
  we call~\(X\)
  a \emph{Morita restriction} of~\(Y\).
  If, in addition, \(\E\)
  is an equivalence bimodule, we call \(X\)
  and~\(Y\) \emph{Morita equivalent}.
\end{definition}

\begin{corollary}
  \label{cor:Morita_kishimoto}
  If \(X\) and~\(Y\) are Morita equivalent Hilbert bimodules,
  then~\(X\) satisfies Kishimoto's condition if and only if~\(Y\)
  does.
\end{corollary}

Our Morita equivalences between Hilbert bimodules are the same as
in~\cite{Abadie-Eilers-Exel:Morita_bimodules}:

\begin{lemma}
  \label{lem:equivalence_for_Hilbert_bimodules}
  An equivalence \(A\)\nb-\(B\)-bimodule~\(\E\)
  witnesses an equivalence between \(X\)
  and~\(Y\) if and only if \(X\otimes_A \E \cong \E \otimes_B Y\).
\end{lemma}

\begin{proof}
  On the one hand, \(X \cong \E \otimes_B Y \otimes_B \E^*\) implies
  \[
  X\otimes_A \E \cong \E \otimes_B Y \otimes_B \E^* \otimes_A \E \cong
  \E \otimes_B Y \otimes_B B \cong \E \otimes_B Y
  \]
  by \eqref{eq:right_innprod} and~\eqref{eq:right_mult}.  On the other
  hand, \(X\otimes_A \E \cong \E \otimes_B Y\)
  and \eqref{eq:left_innprod} and~\eqref{eq:right_mult} imply
  \[
  \E \otimes_B Y \otimes_B \E^* \cong X\otimes_A \E \otimes_B \E^*
  \cong X\otimes_A A \cong X.\qedhere
  \]
\end{proof}

\begin{remark}
  \label{rem:restrict_Hilbert_bimodules}
  If~\(\E\)
  is only full over~\(A\),
  then it is an equivalence between~\(A\)
  and the ideal \(J\defeq \braket{\E}{\E}_B \idealin B\).
  Then \(\E = \E\cdot J \cong \E\otimes_B J\)
  by~\eqref{eq:restricted_mult} and hence
  \[
  X \defeq \E\otimes_B Y \otimes_B \E^*
  \cong \E\otimes_B J\otimes_B Y \otimes_B J \otimes_B \E^*
  \cong \E\otimes_B (J\cdot Y\cdot J) \otimes_B \E^*.
  \]
  Thus~\(\E\)
  witnesses a Morita equivalence between~\(X\)
  and the Hilbert \(J\)\nb-bimodule~\(J Y J\).
  We may split a Morita restriction into two steps.  First, we
  restrict the Hilbert \(B\)\nb-module~\(Y\)
  to a Hilbert \(J\)\nb-module~\(J Y J\)
  for an ideal~\(J\) in~\(B\);
  secondly, we replace this by a Morita equivalent Hilbert bimodule.
  Restriction to any ideal is a special case of Morita restriction.
\end{remark}

\begin{remark}
  \label{rem:restrictions_of_Fell_bundles}
  Let \(\mathcal{C} = (C_g)_{g\in G}\)
  be a Fell bundle with unit fibre \(C \defeq C_e\).
  Let~\(\E\)
  be a Hilbert \(A,C\)\nb-bimodule
  that is full over~\(A\).
  We transfer the Fell bundle~\(\mathcal{C}\)
  to a Fell bundle
  \(\E \mathcal{C} \E^* \defeq (\E\otimes_C C_g \otimes_C \E^*)_{g\in
    G}\),
  which has the unit fibre
  \(\E \otimes_C C\otimes_C \E^* \cong \E \otimes_C \E^*\cong A\)
  by~\eqref{eq:left_innprod}; it is equipped with the multiplication
  maps and involutions defined by
  \begin{align*}
    (x_1\otimes (c_1,g_1) \otimes y_1^*) \cdot
    (x_2\otimes (c_2,g_2) \otimes y_2^*)
    &\defeq x_1\otimes (c_1,g_1)\cdot (\braket{y_1}{x_2}_C\cdot c_2,g_2)
    \otimes y_2^*,
    \\
    (x_1\otimes (c_1,g_1) \otimes y_1^*)^*
    &\defeq y_1 \otimes (c_1,g_1)^* \otimes x_1^*
  \end{align*}
  for \(x_1,x_2,y_1,y_2\in \E\), \(c_1,c_2\in C\), \(g_1,g_2\in G\).
  Here~\((c_i,g_i)\) denotes~\(c_i\) viewed as an element of the
  fibre~\(C_{g_i}\) of our Fell bundle at~\(g_i\), and~\(y_1^*\)
  denotes~\(y_1\) viewed as an element of the complex conjugate
  Banach space~\(\E^*\).  If \(\mathcal{C}_\gamma \defeq
  (C_{\gamma_g})_{g\in G}\) is the Fell bundle associated to a group
  action \(\gamma\colon G\to \Aut(C)\), then the multiplication and
  involution on \(\E \mathcal{C}_\gamma \E^*\) become
   \begin{align*}
    (x_1\otimes (c_1,g_1) \otimes y_1^*) \cdot
    (x_2\otimes (c_2,g_2) \otimes y_2^*)
    &= x_1\otimes (c_1 \gamma_{g_1}(\braket{y_1}{x_2}_C\cdot c_2), g_1 g_2),
    \\
    (x_1\otimes (c_1,g_1) \otimes y_1^*)^*
    &= y_1 \otimes (\gamma_{g_1^{-1}}(c_1^*),g_1^{-1}) \otimes x_1^*.
  \end{align*}
\end{remark}

Our next goal is
Proposition~\ref{pro:dense_Morita_covering}, which allows, roughly
speaking, to detect various non-triviality conditions for a Hilbert
bimodule~\(X\)
by looking at other Hilbert bimodules that are Morita equivalent to
pieces of~\(X\).
It combines various permanence properties of Kishimoto's condition.
The first property is Morita invariance.  The second is invariance
in both directions
under passing to \emph{essential} ideals.  The third is a locality
condition: if~\(A\)
is a sum of ideals, then~\(X\)
satisfies Kishimoto's condition if and only if its restrictions to
all these ideals do so.  The proofs of our main theorems will show the
usefulness of the last technical property.
The second property is
related to the equivalence of \ref{en:Kishimoto_spectrum}
and~\ref{en:Kishimoto_spectrum_Mult} because \(B\subseteq \Mult(I)\)
if~\(I\)
is an essential ideal in~\(B\).

\begin{definition}
  \label{def:covering_Hilbert_bimodule}
  Let~\(A\)
  be a \(\Cst\)\nb-algebra
  and~\(X\)
  a Hilbert \(A\)\nb-bimodule.
  Let~\(S\)
  be a set.  For each \(i\in S\),
  let \(B_i\)
  be a \(\Cst\)\nb-algebra
  and~\(\E_i\)
  a Hilbert \(B_i,A\)-bimodule.
  Define \(I\defeq \left({}_A\braket{X}{X}+ \braket{X}{X}_A\right)\idealin A\),
  \(K_i \defeq \braket{\E_i}{\E_i}_A \idealin A\)
  and \(Y_i \defeq \E_i\otimes_A X\otimes_A \E_i^*\)
  for \(i\in S\).
  Let \(K\in\I^X(A)\)
  be the smallest \(X\)\nb-invariant
  ideal that contains~\(K_i\)
  for all \(i\in S\).
  We say that~\((Y_i)_{i\in S}\)
  \emph{essentially covers~\(X\)
    \textup{(}up to Morita equivalence\textup{)}} or, more briefly,
  that it is a \emph{dense Morita covering} for~\(X\)
  if \(I\cap K\)
  is an essential ideal in~\(I\).
  We say that~\((Y_i)_{i\in S}\)
  \emph{covers~\(X\)
    \textup{(}up to Morita equivalence\textup{)}} and call it a
  \emph{Morita covering} for~\(X\) if \(I\subseteq K\).
\end{definition}

In this definition, \(I\)
is the smallest ideal for which~\(X\)
is a Hilbert \(I\)\nb-bimodule.
The Hilbert \(B_i\)\nb-bimodule~\(Y_i\)
is Morita equivalent through~\(\E_i\)
to the restriction \(K_i \cdot X\cdot K_i\)
of~\(X\)
to the ideal~\(K_i\).
An ideal~\(J\)
in~\(A\)
corresponds to an open subset~\(\widehat{J}\)
in~\(\widehat{A}\),
and~\(J\)
is \(X\)\nb-invariant
if and only if~\(\widehat{J}\)
is \(\widehat{X}\)\nb-invariant.
The ideal \(I\cap K\)
is essential in~\(I\)
if and only if \(\widehat{I}\cap\widehat{K}\)
is dense in~\(\widehat{I}\).
So the condition for~\((Y_i)_{i\in S}\)
to \emph{essentially} cover~\(X\)
says that the \(\widehat{X}\)\nb-orbit
of \(\bigcup \widehat{K_i} \cap \widehat{I}\)
is dense in~\(\widehat{I}\);
equivalently, any \(\widehat{X}\)\nb-invariant
open subset of~\(\widehat{I}\)
meets \(\bigcup \widehat{K_i}\).
And \((Y_i)_{i\in S}\)
covers~\(X\)
if the \(\widehat{X}\)\nb-orbit
of \(\bigcup \widehat{K_i}\)
contains~\(\widehat{I}\).

\begin{proposition}
  \label{pro:dense_Morita_covering}
  Let~\((Y_i)_{i\in S}\)
  essentially cover~\(X\) up to Morita equivalence.
  \begin{enumerate}
  \item \label{en:dense_Morita_covering_Kish}%
    \(X\)~satisfies Kishimoto's condition if and only if~\(Y_i\)
    does so for all \(i\in S\);
  \item \label{en:dense_Morita_covering_top}%
    \(X\)~is topologically non-trivial if and only if~\(Y_i\)
    is so for all \(i\in S\);
  \item \label{en:dense_Morita_covering_outer}%
    \(X\)~is purely outer if and only if~\(Y_i\)
    is so for all \(i\in S\);
  \item \label{en:dense_Morita_covering_weakly_outer}%
    \(X\)~is purely universally weakly outer if and only if~\(Y_i\)
    is so for all \(i\in S\).
  \end{enumerate}
\end{proposition}

\begin{proof}
  First we improve our Morita covering by enlarging~\(S\).
  Let~\((X_n)_{n\in\Z}\) be the Fell bundle generated by~\(X\) and let
  \[
  \E_{i,n}\defeq \E_i \otimes_A X_{-n},\qquad
  K_{i,n} \defeq \braket{\E_{i,n}}{\E_{i,n}}_A
  \]
  for \(i\in S\), \(n\in\Z\).  Then
  \[
  Y_{i,n}
  \defeq \E_{i,n} \otimes_A X \otimes_A \E_{i,n}^*
  = \E_i \otimes_A X_{-n} \otimes_A X\otimes_A X_n \otimes_A \E_i^*
  \subseteq \E_i \otimes_A X \otimes_A \E_i^*
  = Y_i
  \]
  because~\((X_n)_{n\in\Z}\)
  is a Fell bundle.  Kishimoto's condition is inherited by
  submodules, and so are the properties in
  \ref{en:dense_Morita_covering_top}--%
  \ref{en:dense_Morita_covering_weakly_outer}.
	Thus we may
  replace~\(Y_i\) by the family~\((Y_{i,n})_{i\in S, n\in\Z}\).

  The tensor product of Hilbert bimodules corresponds to the
  composition of partial maps.  Thus the ideal
  \(\braket{\E_{i,n}}{\E_{i,n}}_A\)
  corresponds to the domain of the partial map
  \(\widehat{A}\to\widehat{B}_i\)
  associated to \(\widehat{\E_i}\circ \widehat{X_{-n}}\).
  This is the same as \(\widehat{X_{-n}}^{-1} = \widehat{X_n}\)
  applied to the domain of~\(\widehat{\E_i}\),
  which is~\(\widehat{K_i}\).  Thus
  \(\widehat{K_{i,n}} = \widehat{X_n}(\widehat{K_i})\).
  So \(\widehat{K} \defeq \bigcup \widehat{K_{i,n}}\)
  is the \(\widehat{X}\)\nb-orbit
  of~\(\bigcup \widehat{K_i}\).
  The ideal~\(K\)
  corresponding to~\(\widehat{K}\)
  is \(X\)\nb-invariant.
  By assumption, each open subset of~\(\widehat{I}\)
  meets~\(\widehat{K}\).
  Equivalently, \(I\cap K\) is essential in~\(I\).

  Now we prove~\ref{en:dense_Morita_covering_Kish}.  First assume
  that~\(X\)
  satisfies Kishimoto's condition.  By definition, \(Y_i\)
  for \(i\in S\)
  is a Morita restriction of~\(X\).
  Thus it satisfies Kishimoto's condition by
  Proposition~\ref{pro:Kish_Morita_restriction}.  Conversely, assume
  that~\(Y_i\)
  satisfies Kishimoto's condition for all \(i\in S\).
  We want to prove Kishimoto's condition for~\(X\).
  So we fix \(x\in X\)
  and \(D\in\Her(A)\).
  If \(D\cap I=0\),
  then \(X\cdot D=0\)
  and hence any \(x\in D^+\)
  with \(\norm{x}=1\)
  will do.  So we may assume that \(D\cap I\neq0\).
  This implies \(D\cap K\neq 0\)
  because \(I\cap K\)
  is essential in~\(I\).

  The ideal~\(K\)
  is the closed linear span of the ideals
  \(K_{i,n}\defeq \braket{\E_{i,n}}{\E_{i,n}}_A\)
  with~\(\E_{i,n}\)
  for \(i\in S\),
  \(n\in \Z\)
  as above.  Hence there are \(i\in S\),
  \(n\in \Z\)
  with \(D\cap K_{i,n}\neq0\).
  Kishimoto's condition for \(x\)
  and~\(D\)
  is weaker than the same condition for \(x\)
  and~\(D\cap K_{i,n}\).
  Hence we may assume without loss of generality that
  \(D\in\Her(K_{i,n})\).
  The Hilbert bimodule~\(\E_{i,n}\)
  witnesses that the Hilbert \(K_{i,n}\)-bimodule
  \(K_{i,n} X K_{i,n}\) is a Morita restriction of~\(Y_i\):
  \[
  K_{i,n} X K_{i,n} = \E_{i,n}^* \otimes_{B_i} \E_{i,n} \otimes_A X
  \otimes_A \E_{i,n}^* \otimes_{B_i} \E_{i,n}
  = \E_{i,n}^* \otimes_{B_i} Y_i \otimes_{B_i} \E_{i,n}.
  \]
  Since~\(Y_i\)
  satisfies Kishimoto's condition, so does~\(K_{i,n} X K_{i,n}\)
  by Proposition~\ref{pro:Kish_Morita_restriction}.
  Lemma~\ref{lem:technical} gives \(d\in D^+\)
  with \(\norm{d}=1\)
  such that \(D_0 \defeq\{a\in A\mid da=ad=a\}\)
  is in \(\Her(D)\).
  Since \(d x d\in DXD\subseteq K_{i,n} X K_{i,n}\),
  we may apply Kishimoto's condition for \(K_{i,n} X K_{i,n}\)
  to the pair~\((d x d,D_0)\).
  This gives \(a_\varepsilon\in D_0^+\)
  with \(\norm{a_\varepsilon}=1\)
  and \(\norm{a_\varepsilon d x d a_\varepsilon}<\varepsilon\)
  for all \(\varepsilon>0\).
  Since \(D_0\subseteq D\)
  and \(a d = a = a d\)
  for \(a\in D_0\),
  the elements~\(a_\varepsilon\)
  witness Kishimoto's condition for~\((x,D)\).
  Hence~\(X\) satisfies Kishimoto's condition.

  We prove~\ref{en:dense_Morita_covering_top}.  Assume first
  that~\(Y_i\)
  is not topologically non-trivial for some \(i\in S\).
  That is, there are \(i\in S\)
  and an open subset~\(U\)
  of~\(\widehat{B_i}\)
  on which~\(\widehat{Y_i}\)
  acts identically.  The Hilbert bimodule~\(\E_i\)
  induces a homeomorphism~\(\widehat{\E_i^*}\)
  from~\(\widehat{B_i}\)
  to the open subset~\(\widehat{K_i}\)
  in~\(\widehat{A}\).
  Since the action of Hilbert bimodules on representations is
  functorial with respect to the composition of Hilbert bimodules,
  \(\widehat{X}\)
  acts identically on~\(\widehat{\E_i^*}(U)\),
  which is open in~\(\widehat{A}\).
  Conversely, assume that~\(X\)
  is not topologically non-trivial.  So there is a nonempty open subset~\(U\)
  of~\(\widehat{I}\)
  with \(\widehat{X}|_U = \id_U\).
  The subset~\(U\)
  is \(\widehat{X}\)\nb-invariant.
  It intersects the image~\(\widehat{K_i}\)
  of~\(\widehat{\E_i^*}\)
  non-trivially for some \(i\in S\)
  because~\((Y_i)_{i\in S}\) essentially covers~\(X\).
  Then \(V \defeq \widehat{\E_i}(U)\)
  is a nonempty open subset of~\(\widehat{B_i}\)
  with \(\widehat{Y_i}|_V = \id_V\).
  So~\(Y_i\) is not topologically non-trivial.

  We prove~\ref{en:dense_Morita_covering_outer}.
  The same argument for the bidual \(\textup{W}^*\)\nb-algebras
  shows~\ref{en:dense_Morita_covering_weakly_outer}.
  Assume first
  that~\(Y_i\)
  is partly inner for some \(i\in S\).
  Then there is a non-zero invariant ideal \(L\idealin B_i\)
  such that \(L\cdot Y_i = Y_i\cdot L\)
  is isomorphic to~\(L\)
  with the obvious Hilbert bimodule structure.  For any ideal~\(L\)
  in~\(B_i\),
  \(\E_i^* \otimes_{B_i} Y_i \cdot L \otimes_{B_i} \E_i\)
  is a Hilbert subbimodule in
  \(\E_i^* \otimes_{B_i} Y_i \otimes_{B_i} \E_i\),
  which is, in turn, a Hilbert subbimodule in~\(X\).
  Thus
  \(\E_i^* \otimes_{B_i} Y_i \cdot L \otimes_{B_i} \E_i = X\cdot L' =
  L'\cdot X\),
  where \(L' \defeq \braket{L\cdot \E_i}{L\cdot \E_i}_A \idealin A\)
  is the ideal in~\(A\)
  corresponding to~\(L\)
  under the Rieffel correspondence for~\(\E_i\).
  Since \(Y_i\cdot L \cong L\),
  the restriction \(L'\cdot X\)
  is isomorphic to
  \(\E_i^* \otimes_{B_i} Y_i \cdot L \otimes_{B_i} \E_i \cong \E_i^*
  \otimes_{B_i} L \otimes_{B_i} \E_i \cong \braket{\E_i\cdot
    L}{\E_i\cdot L}_A = L'\).  Thus~\(X\) is partly inner.

  Conversely, if~\(X\)
  is partly inner, then there is a non-zero ideal \(L\idealin A\)
  such that \(X\cdot L \cong L\)
  as Hilbert bimodules.  We have \(L\subseteq I\)
  because \(X=X\cdot I\).
  Since \(K\cap I\)
  is essential in~\(I\),
  the intersection \(L\cap K\)
  is non-zero.  Then \(L\cap K_{i,n} \neq0\)
  for some \(i\in S\),
  \(n\in\Z\)
  because
  \(\widehat{L} \cap \bigcup \widehat{K_{i,n}} \neq\emptyset\).
  Then \(Y_{i,n} \defeq \E_{i,n} \otimes_A X \otimes_A \E_{i,n}^*\)
  contains the non-zero Hilbert subbimodule
  \(\E_{i,n} \otimes_A (X\cdot L) \otimes_A \E_{i,n}^* \cong \E_{i,n}
  \otimes_A L \otimes_A \E_{i,n}^* \cong {}_{B_i}\braket{\E_{i,n}
    L}{\E_{i,n}L}\),
  which is isomorphic to the ideal of~\(B_{i,n}\)
  corresponding to \(L\cap K_{i,n}\)
  under the Rieffel correspondence.  This is a non-zero ideal~\(L'\)
  with \(Y_{i,n} L' \cong L'\).  Thus~\(Y_{i,n}\) is partly inner.
\end{proof}


\begin{corollary}
  \label{cor:restriction_to_essential_ideals}
  Let~\(K\) be an essential ideal in \(\left({}_A\braket{X}{X}+
    \braket{X}{X}_A\right)\).  The Hilbert \(A\)\nb-bimodule~\(X\)
  has one of the properties mentioned in
  Proposition~\textup{\ref{pro:dense_Morita_covering}} if and only
  if the restricted Hilbert \(K\)\nb-bimodule~\(KXK\) has that
  property.
\end{corollary}

\begin{proof}
  Let \(\E\defeq K\) with the Hilbert \(K,A\)-bimodule structure
  inherited from~\(A\).  It establishes that~\(K X K\) essentially
  covers~\(X\) because \(\E\otimes_A X\otimes_A \E= K \otimes_A
  X\otimes_A K\cong K X K\).
\end{proof}

\section{A convenient Morita globalisation}
\label{sec:Morita_globalisation}

Let \(\mathcal{B} = (B_g)_{g\in G}\)
be a Fell bundle over a discrete group~\(G\)
with unit fibre \(A\defeq B_e\).
A \emph{Morita globalisation} of~\(\mathcal{B}\)
consists of a \(\Cst\)\nb-algebra~\(C\),
a group action \(\gamma\colon G\to\Aut(C)\),
a Hilbert \(A,C\)-bimodule~\(\E\)
that is full over~\(A\),
and an isomorphism between~\(\mathcal{B}\)
and the Fell bundle \(\E \mathcal{C}_\gamma \E^*\)
constructed in Remark~\ref{rem:restrictions_of_Fell_bundles}; this
Fell bundle isomorphism consists of Hilbert bimodule isomorphisms
\(B_g \cong \E\otimes_C C_{\gamma_g} \otimes_C \E^*\)
for \(g\in G\)
that are compatible with the multiplication maps and involutions.  We
are going to construct a canonical Morita globalisation using a
variant of Takesaki--Takai duality.  This construction is already
studied by Quigg~\cite{Quigg:Discrete_coactions_and_bundles} in the
language of coactions.  If~\(\mathcal{B}\)
comes from a partial action~\(\alpha\),
then this Morita globalisation agrees with the Morita enveloping
action of~\(\alpha\)
analysed in~\cite{Abadie:Enveloping}.

Following
\cite{Quigg:Discrete_coactions_and_bundles}*{Definition~2.13}, we call
a property~\(\mathcal{P}\)
of \(\Cst\)\nb-algebras
\emph{ideal} if it is invariant under Morita equivalence and inherited
by ideals and if every \(\Cst\)\nb-algebra
has a largest ideal with this property.  Examples of such properties
are: being liminal, antiliminal, Type~I\(_0\),
Type~I or nuclear (see~\cite{Quigg:Discrete_coactions_and_bundles}).
We warn the reader that the ``ideal property'' is not ``ideal'' in
this sense.

\begin{proposition}
  \label{pro:Morita_globalisation}
  Let \(\mathcal{B} = (B_g)_{g\in G}\)
  be a Fell bundle over a discrete group~\(G\)
  with unit fibre~\(A\).
  Let~\(\Cred(\mathcal{B})\)
  be its reduced section \(\Cst\)\nb-algebra,
  equipped with the dual \(G\)\nb-coaction~\(\delta_G\),
  and let
  \(C\defeq \Cst(\mathcal{B})\rtimes_{\red,\delta_G} \widehat{G}\)
  be the corresponding reduced crossed product.  Let
  \(\gamma\colon G\to\Aut(C)\)
  be the \(G\)\nb-action
  on~\(C\)
  dual to~\(\delta_G\).
  Let~\(\mathcal{P}\)
  be a property of \(\Cst\)\nb-algebras that is ideal.  Then
  \begin{enumerate}
  \item \label{en:Morita_globalisation_1}%
    \(\gamma\colon G\to \Aut(C)\) is a Morita globalisation
    of~\(\mathcal{B}\);
  \item \label{en:Morita_globalisation_2}%
    for each \(g\in G\), the Hilbert bimodules~\((B_{h g
      h^{-1}})_{h\in G}\) cover~\(C_{\gamma_g}\) up to Morita
    equivalence;
  \item \label{en:Morita_globalisation_3}%
    \(\Cred(\mathcal{B})\) and \(C\rtimes_{\red,\gamma} G\)
    are Morita--Rieffel equivalent, and the induced lattice
    isomorphism \(\I(\Cred(\mathcal{B}))\cong
    \I(C\rtimes_{\gamma,\red} G)\) restricts to an isomorphism
    between the lattices of graded ideals,
    \(\I^{\widehat{G}}(\Cred(\mathcal{B}))\cong
    \I^{\widehat{G}}(C\rtimes_{\gamma,\red} G)\);
  \item \label{en:Morita_globalisation_3_and_a_bit}%
    there is lattice isomorphism
    \(\I^{\mathcal{B}}(A)\cong\I^{\gamma}(C)\),
    \(I\mapsto \Cst(\mathcal{B}|_{I})\rtimes_{\red,\delta_G}
    \widehat{G}\);
  \item\label{en:Morita_globalisation_3_and_half}%
    the Hilbert bimodule~\(\E\)
    that witnesses that \(\gamma\colon G\to \Aut(C)\)
    is a Morita globalisation of~\(\mathcal{B}\)
    is an equivalence bimodule if and only if~\(\mathcal{B}\)
    is saturated, that is, the fibres~\(B_g\),
    \(g\in G\), are equivalence bimodules over~\(A\);
  \item \label{en:Morita_globalisation_4}%
    \(A\)~has the property~\(\mathcal{P}\) if and only if~\(C\) has
    that property;
  \item \label{en:Morita_globalisation_5}%
    if~\(G\) is countable, then~\(A\) is separable if and only
    if~\(C\) is separable;
  \item \label{en:Morita_globalisation_8}%
    \(C\)~is simple if and only if~\(A\) is simple and \(B_g\neq0\)
    for all \(g\in G\).
  \end{enumerate}
\end{proposition}

\begin{proof}
  The double crossed product
  \(C\rtimes_{\gamma,\red} G \cong
  (\Cst(\mathcal{B})\rtimes_{\delta_G,\red}
  \widehat{G})\rtimes_{\gamma,\red} G\)
  is ``\(\widehat{G}\)\nb-\hspace{0pt}equivariantly''
  isomorphic to \(\Cst(\mathcal{B}) \otimes \Comp(L^2(G))\)
  with the diagonal coaction by Katayama's version of Takesaki
  duality~\cite{Katayama:Takesaki_Duality}.  Thus
  \(\Cred(\mathcal{B})\)
  and \(C\rtimes_{\gamma,\red} G\)
  are \(\widehat{G}\)\nb-equivariantly
  Morita--Rieffel equivalent.  This implies an isomorphism between the
  ideal lattices that preserves the sublattices of graded ideals,
  proving~\ref{en:Morita_globalisation_3}.  This
  implies~\ref{en:Morita_globalisation_3_and_a_bit} by
  Proposition~\ref{pro:gauge-invariance_vs_separation}.

  Next we prove~\ref{en:Morita_globalisation_1}.  We first recall a
  more concrete description of~\(C\) from
  \cite{Abadie:Enveloping}*{Proposition~8.1}.  The finitely
  supported functions \(k\colon G\times G\to \mathcal{B}\) with
  \(k(r,s)\in B_{rs^{-1}}\) form a \Star{}algebra with the product
  and involution
  \[
  k_1 *k_2 (r,s) \defeq \sum_{t\in G} k_1(r,t) k_2(t,s),\qquad
  k^*(r,s)\defeq k(s,r)^*.
  \]
  This is a normed \Star{}algebra for the ``Hilbert--Schmidt'' norm
  \[
  \norm{k}_2 \defeq \biggl(\sum_{r,s\in G} \norm{k(r,s)}^2\biggr)^{\nicefrac12}.
  \]
  The group~\(G\) acts on it by \Star{}automorphisms:
  \[
  \gamma_t(k)(r,s) \defeq k(rt,st), \qquad r,s,t\in G.
  \]
  The enveloping \(\Cst\)\nb-algebra of this normed \Star{}algebra
  is identified in \cite{Abadie:Enveloping}*{Proposition~8.1} with
  the crossed product~\(C\) with its dual \(G\)\nb-action.

  Let~\(b \,1_{(s,r)}\) for \(r,s\in G\), \(b\in B_{r s^{-1}}\) be the
  section \(G\times G\to\mathcal{B}\) with \(b\,1_{(s,r)}(s,r)=b\) and
  \(b\,1_{(s,r)}(t,u) =0\) for \((s,r)\neq (t,u)\).  These elements
  span~\(C\), and they satisfy \((b \,1_{(s,r)})* (c \,1_{(t,u)}) =
  \delta_{r,t} (b\cdot c)\,1_{(s,u)}\), \((b \,1_{(s,r)})^* = b^*
  \,1_{(r,s)}\), and \(\gamma_t(b \,1_{(s,r)}) = b \,1_{(s t^{-1}, r
    t^{-1})}\).  When \(r=s\) we extend this notation by considering
  also \(b=1\), a multiplier of~\(B_e\).
  Hence \(p_r\defeq \,1_{r,r}\) is a multiplier of~\(C\).  This is
  an orthogonal projection with \(\gamma_t(p_r) = p_{r t^{-1}}\).
  (If~\(C\) is treated as a \(\Cst\)\nb-subalgebra of adjointable
  operators on the Fock module \(L^2(\mathcal{B})\), then~\(p_g\)
  corresponds to the projection onto the summand~\(B_g\), see
  \cite{Abadie:Enveloping}*{Proposition~5.6}.)  The convolution
  formula implies
  \begin{align*}
    C * p_e &= \clsp\{b \,1_{(g,e)}\mid b\in B_g,\ g\in G\},\\
    p_g*C*p_e &= \{b_g \,1_{(g,e)}\mid b\in B_g\}
  \end{align*}
  for all \(g\in G\).  In particular, there are linear isomorphisms
  \begin{equation}
    \label{eq:iso_Bg_corner_C}
    B_g \xrightarrow[\cong]{\psi_g} p_g*C*p_e,\quad
    b_g \mapsto b_g \,1_{(g,e)},\qquad \forall g\in G.
  \end{equation}
  The subspace \(p_g*C*p_e\subseteq C\) is a Hilbert
  \(p_g*C*p_g,p_e*C*p_e\)-bimodule with operations inherited
  from~\(C\).  The group action restricts to isomorphisms
  \(\gamma_{g^{-1}}\colon p_e*C*p_e\congto p_g*C*p_g\) for \(g\in
  G\).  Composing the left Hilbert \(p_g*C*p_g\)-module structure on
  \(p_g*C*p_e\)
  with this isomorphism gives a Hilbert bimodule
  \({}_{\gamma_g^{-1}}(p_g * C * p_e)\) over \(p_e*C*p_e\).
  Its Hilbert bimodule structure is given by
  \begin{alignat*}{2}
    \gamma_g^{-1}(a \,1_{(e,e)}) * (b\,1_{(g,e)})
    &= (a b) \,1_{(g,e)},&\qquad
    \gamma_g\bigl((b \,1_{(g,e)}) * (c \,1_{(g,e)})^*\bigr)
    &= (b c^*)\,1_{(e,e)},\\
    (b\,1_{(g,e)}) * (a \,1_{(e,e)})
    &= (b a) \,1_{(g,e)},&\qquad
    (b \,1_{(g,e)})^* * (c \,1_{(g,e)})
    &= (b^* c)\,1_{(e,e)}
  \end{alignat*}
  for \(a\in A\), \(b,c\in B_g\).
  Thus the pair \((\psi_e,\psi_g)\) of isomorphisms \(p_e*C*p_e\cong
  B_e = A\) and \(p_g*C*p_e\cong B_g\) in~\eqref{eq:iso_Bg_corner_C}
  is an isomorphism of Hilbert bimodules \(B_g \cong
  {}_{\gamma_g^{-1}}(p_g * C * p_e)\).  In particular, the
  maps~\(\psi_g\) for \(g\in G\) are isometric.

  We turn the group action~\(\gamma\) on~\(C\) into a Fell bundle
  over~\(G\) as in Example~\ref{exa:automorphism_to_Fell}.  We shall
  use the isomorphic variant \({}_{\gamma_g^{-1}} C\) instead
  of~\(C_{\gamma_g}\).  The right ideal \(\E \defeq p_e* C\) is a
  Hilbert \(p_e* C * p_e,C\)-bimodule full over \(p_e* C * p_e\).
  Identifying \(p_e* C * p_e\) with~\(A\) by~\(\psi_e\), we
  view~\(\E\) as a Hilbert \(A,C\)-bimodule.  The multiplication
  isomorphisms \eqref{eq:left_mult} and~\eqref{eq:right_mult} and
  the isomorphisms~\(\psi_g\) give Hilbert bimodule isomorphisms
  \begin{multline*}
    \E\otimes_C C_{\gamma_g} \otimes_C \E^*
    \cong \E\otimes_C {}_{\gamma_g^{-1}}C \otimes_C \E^*
    \defeq p_e * C \otimes_C {}_{\gamma_g^{-1}}C \otimes_C C * p_e
    \\\cong p_e * {}_{\gamma_g^{-1}}C * p_e
    \cong {}_{\gamma_g^{-1}}\bigl(\gamma_g^{-1}(p_e) * C * p_e\bigr)
    \cong {}_{\gamma_g^{-1}}(p_g * C * p_e)
    \xrightarrow[\cong]{\psi_g} B_g.
  \end{multline*}
  All isomorphisms are explicit, and it is easy to check that they
  give an isomorphism of Fell bundles \(\mathcal{B}\cong \E
  \mathcal{C}_\gamma \E^*\) for the canonical Fell bundle
  structure on \(\E \mathcal{C}_\gamma \E^*
  =(\E\otimes_C C_{\gamma_g} \otimes_C \E^*)_{g\in G}\) described in
  Remark~\ref{rem:restrictions_of_Fell_bundles}.  This finishes the
  proof of~\ref{en:Morita_globalisation_1}.

  Now we prove~\ref{en:Morita_globalisation_2}.  Fix \(g\in G\).
  Let \(A_h\defeq A\) and \(\E_h \defeq \E\otimes_C C_{\gamma_h}\)
  for \(h\in G\).  Then
  \begin{align*}
    Y_h &\defeq \E_h\otimes C_{\gamma_g} \otimes \E_h^*
    \defeq
    (\E \otimes_C C_{\gamma_h}) \otimes_C C_{\gamma_g} \otimes_C
    (\E\otimes_C C_{\gamma_h})^*
    \\&\cong \E \otimes_C C_{\gamma_h} \otimes_C C_{\gamma_g} \otimes_C
    C_{\gamma_{h^{-1}}} \otimes_C \E^*
    \\&\cong \E \otimes_C C_{\gamma_{h g h^{-1}}} \otimes_C \E^*
    \cong B_{h g h^{-1}}.
  \end{align*}
  We claim that the family of Hilbert \(A\)\nb-bimodules~\((B_{h g
    h^{-1}})_{h\in G}\) covers~\(C_{\gamma_g}\) up to Morita
  equivalence, witnessed by the Hilbert bimodules~\((\E_h)_{h\in
    G}\).  The Hilbert \(A,C\)-bimodules \(\E_h \defeq \E\otimes_C
  C_{\gamma_h}\) for \(h\in G\) are full over~\(A\).
  The Hilbert bimodule \(\E = p_e * C\) is an
  equivalence bimodule for~\(A\) and the ideal \(I\defeq C*p_e*C
  \idealin C\) generated by~\(p_e\).  Thus
  \begin{multline*}
    \E_h^* \otimes_A \E_h
    \cong C_{\gamma_h}^* \otimes_C I \otimes_C C_{\gamma_h}
    \\\cong \braket{I\cdot C_{\gamma_h}}{I\cdot C_{\gamma_h}}_C
    = \gamma_{h^{-1}}(I)
    = C*\gamma_{h^{-1}}(p_e)*C
    = C*p_h*C.
  \end{multline*}
  The series \(\sum_{h\in G} p_h\) converges to~\(1\) in the strict
  topology on \(\Mult(\Cst(\widehat{G})) \subseteq \Mult(C)\).  Hence
  \(\sum C*p_h*C = C\).  Thus~\((B_{h g h^{-1}})_{h\in G}\)
  covers~\(C_{\gamma_g}\) up to Morita equivalence through the
  Hilbert \(A,C\)-bimodules~\(\E_h\) for \(h\in G\).

  Statement~\ref{en:Morita_globalisation_3_and_half} is
  \cite{Quigg:Discrete_coactions_and_bundles}*{Corollary 2.7}.  We
  include a proof in our notation.  Let \(g\in G\).
  The isomorphism \(\E\otimes_C C_{\gamma_g} \otimes_C \E^*\cong B_g\)
  implies that the Hilbert \(A\)\nb-bimodule~\(B_g\)
  is Morita equivalent to the restriction~\(I C_{\gamma_g} I\)
  of \(C_{\gamma_g}\)
  to \(I\defeq C*p_e*C\),
  see Remark~\ref{rem:restrict_Hilbert_bimodules}.  Clearly,
  \(I C_{\gamma_g} I = I_{\alpha_g}\),
  where~\(\alpha_g\)
  is the restriction of~\(\gamma_g\)
  to a partial isomorphism with domain
  \(D_{g^{-1}}\defeq I\cap \gamma_{g^{-1}}(I)\).
  Hence~\(\mathcal{B}\)
  is saturated if and only if the Hilbert bimodules~\(I_{\alpha_g}\)
  for \(g\in G\)
  are full (both on the left and right), if and only if~\(I\)
  is \(\gamma\)\nb-invariant.
  Since \(C= \sum_{g\in G} \gamma_g(I)\),
  this happens only if \(I=C\).

  Statement~\ref{en:Morita_globalisation_4} is
  \cite{Quigg:Discrete_coactions_and_bundles}*{Theorem 2.15} or
  \cite{Abadie:Enveloping}*{Corollary~5.2}.

  We prove~\ref{en:Morita_globalisation_5}.  If~\(A\) is separable,
  then each~\(B_g\) is separable by Lemma~\ref{lem:separable_A_X} below.
  Then~\(C\) is separable because~\(G\) is countable.  Conversely,
  if~\(C\) is separable, then so is \(p_e * C * p_e \cong A\).

  Statement~\ref{en:Morita_globalisation_8} is
  \cite{Quigg:Discrete_coactions_and_bundles}*{Theorem 2.10}.  We
  include a proof.  Let~\(A\)
  be simple and \(B_g\neq0\)
  for all \(g\in G\).
  Then~\(B_g\)
  is full for all \(g\in G\),
  so that~\(\mathcal{B}\)
  is saturated.  Hence \(A\)
  and~\(C\)
  are Morita--Rieffel equivalent
  by~\ref{en:Morita_globalisation_3_and_half}.  So~\(C\)
  is simple.  Conversely, let~\(C\)
  be simple.  Then \(I=C\),
  that is, \(A\)
  is Morita--Rieffel equivalent to~\(C\).
  Then~\(A\)
  is simple.  And~\(\mathcal{B}\)
  is saturated by~\ref{en:Morita_globalisation_3_and_half}.
\end{proof}

\begin{lemma}
  \label{lem:separable_A_X}
  If~\(A\) is separable, then any Hilbert \(A\)\nb-bimodule~\(X\) is
  separable.
\end{lemma}

\begin{proof}
  Since~\(\Comp(X)\) is isomorphic to an ideal in~\(A\) (see, for
  instance, \cite{Kwasniewski:Cuntz-Pimsner-Doplicher}*{Proposition
    1.11}),
  \(\Comp(X)\) is  separable.  Thus~\(\Comp(X)\)
  has a countable approximate unit~\((u_n)\).
  Approximating each~\(u_n\)
  by a sequence of finite-rank operators, we see that~\(X\)
  is countably generated as a right Hilbert \(A\)\nb-module.
  Let~\((x_n)_{n\in \N}\subseteq X\)
  be a sequence generating~\(X\)
  as a right Hilbert \(A\)\nb-module and let \((a_n)_{n\in \N} \subseteq A\)
  be a dense sequence.  The linear span of
  \(\{x_n a_m \mid n,m\in\N\}\) is dense in~\(X\).
\end{proof}

\section{Application of Morita coverings to Hilbert bimodules}
\label{sec:tok}

In this section, we use Morita coverings to generalise
Theorem~\ref{the:automorphism_Kishimoto_vs_properly_outer} to
Hilbert bimodules.  Moreover, we show that separability is not
needed in the Type~I case or, more generally, if there is an
essential ideal of Type~I.

\begin{theorem}
  \label{the:Hilbi_Kish_vs_weakly_inner}
  Let~\(X\)
  be a Hilbert \(A\)\nb-bimodule~\(X\).
  Consider the following conditions:
  \begin{enumerate}[wide,label=\textup{(\ref*{the:Hilbi_Kish_vs_weakly_inner}.\arabic*)}]
  \item \label{en:kish}%
    \(X\)~satisfies Kishimoto's condition;
  \item \label{en:topol}%
    \(X\)~is topologically non-trivial;
  \item \label{en:univ_outer}%
    \(X\)~is purely universally weakly outer;
  \item \label{en:outer}%
    \(X\)~is purely outer.
  \end{enumerate}
  Then \ref{en:topol}%
  \(\Rightarrow\)\ref{en:univ_outer}%
  \(\Rightarrow\)\ref{en:outer}%
  \(\Leftarrow\)\ref{en:kish}.  If~\(A\) contains a separable,
   essential ideal, then the conditions
  \ref{en:kish}--\ref{en:univ_outer} are equivalent.  If~\(A\)
  contains a simple, essential ideal, then
  \ref{en:kish}\(\Leftrightarrow\)\ref{en:outer}.  If~\(A\) contains
  an essential ideal of Type~I, then all the conditions
  \ref{en:kish}--\ref{en:outer} are equivalent.
\end{theorem}

\begin{proof}
  Extend~\(X\) to a Fell bundle~\((X_n)_{n\in\N}\) over~\(\Z\) as in
  Example~\ref{def:Fell_bundle_for_Hilbert_bimodule}.
  Proposition~\ref{pro:Morita_globalisation} gives a canonical
  Morita globalisation \(\gamma\colon \Z\to\Aut(C)\) for this Fell
  bundle such that a countable number of copies of~\(X\) cover the
  Hilbert \(C\)\nb-bimodule~\(C_{\gamma_{1}}\).  By
  Proposition~\ref{pro:dense_Morita_covering}, \(X\) has one of the
  four properties in our theorem if and only if~\(C_{\gamma_1}\) has
  it.  Hence the implications \ref{en:topol}%
  \(\Rightarrow\)\ref{en:univ_outer}%
  \(\Rightarrow\)\ref{en:outer}%
  \(\Leftarrow\)\ref{en:kish} follow from the corresponding
  implications in
  Theorem~\ref{the:automorphism_Kishimoto_vs_properly_outer}.
  Furthermore, by Corollary
  \ref{cor:restriction_to_essential_ideals}, if~\(K\) is an
  essential ideal in~\(A\), then~\(X\) has one of the four
  properties in our theorem if and only if its restriction~\(KXK\)
  has it.  Therefore, to prove the remaining part of the assertion
  we may assume that the essential ideal in question is equal
  to~\(A\).  If~\(A\) is separable, then~\(C\) is separable by
  Proposition~\ref{pro:Morita_globalisation}, and the equivalence
  \ref{en:kish}--\ref{en:univ_outer} follows from
  Theorem~\ref{the:automorphism_Kishimoto_vs_properly_outer}.
  If~\(A\) is simple and \(X\neq0\), then~\(X\) is full on both
  sides and hence \(X_n\neq0\) for all \(n\in\Z\).  So~\(C\) is simple by
  Proposition~\ref{pro:Morita_globalisation}.%
  \ref{en:Morita_globalisation_8}.  By
  Theorem~\ref{the:automorphism_Kishimoto_vs_properly_outer},
  \ref{en:kish} and~\ref{en:outer} are equivalent for~\(\gamma_1\)
  and hence for~\(X\).  The case \(X=0\) is trivial because, by
  convention, \(0\) satisfies Kishimoto's condition and is purely
  outer.

  Finally, assume that~\(A\)
  contains an essential ideal of Type~I.  Now a different proof is
  needed because
  Theorem~\ref{the:automorphism_Kishimoto_vs_properly_outer} does not
  apply.  We reduce this case in two steps to the case of Hilbert
  bimodules over commutative \(\Cst\)\nb-algebras
  coming from partial homeomorphisms.  Let~\(Z\)
  be a locally compact space and let~\(\theta\)
  be a partial homeomorphism of~\(Z\).
  Let~\(\theta^*\)
  be the induced partial isomorphism of~\(\Cont_0(Z)\).
  Since \(\Cont_0(Z)^{**}\)
  is commutative, \(\theta^*\)
  is partly inner if and only if it is partly universally weakly
  inner, if and only if~\(\theta\)
  restricts to the identity map on some open subset.  This is
  equivalent to not being topologically non-trivial because the space
  of irreducible representations of~\(\Cont_0(Z)\)
  is just~\(Z\),
  and it is equivalent to Kishimoto's condition by an elementary
  argument, see also
  \cite{Giordano-Sierakowski:Purely_infinite}*{Proposition~A.7}.  Now
  assume that~\(A\)
  contains an essential ideal of Type~I.  We are going to construct a
  dense Morita covering of~\(X\)
  where each piece is of the form~\(\theta^*\)
  for some partial isomorphism.  We have just seen that our four
  properties are equivalent for these pieces~\(\theta^*\).
  By Proposition~\ref{pro:dense_Morita_covering}, \(X\)
  has one of our four properties if and only if each of the pieces of
  a dense Morita covering has it.  Thus the construction of the dense
  Morita covering will finish the proof.

  The essential Type~I ideal in~\(A\)
  contains an essential ideal~\(K\)
  with continuous trace by
  \cite{Pedersen:Cstar_automorphisms}*{Theorem 6.2.11}.  This ideal is
  still essential in~\(A\).
  There is a set of ideals~\((K_j)_{j\in S}\)
  in~\(K\)
  with \(K = \sum K_j\)
  and such that each~\(K_j\)
  is Morita--Rieffel equivalent to a commutative \(\Cst\)\nb-algebra
  \(\Cont_0(Z_j)\)
  for a locally compact space~\(Z_j\),
  compare \cite{Huef-Kumjian-Sims:Dixmier-Douady}*{Theorem~3.3}.
  Let~\(\E_j\)
  be the equivalence \(\Cont_0(Z_j),K_j\)-bimodule.
  Let \(Y_j \defeq \E_j \otimes_A X \otimes_A \E_j^*\).
  The family of Hilbert \(\Cont_0(Z_j)\)-bimodules~\((Y_j)_{j\in
    S}\)
  essentially covers~\(X\)
  up to Morita equivalence because \(\sum K_j = K\)
  is essential in~\(A\).
  Thus we have reduced the case where~\(A\)
  contains an essential ideal of Type~I to the case where~\(A\)
  is commutative.  So let \(A=\Cont_0(Z)\)
  and let~\(X\) be a Hilbert \(\Cont_0(Z)\)-bimodule.

  The structure of a Hilbert bimodule over \(\Cont_0(Z)\)
  is well known.  Namely, let~\(\theta\)
  be the partial homeomorphism on the spectrum~\(Z\)
  of \(\Cont_0(Z)\)
  induced by~\(X\).
  Then there is a line bundle~\(L\)
  over the domain of~\(\theta\)
  such that~\(X\)
  is isomorphic to the space of sections of~\(L\),
  with \(\Cont_0(Z)\)
  acting by pointwise multiplication on the right and by pointwise
  multiplication combined with~\(\theta^*\)
  on the left.  We could now treat this case directly.  We prefer,
  however, to remove the line bundle by another dense Morita covering.
  We construct this without reference to the structure of Hilbert
  bimodules over \(\Cont_0(Z)\).

  Let \(U\subseteq Z\)
  be the open subset that corresponds to the ideal
  \(\braket{X}{X}_{\Cont_0(Z)} \idealin \Cont_0(Z)\).
  For each \(z\in U\),
  there is \(x_z\in X\)
  with \(\braket{x_z}{x_z}(z)\neq0\).
  Let \(U_z \defeq \{z\in U \mid \braket{x_z}{x_z}(z)\neq0\}\)
  and let \(U_\infty \defeq Z\setminus \overline{U}\).
  Then \(U_\infty \cup \bigcup_{z\in U} U_z = Z \setminus \partial U\)
  is dense in~\(Z\).
  So the Hilbert \(\Cont_0(U_z)\)-bimodules
  \(X_z\defeq \Cont_0(U_z)\cdot X\cdot \Cont_0(U_z)\)
  for \(z\in U\)
  and \(z=\infty\)
  essentially cover~\(X\), compare
  Corollary~\ref{cor:restriction_to_essential_ideals}.
  Thus it suffices to prove the theorem for each of these restrictions
  of~\(X\).
  The case \(X_\infty = 0\)
  is trivial.  So it remains to consider the Hilbert
  \(\Cont_0(U_z)\)-bimodules~\(X_z\)
  for \(z\in U\).
  We are going to prove that each~\(X_z\)
  is associated to a partial homeomorphism of~\(U_z\).

  First, we claim that \(X'_z \defeq X\cdot \Cont_0(U_z)\)
  is isomorphic to \(\Cont_0(U_z)\)
  as a right Hilbert \(\Cont_0(U_z)\)-module.
  We have \(\eta\in X'_z\)
  if and only if \(\braket{\eta}{\eta} \in \Cont_0(U_z)\).
  Hence~\(X'_z\)
  contains~\(x_z\).
  Since~\(X'_z\)
  is a Hilbert \(\Cont_0(Z),\Cont_0(U_z)\)\nb-bimodule,
  the compact operators on~\(X'_z\)
  are isomorphic to an ideal in~\(\Cont_0(Z)\)
  and hence commutative.  Therefore,
  \(\ket{\eta}\bra{x_z} \cdot \ket{x_z}\bra{x_z} = \ket{x_z}\bra{x_z}
  \cdot \ket{\eta}\bra{x_z}\)
  for any \(\eta\in X'_z\).
  This implies
  \(\eta\in \eta\cdot\Cont_0(U_z) \subseteq x_z\cdot \Cont_0(U_z)\)
  because \(\braket{x_z}{x_z}\)
  is strictly positive in \(\Cont_0(U_z)\).
  Hence the rank-\(1\)
  operator \(\Cont_0(U_z) \to X'_z\),
  \(f\mapsto x_z f\),
  has dense range.  Its adjoint also has dense range because
  \(\braket{x_z}{x_z}\)
  is strictly positive.  Polar decomposition now gives the required
  unitary \(X'_z \cong \Cont_0(U_z)\).

  Since~\(X'_z\)
  is a Hilbert bimodule, the left \(\Cont_0(Z)\)-module
  structure on~\(X'_z\)
  must map some ideal in~\(Z\)
  isomorphically onto \(\Comp(X'_z) \cong \Cont_0(U_z)\).
  This isomorphism gives a homeomorphism between~\(U_z\)
  and some open subset of~\(Z\).
  By definition,
  \(X_z \defeq \Cont_0(U_z) \cdot X \cdot \Cont_0(U_z) =
  \Cont_0(U_z)\cdot X'_z\).
  Thus~\(X_z\)
  is the Hilbert \(\Cont_0(U_z)\)-bimodule
  associated to a partial homeomorphism of~\(U_z\),
  as required.  This finishes the proof in case~\(A\)
  contains an essential ideal of Type~I.
\end{proof}

\section{The Connes spectrum and the main results for Fell bundles}
\label{sec:Connes_spectrum}

\begin{definition}
  \label{def:Connes_spectrum}
  Let~\(\mathcal{B}\)
  be a Fell bundle over an Abelian group~\(G\).
  Let~\(\beta\)
  be the dual \(\widehat{G}\)\nb-action
  on~\(\Cst(\mathcal{B})\).
  The \emph{Connes spectrum} of~\(\mathcal{B}\) is
  \[
  \Gamma(\mathcal{B})\defeq \{z\in \widehat{G} \mid
  I \cap \beta_z(I)\neq 0\
  \text{for each non-zero ideal}\ I \text{ in } \Cst(\mathcal{B})\}.
  \]
  The \emph{strong Connes spectrum} of~\(\mathcal{B}\)
  is
  \[
  \tilde\Gamma(\mathcal{B})
  \defeq  \{z\in \widehat{G} \mid
  \beta_z(I)\subseteq  I
  \text{ for any ideal}\ I \text{ in } \Cst(\mathcal{B})\}.
  \]
  If \(X\) is a Hilbert \(A\)\nb-bimodule, let \((X_n)_{n\in \Z}\)
  be the Fell bundle generated by~\(X\) as in Example
  \ref{def:Fell_bundle_for_Hilbert_bimodule} and define
  \(\Gamma(X)\defeq \Gamma((X_n)_{n\in \Z})\) and
  \(\tilde\Gamma(X)\defeq \tilde\Gamma((X_n)_{n\in \Z})\).
\end{definition}

\begin{remark}
  \label{rem:Connes_spectrum_dual_actions}
  Proposition~\ref{pro:Connes_spectrum} says that these definitions
  give the usual notions for \(G\)\nb-actions
  by automorphisms, viewed as Fell bundles as in
  Example~\ref{exa:automorphism_to_Fell}.  In general, Takai Duality
  tells us that the double crossed product
  \(\Cst(\mathcal{B})\rtimes_\beta \widehat{G}\rtimes_\gamma G\)
  for the action
  \(\gamma\colon G\to \Aut(\Cst(\mathcal{B})\rtimes_\beta
  \widehat{G})\)
  that is dual to~\(\beta\)
  is \(\widehat{G}\)\nb-equivariantly
  isomorphic to \(\Cst(\mathcal{B}) \otimes \Comp(L^2(\widehat{G}))\),
  compare the proof of Proposition~\ref{pro:Morita_globalisation}.
  Thus \(\Cst(\mathcal{B})\)
  and \(C\rtimes_\gamma G\)
  are \(\widehat{G}\)\nb-equivariantly
  Morita--Rieffel equivalent.  This induces a
  \(\widehat{G}\)\nb-equivariant
  lattice isomorphism between the ideal lattices of
  \(\Cst(\mathcal{B})\)
  and \(\Cst(\mathcal{B})\rtimes_\beta \widehat{G}\rtimes_\gamma G\).
  The questions whether each ideal in a particular
  \(\widehat{G}\)\nb-\(\Cst\)-algebra
  contains a \(\widehat{G}\)\nb-invariant
  ideal or is \(\widehat{G}\)\nb-invariant
  are invariant under \(\widehat{G}\)\nb-equivariant
  Morita equivalence.  Hence
  \begin{equation}
    \label{eq:Connes_spectra_Fell_vs_Aut}
    \Gamma(\mathcal{B})=\Gamma(\gamma)
    \quad\text{and}\quad
    \tilde\Gamma(\mathcal{B})=\tilde\Gamma(\gamma).
  \end{equation}
  We could use~\eqref{eq:Connes_spectra_Fell_vs_Aut} to define
  the spectra for Fell bundles.
  Schweizer~\cite{Schweizer:Crossed_preprint} defines~\(\Gamma(X)\)
  for an equivalence bimodule~\(X\) in this fashion, as the Connes
  spectrum of the automorphism that generates the \(\Z\)\nb-action
  on \((A\rtimes_X\Z)\rtimes_\beta \T\).
  So our definition generalises Schweizer's.
\end{remark}

As in the case of group actions, the strong Connes spectrum for Fell
bundles is a residual version of the (ordinary) Connes
spectrum.  

\begin{proposition}
  \label{pro:strong_Connes_vs_ordinary}
  Let~\(\mathcal{B}\) be a Fell bundle over a discrete, Abelian
  group~\(G\) with unit fibre~\(A\).  Then
  \[
  \tilde\Gamma(\mathcal{B})
  = \bigcap_{I\in\I^{\mathcal{B}}(A)} \Gamma(\mathcal{B}|_{A/I}),
  \]
  where we use the restrictions of~\(\mathcal{B}\) to the
  \(\mathcal{B}\)\nb-invariant quotients~\(A/I\).
\end{proposition}

\begin{proof}
  Put \(C\defeq \Cst(\mathcal{B})\rtimes_\beta \widehat{G}\) and let
  \(\gamma\colon G\to \Aut(C)\) be the action dual
  to~\(\beta\).  
  Equation~\eqref{eq:residual_Connes_is_strong_Connes} gives
  \begin{equation}
    \label{eq:Kishimoto_residual_strong_Connes}
    \tilde\Gamma(\gamma)
    = \bigcap_{\I\in\I^\gamma(C)} \Gamma(\gamma |_{C/\I}),
  \end{equation}
  where~\(\gamma|_{C/\I}\)
  denotes the \(G\)\nb-action
  on~\(C/\I\)
  induced by~\(\gamma\)
  (we may add \(\I=C\)
  to the intersection because \(C/C=\{0\}\)
  and \(\Gamma(\gamma|_{\{0\}})=\widehat{G}\)).
  Every \(\I\in\I^\gamma(C)\)
  is of the form \(\I=J\rtimes_\beta \widehat{G}\)
  for some \(J\in \I^{\beta}(\Cst(\mathcal{B}))\).
  In turn, this is of the form \(J=\Cst(\mathcal{B}|_I) \)
  for some \(I\in \I^{\mathcal{B}}(A)\)
  by Proposition~\ref{pro:gauge-invariance_vs_separation}.  Equip
  \(\Cst(\mathcal{B})/ J\)
  with the \(\widehat{G}\)\nb-action
  \(\beta|_{\Cst(\mathcal{B})/ J}\)
  induced by~\(\beta\),
  and let~\(\beta|_{A/ I}\)
  be the dual \(\widehat{G}\)\nb-action
  on \(\Cst(\mathcal{B}|_{A/I})\).
  We have \(\widehat{G}\)\nb-equivariant isomorphisms
  \begin{gather*}
    \Cst(\mathcal{B}|_{A/I})
    \cong \Cst(\mathcal{B})/\Cst(\mathcal{B}|_I)
    = \Cst(\mathcal{B})/ J,\\
    (\Cst(\mathcal{B})/ J)\rtimes_{\beta|_{\Cst(\mathcal{B}) / J}} \widehat{G}
    \cong (\Cst(\mathcal{B})\rtimes_{\beta}\widehat{G}) \mathbin{/}
    (J\rtimes_{\beta|_J}\widehat{G})
    = C/\I.
  \end{gather*}
  These induce a \(G\)\nb-equivariant isomorphism
  \(\Cst(\mathcal{B}|_{A/I})\rtimes_{\beta|_{A/ I}}\widehat{G}\cong
  C/\I\).  Therefore,
  \[
  \Gamma(\mathcal{B}|_{A/I}) = \Gamma(\gamma|_{C/\I}),
  \]
  see Remark~\ref{rem:Connes_spectrum_dual_actions}.  Since
  \(\tilde\Gamma(\mathcal{B}) = \tilde\Gamma(\gamma)\),
  \eqref{eq:Kishimoto_residual_strong_Connes} becomes the desired
  formula.
\end{proof}

\begin{corollary}
  If~\(\mathcal{B}\) is a minimal Fell bundle over an Abelian group,
  then \(\tilde\Gamma(\mathcal{B}) = \Gamma(\mathcal{B})\).
\end{corollary}

\begin{proposition}%
  \label{pro:Fell_Connes_spectrum_vs_detect_ideals}
  Let~\(\mathcal{B}\) be a Fell bundle over an Abelian group~\(G\)
  and \(A\defeq B_e\).
  \begin{enumerate}
  \item \label{en:Connes_spectrum_vs_detect_ideals_Connes2}%
    \(A\)~detects ideals in~\(\Cst(\mathcal{B})\) if and only if\/
    \(\Gamma(\mathcal{B}) = \widehat{G}\);
  \item \label{en:strong_Connes_spectrum_vs_separates2}%
    \(A\)~separates ideals in~\(\Cst(\mathcal{B})\) if and only
    if\/ \(\tilde\Gamma(\mathcal{B}) = \widehat{G}\).
  \end{enumerate}
\end{proposition}

\begin{proof}
  Remark~\ref{rem:Connes_spectrum_dual_actions} reduces
  \ref{en:Connes_spectrum_vs_detect_ideals_Connes2} to the case of
  automorphisms, which is
  Theorem~\ref{the:Connes_spectrum_vs_detect_ideals}.  (In fact, our
  definition of~\(\Gamma(\mathcal{B})\) allows for a more elementary
  proof.)  By Corollary~\ref{cor:separation_in_reduced_cross}, \(A\)
  separates ideals in~\(\Cst(\mathcal{B})\) if and only if each
  ideal in~\(\Cst(\mathcal{B})\) is graded.  With our definition of
  \(\tilde\Gamma(\mathcal{B})\), this is clearly equivalent
  \(\tilde\Gamma(\mathcal{B}) = \widehat{G}\).
\end{proof}

\begin{lemma}
  Let~\(\mathcal{B}\) be a Fell bundle over a cyclic group~\(G\)
  generated by an element \(g\in G\).  Then
  \(\Gamma(\mathcal{B})=\Gamma(B_{g})\) and
  \(\tilde\Gamma(\mathcal{B})=\tilde\Gamma(B_{g})\).
\end{lemma}

\begin{proof}
  This follows from~\eqref{eq:Connes_spectra_Fell_vs_Aut} and
  Lemma~\ref{lem:explanation_of_spectra_for_automorphisms} applied to
  the Morita globalisation~\(\gamma\) of~\(\mathcal{B}\).
\end{proof}

We generalise and extend
Lemma~\ref{lem:separable_topological_equivalences}.

\begin{proposition}
  \label{pro:separable_countable_equivalences}
  Let \(\mathcal{B}=(B_g)_{g\in G}\)
  be a Fell bundle over a discrete group~\(G\)
  such that the unit fibre \(A\defeq B_e\)
  contains an essential
  ideal that is separable or whose spectrum is Hausdorff.  The following
  are equivalent:
  \begin{enumerate}[wide,label=\textup{(\ref*{pro:separable_countable_equivalences}.\arabic*)}]
  \item \label{en:topolog_free_Fell}%
    \(\mathcal{B}\) is topologically free;
  \item \label{en:topolog_pointwise_free_Fell}%
    \(\mathcal{B}\) is pointwise topologically nontrivial.
  \end{enumerate}
  If, in addition, \(G\) is countable, the above conditions are
  equivalent to
  \begin{enumerate}[wide,label=\textup{(\ref*{pro:separable_countable_equivalences}.\arabic*)}, resume]
  \item \label{en:densly_free_Fell}%
    \(\widehat{\mathcal{B}}\) is free on a dense subset
    of~\(\widehat{A}\).
  \end{enumerate}
\end{proposition}

\begin{proof}
  As in the proof of
  Lemma~\ref{lem:separable_topological_equivalences}, it suffices to
  assume that~\(G\)
  is countable and prove that~\ref{en:topolog_pointwise_free_Fell}
  implies~\ref{en:densly_free_Fell}.  So
  assume~\ref{en:topolog_pointwise_free_Fell}.

  Let \(K\idealin A\) be an essential ideal.  Let \(K_g\defeq
  KB_gK\) for each \(g\in G\).  Then \((K_g)_{g\in G}\) is naturally
  a Fell bundle with \(K_e=K\).  The dual partial
  homeomorphisms~\(\widehat{K_g}\) for \(g\in G\) are the
  restrictions of the partial homeomorphisms~\(\widehat{B_g}\) to
  the open dense set \(\widehat{K}\).  More precisely,
  \(\widehat{K_g}\) is the restriction of \(\widehat{B_g}\colon
  \widehat{D}_{g^{-1}}\congto \widehat{D_g}\) to
  \(\widehat{D_{g^{-1}}} \cap \widehat{K}\cap
  \widehat{B_{g^{-1}}}(\widehat{D_g}\cap \widehat{K})\).  Thus the
  Fell bundle \((K_g)_{g\in G}\) is pointwise topologically
  nontrivial, and \ref{en:densly_free_Fell} holds if and only if
  \((\widehat{K_g})_{g\in G}\) is free on a dense subset
  of~\(\widehat{K}\).  This replaces~\(A\) by an essential
  ideal~\(K\).  So it suffices to prove the proposition if~\(A\)
  itself is separable or has Hausdorff spectrum.
	
  Suppose first that~\(A\) is separable. 
  Let \(\gamma\colon G\to\Aut(C)\)
  be the Morita globalisation of~\(\widehat{\mathcal{B}}\) described in
  Proposition~\ref{pro:Morita_globalisation}.
  The \(\Cst\)\nb-algebra~\(C\)
  is separable by
  Proposition~\ref{pro:Morita_globalisation}.\ref{en:Morita_globalisation_5},
  and the family of Hilbert bimodules~\((B_{h g h^{-1}})_{h\in G}\)
  covers~\(C_{\gamma_g}\)
  up to Morita equivalence.  Hence~\(\gamma_g\)
  is topologically non-trivial for all \(g\in G\setminus\{e\}\)
  by
  Proposition~\ref{pro:dense_Morita_covering}.\ref{en:dense_Morita_covering_top}.
  Lemma~\ref{lem:separable_topological_equivalences} gives a dense
  subset of~\(\widehat{C}\)
  on which~\(\widehat{\gamma}\)
  is free.  Since \(\gamma\colon G\to \Aut(C)\)
  is a Morita globalisation of~\(\mathcal{B}\),
  the partial action~\(\widehat{\mathcal{B}}\)
  may be identified with a restriction of~\(\widehat{\gamma}\)
  to an open subset.  Hence there is a dense subset of~\(\widehat{A}\)
  on which~\(\widehat{\mathcal{B}}\) is free.

  Suppose now that the spectrum~\(\widehat{A}\) of~\(A\) is
  Hausdorff.  The sets \(F_g \defeq \{[\pi]\in \widehat{A} \mid
  \widehat{B_g}([\pi])=[\pi]\}\) for \(g\in G\) are closed
  in~\(\widehat{A}\), see the proof of
  \cite{Exel-Laca-Quigg:Partial_dynamical}*{Lemma~2.2}.  Since they
  have empty interiors, the union \(\bigcup_{g\in G\setminus\{e\}}
  F_g\) has empty interior in~\(\widehat{A}\) by the Baire Category
  Theorem.
\end{proof}

\begin{theorem}
  \label{thm:separable_type_I_equivalences}
  Let \(\mathcal{B}=(B_g)_{g\in G}\)
  be a Fell bundle whose unit fibre \(A\defeq B_e\)
  contains an essential ideal that is separable or of Type~I.
  Then~\(\mathcal{B}\)
  is aperiodic if and only if~\(\mathcal{B}\) is topologically free.
\end{theorem}

\begin{proof}
  If~\(A\)
  has an essential ideal of Type~I, then it also has an essential
  ideal with Hausdorff spectrum by
  \cite{Pedersen:Cstar_automorphisms}*{Theorem~6.2.11}.  Thus the
  assertion follows from
  Proposition~\ref{pro:separable_countable_equivalences} and
  Theorem~\ref{the:Hilbi_Kish_vs_weakly_inner}.
\end{proof}

\begin{theorem}
  \label{the:Equivalence_for_finite2}
  Let~\(\mathcal{B}\) be a Fell bundle over a discrete group~\(G\).
  The following are equivalent:
  \begin{enumerate}[wide,label=\textup{(\ref*{the:Equivalence_for_finite2}.\arabic*)}]
  \item \label{en:Equivalence_finite_Kishimoto}%
    \(\mathcal{B}\)~is aperiodic;
  \item \label{en:Equivalence_finite_spectra1}%
    \(\Gamma(B_g)=\T_{\ord(g)}\) for all \(g\in G\);
  \item \label{en:Equivalence_finite_spectra2}%
    \(\Gamma(B_g)\neq\{1\}\) for all \(g\in G\setminus\{e\}\).
  \end{enumerate}
  These equivalent conditions imply
  \begin{enumerate}[wide,label=\textup{(\ref*{the:Equivalence_for_finite2}.\arabic*)},resume]
  \item \label{en:Equivalence_finite_outer}%
    \(\mathcal{B}\) is pointwise purely outer.
  \end{enumerate}
  If~\(G\)
  is finite or~\(A\)
  contains an essential
  ideal that is simple or of Type~I, then
  \ref{en:Equivalence_finite_Kishimoto}--\ref{en:Equivalence_finite_outer}
  are equivalent.
\end{theorem}

\begin{proof}
  Propositions \ref{pro:Morita_globalisation}
  and~\ref{pro:dense_Morita_covering}
  and~\eqref{eq:Connes_spectra_Fell_vs_Aut} reduce the theorem to the
  case of automorphisms, which is mostly done in
  Theorem~\ref{the:Equivalence_for_finite}.  If~\(A\)
  contains an  essential ideal which is simple or of Type~I, then
  \ref{en:Equivalence_finite_Kishimoto}\(\Leftrightarrow\)%
  \ref{en:Equivalence_finite_outer}  by
  Theorem~\ref{the:Hilbi_Kish_vs_weakly_inner}.
\end{proof}

\begin{theorem}
  \label{the:residual_for_finite2}
  Let~\(\mathcal{B}\) be a Fell bundle over a discrete group~\(G\).
  Consider the following conditions:
  \begin{enumerate}[wide,label=\textup{(\ref*{the:residual_for_finite2}.\arabic*)}]
  \item \label{en:residual_finite_spectra_full2}%
    \(\tilde\Gamma(B_g)=\T_{\ord(g)}\) for all \(g\in G\);
  \item \label{en:residual_finite_spectra_nontrivial2}%
    \(\tilde\Gamma(B_g)\neq\{1\}\) for all \(g\in
    G\setminus\{e\}\);
  \item \label{en:residual_finite_aperiodic2}%
    \(\mathcal{B}\) is residually aperiodic;
  \item \label{en:residual_finite_purely_outer2}%
    \(\mathcal{B}\) is residually pointwise purely outer;
  \item \label{en:strongly_outer2}%
    for any \(g\in G\setminus\{e\}\) and any two
    \(B_g\)\nb-invariant ideals \(I\subsetneq J\subseteq A\), the
    restriction \(B_g|_{J/I}\) is outer.
  \end{enumerate}
  Then \ref{en:residual_finite_spectra_full2}\(\Leftrightarrow\)%
  \ref{en:residual_finite_spectra_nontrivial2}\(\Rightarrow\)%
  \ref{en:residual_finite_aperiodic2}\(\Rightarrow\)%
  \ref{en:residual_finite_purely_outer2}\(\Leftarrow\)%
  \ref{en:strongly_outer2}.  If~\(G\) is finite, all
  conditions
  \ref{en:residual_finite_spectra_full2}--\ref{en:strongly_outer2}
  are equivalent.
\end{theorem}

\begin{proof}
  Most claims are just residual versions of assertions in
  Theorem~\ref{the:Equivalence_for_finite2}; by
  Proposition~\ref{pro:strong_Connes_vs_ordinary}, the strong Connes
  spectrum is the residual version of the Connes spectrum.  This gives
  the implications
  \ref{en:residual_finite_spectra_full2}\(\Leftrightarrow\)%
  \ref{en:residual_finite_spectra_nontrivial2}\(\Rightarrow\)%
  \ref{en:residual_finite_aperiodic2}\(\Rightarrow\)%
  \ref{en:residual_finite_purely_outer2} for all~\(G\)
  and the converse implication
  \ref{en:residual_finite_purely_outer2}\(\Rightarrow\)%
  \ref{en:residual_finite_aperiodic2} for finite~\(G\).
  Beware that \ref{en:residual_finite_spectra_nontrivial2} is \emph{a
    priori} stronger than~\ref{en:residual_finite_aperiodic2} because
  the former involves all ideals invariant under the single
  automorphism~\(\alpha_g\),
  whereas the latter involves only those ideals invariant under the
  whole action~\(\alpha\).
  Condition~\ref{en:strongly_outer2}
  implies~\ref{en:residual_finite_purely_outer2} because it
  requires~\(B_g|_{J/I}\)
  to be outer for more subquotients.  For finite~\(G\)
  and an action by automorphisms, the converse implication
  \ref{en:residual_finite_purely_outer2}\(\Rightarrow\)%
  \ref{en:strongly_outer2} is asserted in
  \cite{Pasnicu-Phillips:Spectrally_free}*{Lemma~1.15}.  Propositions
  \ref{pro:Morita_globalisation} and~\ref{pro:dense_Morita_covering}
  allow to generalise this implication from actions by automorphisms
  to Fell bundles. Thus it suffices to show that \ref{en:strongly_outer2} implies \ref{en:residual_finite_spectra_full2} when \(G\) is finite.

  To this end, note that \ref{en:strongly_outer2} is equivalent to
  the condition: for any \(g\in G\setminus\{e\}\) and \(I\in
  \I^{B_g}(A)\), the Hilbert \(A/I\)-bimodule \(B_g/B_g I\) is
  purely outer.  For finite~\(G\),
  Theorem~\ref{the:residual_for_finite2} shows that
  \ref{en:strongly_outer2} is equivalent to the condition:
  \(\Gamma(B_g/B_g I) = \T_{\ord(g)}\) for all \(g\in G\) and \(I\in
  \I^{B_g}(A)\).  The latter condition is equivalent to
  \ref{en:residual_finite_spectra_full2} by
  Proposition~\ref{pro:strong_Connes_vs_ordinary}.
\end{proof}

\begin{remark}
  Phillips calls a group action by automorphisms
  \(\alpha\colon G\to \Aut(A)\)
  \emph{strongly pointwise outer} if it
  satisfies~\ref{en:strongly_outer2}, see
  \cite{Phillips:Freeness_actions_finite_groups}*{Definition~4.11} or
  \cite{Pasnicu-Phillips:Spectrally_free}*{Definition~1.1}. In
  particular, for actions of finite groups by automorphisms, the
  equivalence between \ref{en:residual_finite_spectra_nontrivial2},
  \ref{en:residual_finite_aperiodic2} and~\ref{en:strongly_outer2} is
  \cite{Pasnicu-Phillips:Spectrally_free}*{Theorem~1.16}.
\end{remark}

\begin{theorem}
  \label{the:Fell_uniqueness}
  Let \(G=\Z\)
  or \(G=\Z/p\)
  for a square-free number \(p>0\).
  Let \(\mathcal{B}=(B_g)_{g\in G}\)
  be a Fell bundle over~\(G\).
  Suppose that \(A\defeq B_0\)
  contains an essential,
  ideal that is separable or of Type~I.  The following are
  equivalent:
  \begin{enumerate}[wide,label=\textup{(\ref*{the:Fell_uniqueness}.\arabic*)}]
  \item \label{en:Fell_uniqueness_Connes}%
    \(\Gamma(\mathcal{B}) = \widehat{G}\);
  \item \label{en:Fell_uniqueness_detect}%
    \(A\)~detects ideals in~\(\Cst(\mathcal{B})\);
  \item \label{en:Fell_uniqueness_gauge}%
    each non-zero ideal in~\(\Cst(\mathcal{B})\)
    contains a non-zero gauge-invariant ideal;
  \item \label{en:Fell_uniqueness_support}%
    \(A\)~supports~\(\Cst(\mathcal{B})\);
  \item \label{en:Fell_uniqueness_aperiodic}%
    \(\mathcal{B}\)~is aperiodic;
  \item \label{en:Fell_uniqueness_spectra2}%
    \(\Gamma(B_g)\neq\{1\}\) for all \(g\in G\setminus\{0\}\);
  \item \label{en:uniqueness_topolol_pointwise_free}%
    \(\mathcal{B}\)~is pointwise topologically nontrivial;
  \item \label{en:Fell_uniqueness_topological}%
    \(\mathcal{B}\)~is topologically free;
  \item \label{en:uniqueness_densly_free}%
    \(\widehat{\mathcal{B}}\)~is
    free on a dense subset of~\(\widehat{A}\);
  \item \label{en:Fell_uniqueness_partly_weakly_inner}%
    \(\mathcal{B}\)~is pointwise purely universally weakly outer.
  \end{enumerate}
  If~\(G\)
  is finite or if~\(A\)
  contains an essential ideal that is simple or of Type~I, then
  these conditions are also equivalent to
  \begin{enumerate}[wide,label=\textup{(\ref*{the:Fell_uniqueness}.\arabic*)},resume]
  \item \label{en:Fell_uniqueness_partly_inner}%
    \(\mathcal{B}\)~is pointwise purely outer.
  \end{enumerate}
\end{theorem}

\begin{proof}
  Conditions
  \ref{en:Fell_uniqueness_Connes}--\ref{en:Fell_uniqueness_gauge}
  are equivalent by
  Proposition~\ref{pro:Fell_Connes_spectrum_vs_detect_ideals} and
  Corollary~\ref{cor:detection_in_reduced_cross}.  We have
  \ref{en:Fell_uniqueness_aperiodic}\(\Leftrightarrow\)%
  \ref{en:Fell_uniqueness_spectra2}\(\Rightarrow\)%
  \ref{en:Fell_uniqueness_support}\(\Rightarrow\)%
  \ref{en:Fell_uniqueness_detect} by
  Theorem~\ref{the:Equivalence_for_finite2},
  Proposition~\ref{the:general_uniqueness} and
  Lemma~\ref{lem:support_implies_detect}.  Conditions
  \ref{en:uniqueness_topolol_pointwise_free}--\ref{en:uniqueness_densly_free}
  are equivalent by
  Proposition~\ref{pro:separable_countable_equivalences}.  Thus
  \ref{en:Fell_uniqueness_aperiodic}--\ref{en:Fell_uniqueness_partly_weakly_inner}
  are equivalent by Theorem~\ref{the:Hilbi_Kish_vs_weakly_inner};
  and if~\(A\) contains an essential ideal of Type~I, then they are
  also equivalent to~\ref{en:Fell_uniqueness_partly_inner}.
  If~\(G\) is finite, then
  \ref{en:Fell_uniqueness_aperiodic}--\ref{en:Fell_uniqueness_partly_weakly_inner}
  are equivalent to~\ref{en:Fell_uniqueness_partly_inner} by
  Theorem~\ref{the:residual_for_finite2}.  Thus it only remains to
  show that~\ref{en:Fell_uniqueness_detect}
  implies~\ref{en:uniqueness_topolol_pointwise_free}. Thus
  assume~\ref{en:Fell_uniqueness_detect}.
	
  Let \(K\idealin A\) be an essential ideal.  Put \(K_g\defeq
  KB_gK\) for each \(g\in G\).  Then \(\mathcal{K}\defeq (K_g)_{g\in
    G}\) is a hereditary sub-Fell bundle of~\(\mathcal{B}\), see
  \cite{Exel:Partial_dynamical}*{Definition 21.10}.
  Hence~\(\Cst(\mathcal{K})\) is a hereditary subalgebra
  of~\(\Cst(\mathcal{B})\), see
  \cite{Exel:Partial_dynamical}*{Theorem 21.12}.  If
  \(J\triangleleft \Cst(\mathcal{K})\) then
  \(\Cst(\mathcal{B})J\Cst(\mathcal{B}) \cap A\neq 0\)
  by~\ref{en:Fell_uniqueness_detect}.  Since \(\Cst(\mathcal{K})\)
  is hereditary in~\(\Cst(\mathcal{B})\) and~\(K\) is essential
  in~\(A\), we get \(J\cap K =
  \Cst(\mathcal{B})J\Cst(\mathcal{B})\cap K \neq 0\).  Thus
  \(K_e=K\) detects ideals in~\(\Cst(\mathcal{K})\).  Furthermore,
  by Corollary \ref{cor:restriction_to_essential_ideals},
  \(\mathcal{B}\)~is pointwise topologically nontrivial if and only
  if~\(\mathcal{K}\) is.  Therefore, to prove the remaining part of
  the assertion, we may assume that~\(A\) itself is separable or
  of Type~I.
	
  Now, let \(\gamma\colon G\to \Aut(C)\) be the Morita globalisation
  of~\(\mathcal{B}\) described in
  Proposition~\ref{pro:Morita_globalisation}.  It preserves both
  \ref{en:Fell_uniqueness_detect}
  and~\ref{en:uniqueness_topolol_pointwise_free}.  If~\(A\) is
  separable, then~\(C\) is separable by
  Proposition~\ref{pro:Morita_globalisation}.\ref{en:Morita_globalisation_5}.
  Hence \ref{en:Fell_uniqueness_detect}\(\Rightarrow\)%
  \ref{en:uniqueness_topolol_pointwise_free} follows from
  Theorem~\ref{the:automorphism_detect_vs_Kishimoto}.  If~\(A\) is
  of Type~I, then so is~\(C\) by
  Proposition~\ref{pro:Morita_globalisation}.\ref{en:Morita_globalisation_4}.
  Thus \ref{en:Fell_uniqueness_detect} implies
  \ref{en:uniqueness_topolol_pointwise_free} by the proof of
  \cite{Olesen-Pedersen:Applications_Connes_2}*{Theorem~4.6}.  It is
  remarked in
  \cite{Olesen-Pedersen:Applications_Connes_2}*{Remark~4.7} that the
  implication we care about does not need separability.  And the
  proof works both for~\(\Z\) and for~\(\Z/p\) with
  square-free~\(p\), compare the proof of
  Theorem~\ref{the:automorphism_simple}.
\end{proof}

\begin{theorem}
  \label{the:automatic_gauge-invariance}
  Let \(G=\Z\)
  or \(G=\Z/p\)
  with square-free~\(p\).
  Let~\(\mathcal{B}\)
  be a Fell bundle over~\(G\).
  Suppose that \(A\defeq B_0\)
  is separable or of Type~I.  The
  following are equivalent:
  \begin{enumerate}[wide,label=\textup{(\ref*{the:automatic_gauge-invariance}.\arabic*)}]
  \item \label{en:separation_of_ideals0}%
    the strong Connes spectrum~\(\tilde\Gamma(\mathcal{B})\)
    is~\(\widehat{G}\);
  \item \label{en:separation_of_ideals}%
    \(A\) separates ideals of~\(\Cst(\mathcal{B})\);
  \item \label{en:automatic_gauge-invariance_Fourier}%
    each ideal in~\(\Cst(\mathcal{B})\)
    is of the form \(\Cst(\mathcal{B}|_I)\)
    for some \(I\in \I^{\widehat{G} }(A)\);
  \item \label{en:automatic_gauge-invariance12}%
    \(A\)~residually supports~\(\Cst(\mathcal{B})\);
  \item \label{en:automatic_gauge-invariance13}%
    \(A^+\)~is a filling family for~\(\Cst(\mathcal{B})\);
  \item \label{en:automatic_gauge-invariance2}%
    \(\mathcal{B}\)~is residually aperiodic;
  \item \label{en:automatic_gauge-invariance4}%
    \(\mathcal{B}\)~is residually pointwise topologically nontrivial;
  \item \label{en:automatic_gauge-invariance3}%
    \(\mathcal{B}\)~is residually topologically free;
  \item \label{en:automatic_gauge-invariance5}%
    \(\mathcal{B}\)~is
    residually pointwise purely universally weakly outer.
  \end{enumerate}
  Moreover, if~\(G=\Z/p\), these conditions are further equivalent
  to
  \begin{enumerate}[wide,label=\textup{(\ref*{the:automatic_gauge-invariance}.\arabic*)},resume]
  \item \label{en:residually_nonperiodic}%
    \(\mathcal{B}\)~is residually pointwise purely outer;
  \item \label{en:automatic_gauge_spectra1}%
    \(\widetilde\Gamma(B_g)\neq\{1\}\)
    for all \(g\in G\setminus\{e\}\);
    \item \label{en:automatic_gauge_spectra2}%
      \(\widetilde\Gamma(B_g)=\T_{\ord(g)}\) for all \(g\in G\).
  \end{enumerate}
  If \(G=\Z\) and~\(A\) is of Type~I, then conditions
  \ref{en:separation_of_ideals0}--\ref{en:residually_nonperiodic}
  are equivalent.
\end{theorem}

\begin{proof}
  Theorem~\ref{the:general_separation_Fell_bundles} and
  Proposition~\ref{pro:Fell_Connes_spectrum_vs_detect_ideals} give
  the implications
  \begin{center}
    \ref{en:separation_of_ideals0}\(\Leftrightarrow\)%
    \ref{en:separation_of_ideals}\(\Leftrightarrow\)%
    \ref{en:automatic_gauge-invariance_Fourier}\(\Leftarrow\)%
    \ref{en:automatic_gauge-invariance12}\(\Leftarrow\)%
    \ref{en:automatic_gauge-invariance13}\(\Leftarrow\)%
    \ref{en:automatic_gauge-invariance2}.
  \end{center}
  Since separation of ideals
  is the residual version of detection of ideals (see
  Remark~\ref{rem:separates_is_residual_detects}) and being separable or of Type I passes to quotients,
  Theorem~\ref{the:Fell_uniqueness} implies that all conditions
  \ref{en:separation_of_ideals0}--\ref{en:automatic_gauge-invariance5}
  are equivalent.  Theorem~\ref{the:Fell_uniqueness} also shows that
  these conditions are equivalent to~\ref{en:residually_nonperiodic}
  if~\(G\)
  is finite or~\(A\)
  is simple or of Type~I.  If~\(G\)
  is finite, then independently of~\(A\)
  these conditions are equivalent
  to~\ref{en:automatic_gauge_spectra1}%
  \(\Leftrightarrow\)\ref{en:automatic_gauge_spectra2}
  by Theorem~\ref{the:residual_for_finite2}.
\end{proof}

\begin{remark}
  The above theorem proves the conjecture stated in
  \cite{Kwasniewski-Szymanski:Pure_infinite}*{Remark 7.4}.  Namely,
  if~\(E\)
  is a countable directed graph, then the graph
  \(\Cst\)\nb-algebra~\(\Cst(E)\)
  can be viewed as a crossed product~\(A\rtimes_X\Z\),
  where~\(A\)
  is the core subalgebra of~\(\Cst(E)\)
  and~\(X\)
  is the first spectral subspace of~\(\Cst(E)\)
  with respect to the associated gauge action.  Since~\(A\)
  is separable, all conditions (i)--(iv) in
  \cite{Kwasniewski-Szymanski:Pure_infinite}*{Proposition~7.3} are
  equivalent without finiteness assumptions on~\(E\).
\end{remark}

For minimal actions of the above groups, several conditions are
equivalent without any assumptions on~\(A\):

\begin{theorem}
  \label{the:Fell_simple}
  Let \(G=\Z\)
  or~\(\Z/p\)
  for a square-free number~\(p\).
  Assume that~\(\mathcal{B}\)
  is a minimal Fell bundle over~\(G\).
  Then the following are equivalent:
  \begin{enumerate}[wide,label=\textup{(\ref*{the:Fell_simple}.\arabic*)}]
  \item \label{en:Fell_simple_simple}%
    \(\Cst(\mathcal{B})\) is simple;
  \item \label{en:Fell_simple_Connes}%
    \(\Gamma(\mathcal{B})=\widehat{G}\);
  \item \label{en:Fell_simple_centre}%
    \(\Mult(\Cst(\mathcal{B}))\) has trivial centre;
  \item \label{en:Fell_simple_outer}%
    \(\mathcal{B}\) is pointwise outer.
   \end{enumerate}
   If \(G=\Z/p\)
   or \(A\defeq B_0\)
   is simple, then these conditions are further equivalent to
   \begin{enumerate}[wide,label=\textup{(\ref*{the:Fell_simple}.\arabic*)},resume]%
   \item \label{en:Fell_minimal_support}%
     \(A\)~supports \(\Cst(\mathcal{B})\);
   \item \label{en:Fell_minimal_filling}%
     \(A^+\)~is a filling family for \(\Cst(\mathcal{B})\);
   \item \label{en:Fell_minimal_aperiodic}%
    \(\mathcal{B}\)~is aperiodic.
  \end{enumerate}
\end{theorem}

\begin{proof}
  Conditions \ref{en:Fell_simple_simple}
  and~\ref{en:Fell_simple_Connes} are equivalent by
  Proposition~\ref{pro:Fell_Connes_spectrum_vs_detect_ideals}
  because~\(\mathcal{B}\)
  is minimal.  The implications
  \ref{en:Fell_minimal_aperiodic}\(\Rightarrow\)%
  \ref{en:Fell_minimal_filling}\(\Rightarrow\)%
  \ref{en:Fell_minimal_support}\(\Rightarrow\)%
  \ref{en:Fell_simple_simple} follow from
  Theorem~\ref{the:general_separation_Fell_bundles}.  If either
  \(G=\Z/p\)
  or~\(A\)
  is simple, then \ref{en:Fell_minimal_aperiodic}
  and~\ref{en:Fell_simple_outer} are equivalent by
  Theorem~\ref{the:Equivalence_for_finite2}.

  It remains to show the equivalence of \ref{en:Fell_simple_simple},
  \ref{en:Fell_simple_centre} and~\ref{en:Fell_simple_outer}.  For
  actions of~\(G\)
  by automorphisms, this is contained in
  Theorem~\ref{the:automorphism_simple}.
  Proposition~\ref{pro:Morita_globalisation} gives a Morita
  globalisation \(\gamma\colon G \to\Aut(C)\)
  of~\(\mathcal{B}\)
  such that \(\Cst(\mathcal{B})\)
  and~\(C\rtimes_\gamma G\)
  are \(\widehat{G}\)\nb-equivariantly
  Morita equivalent.  Thus~\(\Cst(\mathcal{B})\)
  is simple if and only if~\(C\rtimes_\gamma G\)
  is; and both have isomorphic centre, say, by the Dauns--Hofmann
  Theorem.  Since
  \(\I^\gamma(C) \cong \I^{\widehat{G}}(C\rtimes_\gamma G) \cong
  \I^{\widehat{G}}(\Cst(\mathcal{B})) \cong \I^{\mathcal{B}}(A)\),
  \(\mathcal{B}\)
  is minimal if and only if~\(\gamma\)
  is.  In the minimal case, being outer and purely outer are
  equivalent.  Hence Proposition~\ref{pro:dense_Morita_covering}.%
  \ref{en:dense_Morita_covering_outer} shows that~\(\mathcal{B}\)
  is pointwise outer if and only if~\(\gamma\) is.
\end{proof}

\begin{corollary}
  If~\(X\)
  is a Hilbert \(A\)\nb-bimodule
  which is not an equivalence bimodule, then the crossed product
  \(A\rtimes_X \Z\) is simple if and only if~\(X\) is minimal.
\end{corollary}

\begin{proof}
  For any \(n>0\),
  \(\braket{X^{\otimes_A n}}{ X^{\otimes_A n}}_A \subseteq
  \braket{X}{X}_A\)
  and
  \({}_A\braket{X^{\otimes_A n}}{X^{\otimes_A n}}\subseteq
  {}_A\braket{X}{X}\).
  Since~\(X\)
  is not an equivalence bimodule, this implies
  \(X^{\otimes_A n}\not\cong A\)
  for all \(n>0\).
  So the Fell bundle~\((X_n)_{n\in\N}\) is pointwise outer.
\end{proof}

\end{document}